\documentclass[10pt,a4paper,reqno]{article}
\usepackage{float}
\usepackage{graphicx}
\usepackage{titlesec}
\usepackage{lmodern}
\usepackage{amsmath,amssymb,amsthm,mathtools,mathrsfs,color}
\usepackage{authblk,enumitem}
\titleformat{\section}{\center\normalfont\fontsize{13.5}{10}\bfseries}{\thesection}{0.5em}{}
\titleformat{\subsection}{\normalfont\fontsize{12}{17}\bfseries}{\thesubsection}{0.5em}{}
\allowdisplaybreaks
\usepackage[margin=2cm,includehead]{geometry}
\usepackage{xcolor}
\usepackage{fancyhdr}
\usepackage{tikz}
\usetikzlibrary{arrows.meta,calc,positioning}
\usepackage[linktocpage=true,colorlinks=true,citecolor=blue,linkcolor=blue,urlcolor=blue, pagebackref]{hyperref}
\usepackage[nameinlink,capitalise,noabbrev]{cleveref}
\allowdisplaybreaks

\usetikzlibrary{calc}

\usepackage{comment}
\graphicspath{}

\usepackage{cite}
\usepackage{latexsym}
\usepackage{amscd}
\usepackage{url}
\usepackage[alphabetic,backrefs]{amsrefs}
\usepackage{lipsum} 
\usepackage{tikzpagenodes}
\usepackage{faktor}
\usetikzlibrary{tikzmark}
\usepackage{slashed}

\newtheorem{thm}{Theorem}[section]

\newtheorem{lemma}[thm]{Lemma}
\newtheorem{prop}[thm]{Proposition}
\newtheorem{conj}[thm]{Conjecture}

\theoremstyle{definition}
\newtheorem{defn}[thm]{Definition}
\newtheorem{rem}[thm]{Remark}

\numberwithin{equation}{section}
\def\4{\Sigma_{14}}

\def\R{\mathbb{R}}
\def\SU{\mathrm{SU}}
\def\U{\mathrm{U}}

\def\Z{\mathbb{Z}}

\def\Spin{\operatorname{Spin}}
\def\Id{\textup{Id}}

\def\ns{\kern-5pt}
\def\ba{\begin{array}}
\def\ea{\end{array}}
\def\be{\begin{equation}}

\def\ee{\end{equation}}
\def\d{\mathrm{d}}

\def\y{\\[3pt]}
\def\yy{\\[7pt]}

\def\G2{\mathrm{G}_2}
\def\g2{\varphi}

\def\n{\mathfrak n}

\def\fg{\mathfrak{g}}

\def\SO{\mathrm{SO}}
\def\Sp{\mathrm{Sp}}

\def\lieg2{\mathfrak{g}_2}
\def\Sp{\mathrm{Sp}}

\DeclareMathOperator\vol{vol}

\DeclareFontFamily{U}{MnSymbolC}{}
\DeclareSymbolFont{MnSyC}{U}{MnSymbolC}{m}{n}
\DeclareFontShape{U}{MnSymbolC}{m}{n}{
    <-6>  MnSymbolC5
   <6-7>  MnSymbolC6
   <7-8>  MnSymbolC7
   <8-9>  MnSymbolC8
   <9-10> MnSymbolC9
  <10-12> MnSymbolC10
  <12->   MnSymbolC12}{}
\DeclareMathSymbol{\intprod}{\mathbin}{MnSyC}{'270}

\DeclareUnicodeCharacter{202F}{\,}
\begin{document}
\parskip2pt

\title{New Inhomogeneous Nearly $\text{G}_2$-metric on the Berger Space and a Sine Cone Desingularisation}

\author{Simon Salamon \hspace{2cm} Ragini Singhal}
\date{}
\maketitle

\begin{abstract}
Nearly $\G2$ manifolds are positive Einstein seven-dimensional manifolds whose associated cone metric has special holonomy equal to, or contained in, Spin(7). We study nearly parallel $\G2$-structures invariant under a cohomogeneity-one action of $\SO(4)$, and construct the first inhomogeneous nearly parallel $\G2$-structure on the Berger space SO(5)/SO(3). The cone over this metric has full holonomy $\Spin(7)$. The existence of the solution is established by a rigorous computer-assisted shooting argument.

We also establish a local desingularisation of a finite quotient of the sine cone over $S^3\times S^3$ by gluing in the asymptotically conical $C_7$ bubble constructed by Foscolo--Haskins--Nordstr\"om. We show that as the singular orbit of a family of nearly parallel $\G2$ manifolds collapses, away from the singular orbit the desingularising family converges  to the sine cone quotient, while the natural blow-up converges to the AC torsion-free $\G2$-metric of $C_7$.

\end{abstract}

\tableofcontents{}

\vfill

\newpage

\section{Introduction}
Nearly parallel $\mathrm G_2$-geometry lies at the intersection of exceptional holonomy, spin geometry and the theory of Einstein metrics.  A nearly parallel $\mathrm G_2$-structure on a seven-manifold $M$ is defined by a positive three-form $\varphi$ satisfying
$$
d\varphi=\lambda*_{\varphi}\varphi
$$
for some non-zero constant $\lambda\in\R$.  The form $\varphi$ determines a Riemannian metric $g_\varphi$, which is Einstein with positive scalar curvature, and equivalently determines a real Killing spinor. Moreover, the Riemannian cone over a smooth compact nearly $\G2$ manifold has holonomy contained in $\Spin(7)$.  When the cone has full holonomy $\Spin(7)$, the nearly parallel structure is called proper.  Thus proper nearly parallel $\mathrm G_2$-manifolds provide both a distinguished class of positive Einstein seven-manifolds and the links of holonomy $\Spin(7)$-cones. 

The authors of \cite{book} showed that, apart from the round sphere, complete simply connected nearly parallel $\mathrm G_2$-manifolds can be classified as proper, Sasaki--Einstein, and 3-Sasakian, according as the dimension of the space of Killing spinors is 1, 2, and 3, and the holonomy of the metric cone is $\Spin(7)$, $\SU(4)$, or $\Sp(2)$, respectively. This description explains the appearance of nearly parallel $\mathrm G_2$-geometry in theoretical physics. These spaces, traditionally called manifolds of weak $\mathrm G_2$-holonomy \cite{Gray1971}, arise as internal spaces in supersymmetric Freund--Rubin compactifications of eleven-dimensional supergravity, with the number of Killing spinors measuring the supersymmetry preserved in four dimensions \cite{BDS02}.

Simply connected homogeneous nearly parallel $\G2$-manifolds were classified in  \cite{friedkath}, and up to homothety they are: the round and squashed $7$-spheres, two homogeneous Einstein metrics on $SU(3)/\U(1)$, the Berger space $SO(5)/SO(3)$, and the homogeneous Sasaki–Einstein spaces $V_{5,2}$, $M^{1,1,1}$, and $Q^{1,1,1}$. Infinite families of inhomogeneous Sasaki--Einstein  and $3$-Sasakian $7$-manifolds are known \cites{GalickiLawson1988,Galicki2,Galicki1,BGMR98,GauntlettMartelliSparksWaldram2004,BoyerGalickiKollar2005,CveticLuPagePope2005PRL,BoyerGalickiKollar2006,Sparks2011,TomasielloZaffaroni2011}. Moreover, every $3$-Sasakian $7$-manifold admits a canonical squashing which produces a second Einstein metric carrying a proper nearly parallel $\mathrm G_2$-structure \cites{GalickiSalamon1996,friedkath}.  However, explicit compact proper metrics that are neither homogeneous nor obtained from the $3$-Sasakian squashing construction remain scarce. It is in this direction that the present paper contributes.

Cohomogeneity one provides a natural setting in which to search for such examples. Historically, this approach has been very effective in the construction of inhomogeneous positive Einstein metrics. B\"ohm's landmark work \cite{Boh98} produced the first inhomogeneous Einstein metrics on spheres, together with further examples on other low-dimensional spaces. Foscolo--Haskins  \cite{Haskins-Lorenzo} subsequently used cohomogeneity-one methods to construct the first complete inhomogeneous nearly K\"ahler structures on $S^6$ and $S^3\times S^3$. In our situation, local solutions can be obtained by lifting the nearly half-flat $\mathrm{SU}(3)$-structures on $6$-dimensional principal orbits \cites{nearlyhypo-nhf,liftingsu3tonhf,cvlt-liftnhf}. Previously Podestà \cite{podesta_2021} constructed incomplete inhomogeneous cohomogeneity-one nearly parallel $\mathrm G_2$-metrics on $S^3\times \R^4$. 

The difficulty in obtaining complete examples is a global problem. Indeed, the induced metric has positive Ricci curvature, so Bonnet–Myers implies that every complete nearly parallel $\mathrm G_2$-manifold is compact and the fundamental group $\pi_1(M)$ is finite. The orbit space is therefore a closed interval with a singular orbit at each end. On the open set of principal orbits, the nearly parallel equation reduces to a finite-dimensional system of ODEs. Completeness is then equivalent to a global boundary-value problem.  A trajectory emanating smoothly from one singular orbit must remain non-degenerate and close smoothly at a second singular orbit. 

By a classification result of \cite{Cleyton2004}, the principal orbit $P=G/H$ of a cohomogeneity-one nearly parallel $\G2$ manifold $M$ is up to finite quotients one of
\[
    S^6 = \frac{G_2}{SU(3)},\quad\mathbb{CP}^3 = \frac{Sp(2)}{Sp(1)\times \U(1)},
    \quad \mathbb{F}_{1,2} = \frac{SU(3)}{T^2},\quad
     S^3 \times S^3 = \frac{\SU(2)^3}{\SU(2)}
\]
arising in nearly K\"ahler geometry, or one of $S^5 \times S^1$,
$S^3 \times T^3$, or $T^6$. When $G$ is simple, the authors of \cite{Cleyton2004} showed that the homogeneous structure is the only possibility, hence we cannot get any new examples in these cases. For the non-simple cases, $P=X\times T^m$, $m\ge0$, $\pi_1(X)=0$, and Seifert–Van Kampen implies that
$\operatorname{rank}\pi_1(M) \ge\ m-2$.
It follows that the only possible principal orbits for a cohomogeneity-one nearly parallel $\G2$-structure are $S^3\times S^3$ and $S^5\times S^1$, up to finite quotients.

In this article, we focus on the case where $P$ is a finite quotient of $S^3\times S^3$. We study nearly parallel structures invariant under a cohomogeneity-one action of $\mathrm{SO}(4)$, with principal orbit
$$P=\mathrm{SO}(4)/\mathbb Z_2^2
\cong (S^3\times S^3)/\Delta Q_8.$$ 
Each singular orbit has codimension two and corresponds to the collapse
of a circle in the principal orbit. Passing to the double cover
$\Sp(1)\times\Sp(1)$ of $\SO(4)$, each collapsing circle takes the form
\[
C^u_{p,q}
 =\bigl\{(e^{pu\theta},e^{qu\theta}):\theta\in\mathbb R\bigr\}
\subset S_u^1\times S_u^1\subset\Sp(1)\times\Sp(1),\]
for $u\in\{i,j,k\}$. 
Thus $u$ specifies the maximal torus containing the collapsing circle, while the weights $p,q$ determine the so-called slope $(p,q)$ of the singular orbit. 

Known examples of nearly parallel $\G2$-structures that lie in this class are the round and squashed $7$-sphere, the Berger space $SO(5)/SO(3)$, and the 3-Sasakian structures discovered by Grove--Wilking--Ziller \cite{ziller} that are associated to the $\Lambda^2_-$-bundle over the self-dual Einstein orbifold metrics on $S^4$ introduced by Hitchin \cite{Hitchin_Einstein}. This $SO(4)$ action constitutes a natural testing ground for constructing inhomogeneous examples of proper nearly $\G2$-manifolds, and we shall construct complete inhomogeneous proper nearly parallel $\G2$-metrics of this type on two distinct manifolds.
The first is a new inhomogeneous nearly $\G2$ metric on the isotropy irreducible space SO(5)/SO(3). It admits a cohomogeneity-one action arising from the choice of an inclusion of $SO(4)$ in $SO(5)$.

\medskip
\noindent
\textbf{Theorem \ref{thm:certified-berger}} \textit{There exists a smooth inhomogeneous complete nearly parallel $G_2$-structure on
$SO(5)/SO(3)$, invariant under the cohomogeneity-one $SO(4)$-action.}

\medskip
We also apply our machinery to obtain a finite quotient of the $\SU(2)^2\times \rm{U}(1)$-invariant Sasaki--Einstein structure on a member of the non-primitive family $N_A^7$ in Hoelscher's classification of compact simply connected cohomogeneity-one
seven-manifolds \cite{Hoelscher2010}. The corresponding simply connected cohomogeneity-one manifold is the
Sasaki--Einstein manifold $A^{1,3,2}$ of \cite{TomasielloZaffaroni2011}, and we show that
the certified solution coincides with their explicit metric. Although this construction does not produce a new metric, we nevertheless retain the validated shooting calculation for this solution. It recovers $A^{1,3,2}$ from the singular initial-value problem and provides a benchmark for the computer-assisted framework. The metric admits a free involution $\iota$, and we write $M_{23}=A^{1,3,2}/\langle\iota\rangle$. Since $\iota$ acts as a reflection on the two-dimensional space of real Killing spinors on $A^{1,3,2}$, exactly one Killing spinor descends to $M_{23}$, defining a proper nearly parallel $\G2$-structure on the quotient. 
\smallskip
Our construction is inspired by work of Foscolo--Haskins \cite{Haskins-Lorenzo}, who combine the analysis of
cohomogeneity-one solutions emanating from singular orbits with discrete
symmetries of the evolution equations. A solution that meets the fixed
locus of a time-reversing symmetry can be reflected across the
corresponding principal orbit, thereby producing a smooth compact
solution. We first
determine the families of solutions that extend smoothly from the
relevant singular orbits. The problem of obtaining a complete solution is
reduced to finding initial data and a positive time at which the
corresponding trajectory meets the fixed locus of an appropriate
time-reversing symmetry. Reflection across the principal orbit at that
time produces a solution closing smoothly at a second singular orbit. 

The geometric reflection principle in \cite{Haskins-Lorenzo} carries over to the $\G2$ setting, but the argument used to find the reflecting trajectory does not. In the nearly K\"ahler problem, the invariant evolution equations and the smooth singular initial data can be reduced, after passing to maximal-volume orbits, to an essentially one-dimensional matching problem. The required compact solutions can therefore be detected through intersections of curves and controlled by qualitative estimates. The $\G2$ problem has a substantially larger deformation space. Both the space of invariant structures on a principal orbit and the families of solutions extending smoothly from a singular orbit have additional parameters, so the maximal-volume reduction no longer turns the reflection conditions into a one-dimensional problem.
This difference is already visible at a singular orbit with slope $(3,-1)$. Although the induced geometry of the singular orbit is described by two parameters, it does not determine the nearby nearly parallel $\G2$-structure uniquely. A further parameter enters at second order in the smooth Taylor expansion and generically breaks the additional $\U(1)$-symmetry. Consequently, the relevant local solutions form a three-parameter family, and the reflection conditions become multidimensional shooting problems, a four-dimensional map for the Berger space and, after imposing the additional $\U(1)$-symmetry, a three-dimensional map for $M_{23}$. Thus the curve-intersection argument of \cite{Haskins-Lorenzo} cannot be transplanted directly. We instead combine the Eschenburg--Wang smoothness criterion, which determines the admissible singular initial data, with validated ODE integration and the Krawczyk inclusion theorem to prove the existence of the required zeros. Similar arguments were recently used by \cites{NienhausWink2025,Buttsworth-Hodgkinson,wang2026doublywarpedproducteinstein,Vasquez} to construct complete inhomogeneous Einstein metrics on higher-dimensional spheres, products of spheres, and complex projective spaces.

The same evolution equations also reveal what happens when a smooth singular orbit shrinks to a point. This gives rise to an alternative approach to constructing complete nearly parallel
$\G2$-manifolds: begin with a compact singular model and attempt to desingularise it. A natural source of such models is the sine cone construction.  If $N^6$ carries a nearly K\"ahler structure, then $SC(N)= (0,\pi)\times N$ carries a nearly parallel $\G2$-structure endowed with the metric $dt^2+(\sin t)^2g_N$ \cites{Bar_realkilling, friedkath} and generally has conical vertices at $t=0$ and $t=\pi$. The tangent cone $C(N)=(0,\infty)\times N$ at each vertex of the sine cone is equipped with
a metric $dr^2+r^2g_N$ possessing a torsion-free $\G2$-structure, i.e.\ $d\g2=0$ and $d^*\g2=0$.
When $N=S^6$, the cone $C(N)$ is Euclidean and the metric completion of $SC(N)$ is the round
$7$-sphere, but in general there are genuine singularities. A prerequisite for smoothing one end is
the existence of a complete asymptotically conical (AC) torsion-free $\G2$-manifold with asymptotic cone $C(N)$. After rescaling, such a
manifold provides the local model and `bubble' that one expects to insert at the singular point.

Several asymptotically conical torsion-free $\G2$-metrics with
homogeneous nearly K\"ahler links are known. The classical metrics of
Bryant--Salamon \cite{BryantSalamon1989} resolve the cones over $\mathbb{CP}^3$,
$\SU(3)/T^2$, and $S^3\times S^3$. The study of complete
cohomogeneity-one torsion-free $\G2$-metrics asymptotic to the cone
over $S^3\times S^3$ and its finite quotients has produced several
new families; see, in particular,
\cites{BrandhuberGomisGubserGukov2001,Bogoyavlenskaya2013,FHN}.
Desingularisations of compact torsion-free $\G2$-spaces with isolated
conical singularities have also been studied by Karigiannis
\cite{KarigiannisDesingularization}. The nearly parallel problem,
however, is different; the singular background is not torsion-free, and
the AC metric is expected to arise only after blowing up a neighbourhood
of the singularity.

The analogous picture has already proved fruitful in nearly K\"ahler geometry. The cohomogeneity-one metrics of \cite{Haskins-Lorenzo} may be viewed as smoothings of the sine cone over the Sasaki--Einstein
manifold $S^2\times S^3$, with asymptotically conical Calabi--Yau metrics associated with the conifold appearing as bubbles in the
degenerating limit \cites{CandelasDeLaOssa1990, Foscolo2018}.  This six-dimensional picture motivates our
local desingularisation of the nearly parallel $\G2$ sine cone by the
AC $C_7$ metric.

In this paper, we desingularise one end of the sine cone over $N=(S^3\times S^3)/\mathbb Z_4$, equipped with the nearly K\"ahler structure induced from the homogeneous one on $S^3\times S^3$. The smoothing is obtained from solutions extending across a singular orbit whose collapsing circle has slope $(1,-1)$. 
We first determine the family of nearly parallel
$\G2$-solutions that extend smoothly over this singular orbit and
then select a limiting one-parameter subfamily for which the orbit collapses.
Away from the collapsing orbit, these solutions converge to the sine
cone, so that its cone vertex at $t=0$ is replaced by a smooth
singular orbit. The geometry near the collapsing orbit is revealed by rescaling it to
have fixed size. Under this rescaling, the nearly parallel constant
tends to zero and the equations converge to the torsion-free
$\G2$-equations. The resulting limit is the AC
member of the $C_7$ family constructed in \cite{FHN}. Up to conjugation its singular orbit also has
slope $(1,-1)$, and its asymptotic cone agrees with $C(N)$. Thus the AC $C_7$ metric is precisely the one that incorporates the
bubble that replaces the conical tip. We obtain

\medskip
\noindent
\textbf{Theorem \ref{thm:local-desingularisation}} \textit{Consider the locally homogeneous 
nearly K\"ahler manifold $N=(S^3\times S^3)/\mathbb Z_4$.
There exists a one-parameter family of cohomogeneity-one nearly parallel $\G2$-structures that closes
smoothly on a singular orbit of slope $(1,-1)$ and degenerates to $SC(N)$ with the
AC $C_7$ metric as its blow-up limit.}

\medskip
More precisely, the family converges smoothly to the sine cone away
from the collapsing singular orbit, while, after blowing up a
neighbourhood of that orbit, it converges smoothly on compact subsets
to the AC $C_7$ metric. The proof is formulated as a dynamical
matching problem near the cone common to these two limiting geometries.
After introducing scale-invariant variables and logarithmic radial
time, the evolution equations become autonomous and the torsion-free
cone becomes a hyperbolic fixed point. A similar 
approach to conical limits in cohomogeneity-one special-holonomy
geometry was used by Lehmann in the $\Spin(7)$ setting
\cite{LehmannGeometricTransitions}. The $C_7$ trajectory approaches
the fixed point along a stable direction, whereas the sine cone
trajectory leaves it along one of two unstable directions. The other
unstable direction obstructs the desired matching. Using the comparison
theory for the $C_7$ family developed in \cite{FHN}, we define a
scalar matching function that measures the component in this unwanted
direction and show that its derivative with respect to the $C_7$
parameter is non-zero at the AC solution. The implicit function theorem
then allows us to choose the parameter so that this component vanishes,
thereby selecting a desingularising family whose outgoing limit is the
sine cone trajectory.

The possibility of desingularising nearly parallel $\G2$-conifolds
by AC $\G2$-holonomy manifolds was investigated by
Schiemanowski \cite{Schiemanowski2022}. He derived a topological
obstruction to such constructions and showed that AC $\G2$-manifolds decaying faster than rate $-7/2$ cannot occur as bubbles. The AC $C_7$ metric has rate $-3$ and lies outside this non-existence regime. Our family realises this borderline picture, and smooths one vertex of the sine cone.  This is the first positive desingularisation result of this kind in nearly parallel
$\G2$-geometry and identifies the AC $C_7$ metric as a genuine
bubble arising from a degenerating family with non-zero nearly parallel
torsion. It provides a local prototype for constructing
compact nearly parallel $\G2$-manifolds by desingularising sine
cones. 

\paragraph{Organisation of the paper.} Our work divides into two complementary parts. The first concerns the global
shooting problem required to produce the new inhomogeneous nearly parallel $\G2$-metric on SO(5)/SO(3). The second studies a degenerating family of local
solutions and identifies the AC $C_7$ metric as the appropriate model appearing
at a collapsing singular orbit. We now describe the structure of the
paper and the main steps in these two threads. 

In Section~\ref{section:prelims} we first recall the description of nearly parallel $\G2$-structures as lifts of the nearly half-flat $\SU(3)$-structures induced on oriented hypersurfaces. This turns the SO(4)-invariant nearly parallel equation into an evolution problem on the principal orbits. We then review the cohomogeneity-one theory for $\SO(4)$, 
and the group diagrams, which determine both the smoothness conditions at the two ends and the topology of the compact manifolds obtained by reflection.

In Section~\ref{sec:open-orbit} we pass from the geometric problem to a system of ordinary differential equations \eqref{eqns:gen_ng2} whose solutions describe the invariant nearly parallel $\G2$-structure on $I\times SO(4)/\Z_2^2$ for an open interval $I\subset\R$. We first describe the space of 
invariant differential 3-forms on the principal orbit
$\SO(4)/\mathbb Z_2^2$, impose the nearly half-flat constraints, and
identify the domains in which the induced metric is positive
definite. The system possesses several discrete symmetries, some of which we record in Proposition~\ref{prop:discrete_symm}.  These symmetries play a crucial role in the reflection principle used to close a solution at its second singular orbit. 

Section~\ref{sec:smoothness} concerns the local problem near a singular orbit. Since the evolution equations become singular when one orbit direction collapses, smooth extension is a singular initial value problem. We use the Eschenburg--Wang smoothness criterion and a general singular IVP existence statement, Theorem~\ref{thm:IVPsoln}, to determine the admissible Taylor expansions for a solution to extend to a singular orbit. We obtain 3-parameter families of solutions extending to the relevant singular orbits with slopes $(1,-1)$ and $(3,-1)$.  Theorem~\ref{thm:complete-11} constructs two local families $\xi_{a,b,c},\widetilde\xi_{a,b,c}$ for slope $(1,-1)$, while 
Theorem~\ref{thm:complete-31} constructs the three-parameter family
$\eta_{a,c,\nu}$ for slope $(3,-1)$. It is the family
$\eta_{a,c,\nu}$ that we use to reflect so as to obtain the compact constructions.

In Section~\ref{sec:reflection} we attend to the global problem. The reflection principle Lemma~\ref{lemma:timereversing} gives a sufficient condition for a solution emanating from one singular orbit to give a complete solution precisely when it reaches the fixed-point set of a suitable time-reversing symmetry. This reduces the geometric existence problem to finding a zero of a finite-dimensional shooting map.

The main effort goes into establishing the existence of the required zeros, which we do by means of a rigorous computer-assisted argument. Because the equations are singular at the initial orbit $t=0$, they cannot be numerically integrated starting from $t=0$. Hence, we start the integration from a short positive time $t=\delta>0$.  We compute the Taylor expansion of the solution and rigorously bound the error terms. This gives an interval containing the solution and its parameter derivatives at $t=\delta$, where validated integration can begin.
Then we use the $C^1$-interval solver from CAPD to integrate the flow from $t=\delta$ up to the matching time. The final zero is certified using the Krawczyk inclusion criterion stated in Theorem~\ref{thm:krawczyk}. This produces the complete solution on the Berger space in
Theorem~\ref{thm:certified-berger}.
Proposition~\ref{prop:not-homogeneous-berger} distinguishes this solution from the classical homogeneous Berger metric, while
Proposition~\ref{prop:proper-solutions} shows that the structure is proper, or equivalently that the metric cone has full holonomy $\Spin(7)$. 

Section~\ref{sec:desingularisation} is devoted to the sine cone desingularisation. We begin by relating the $\xi_{a,b,c}$ family to the $C_7$ family of \cite{FHN} in Proposition~\ref{prop:c7-finite-quotient}. We then rescale the metric so that the singular orbit has fixed size. Under this rescaling the nearly parallel constant tends to zero, and Proposition~\ref{prop:inner-limit} shows that the rescaled solutions converge on compact subsets to members of the torsion-free $C_7$ family. To connect this inner limit to the sine cone geometry, we rewrite the equations in scale-invariant variables and logarithmic radial time. The common torsion-free cone then becomes a hyperbolic fixed point of the new system. The linearisation of the system at the cone fixed point then identifies the infinitesimal deformations in Lemma~\ref{lem:fixed-r0-modes}, and in the overlap region both the rescaled small-orbit
solutions and the sine cone solution enter a common neighbourhood of this fixed point, as shown in Lemma~\ref{lem:overlap}. The main issue is that a general trajectory can leave the cone in an unwanted unstable direction rather than following the
sine cone trajectory. Proposition~\ref{prop:exchange-nonzero} proves that varying the $C_7$ parameter moves the incoming trajectory transversely to this unwanted direction. In Theorem~\ref{thm:local-desingularisation} we choose the parameter so that the unwanted component vanishes, giving the desingularising family.

\paragraph{Acknowledgments.} This project grew out of work within the Simons Collaboration on Special Holonomy in Geometry, Analysis, and Physics (\#488635 Simon Salamon). Ragini Singhal is funded by the Deutsche Forschungsgemeinschaft (DFG, German Research Foundation) under Germany's Excellence Strategy EXC 2044 –390685587, Mathematics Münster: Dynamics–Geometry–Structure. R.S.\ would like to thank Christoph Böhm, Hans-Joachim Hein, Jason Lotay, Lorenzo Foscolo, and Marco Freibert for many fruitful conversations and suggestions about the project. Both authors would like to thank Daniel Platt for assistance with numerical methods and codes, and Shubham Dwivedi for careful reading of the manuscript.  

\paragraph{Use of generative AI.} ChatGPT 5.6 Sol was used to assist in the development of the computer programs and the preparation of the figures in this paper. The authors reviewed and tested all AI-assisted code and independently verified the resulting computations and figures, and take full responsibility for their correctness.

\section{Preliminaries}\label{section:prelims}

In this section, we discuss the setup for the study of a nearly parallel $\G2$-structure on a 7-dimensional compact manifold $M$ admitting a cohomogeneity-one action by a compact Lie group $G$. We can reduce attention to the \textit{nearly half-flat} $\SU(3)$-structures that will be induced on the invariant $6$-dimensional hypersurfaces of $M$. Nearly half-flat $\SU(3)$-structures were first introduced in \cite{nearlyhypo-nhf} in the context of
evolution equations on six-manifolds $M$ producing nearly parallel $\G2$-structures on the product of $M$ and an interval. The authors of \cites{cvlt-liftnhf,liftingsu3tonhf,Fabian-thesis,Conti-embedding} showed that one can construct nearly parallel $\G2$-structures by lifting certain nearly half-flat structures, by analogy with Hitchin's result \cite{Hitchin-stableforms} that a half-flat $\SU(3)$-structure on a six-dimensional manifold $N$ can be lifted to a parallel $\G2$-structure on $N\times(a,b)$ under certain conditions. The second author characterised the left invariant nearly half-flat structures on $S^3\times S^3$ \cite{SINGHAL-NHF}; as we shall see, this can be used to systematically analyse nearly parallel $\G2$-structures on $S^3\times S^3$ times an interval.

\subsection{Nearly Parallel $\G2$-structures}\label{section:g2-su3-prelim}

A $\G2$-structure on $M$ is defined by the reduction of the structure group of the frame bundle to the Lie group $\G2\subset \SO(7)$. A $\G2$-structure is characterized by a \textit{positive} $3$-form $\g2\in\Omega^3(M)$, meaning that $\g2$ lies in the open orbit with stabiliser $\G2$ at each point \cites{Gray1971,Salamon1989}. Such a structure exists if and only if the manifold is orientable and spin, conditions that are equivalent to the vanishing of the first and second Stiefel--Whitney classes respectively. The $3$-form $\g2$ induces a Riemannian metric $g_\g2$ in  a non-linear fashion, an orientation $\vol_\g2$ on $M$, and hence a Hodge
star operator $*_\g2$. We denote the Hodge dual $4$-form $*_{\g2}\g2$ by $\psi$. Pointwise, one has $\|\g2\|^2=\|\psi\|^2 = 7$, where the norm is taken with respect to $g_\g2$.  

The action of $\G2$ on the space of differential forms $\Omega^k$ induces splittings
\begin{align*}
\Omega^2&=\Omega^2_7\oplus \Omega^2_{14}, \\
\Omega^3&=\Omega^3_1\oplus \Omega^3_7\oplus \Omega^3_{27}.
\end{align*}
where
\begin{align*} 
 \Omega^2_7 &=\{X\lrcorner \g2\mid X\in \Gamma(TM)\} = \{\beta \in \Omega^2\mid *(\g2\wedge \beta)=2\beta\} , \\
 \Omega^2_{14} &=\{\beta \in \Omega^2(M)\mid \beta \wedge \psi =0 \} = \{\beta\in \Omega^2\mid *(\g2\wedge \beta)=-\beta\}. 
 \end{align*}
and
\begin{align*}
\Omega^3_1 &=\{ f\g2 \mid f\in C^{\infty}(M)\}, \\
\Omega^3_7 & = \{ X\lrcorner \psi \mid X\in \Gamma(TM)\} = \{*(\alpha \wedge \g2) \mid \alpha \in \Omega^1\}, \\
\Omega^3_{27} & = \{ \eta \in \Omega^3(M) \ \mid \ \eta\wedge \g2 = 0 = \eta\wedge \psi\}.
\end{align*}
The decompositions of $\Omega^4$ and $\Omega^5$ are obtained by taking the respective Hodge stars with respect to $*_\g2$. See \cite{Salamon1989}.

Given a $\G2$-structure $\g2$ on $M$, we can decompose $d\g2$ and $d\psi$ according to the above decomposition. This defines the \emph{intrinsic torsion forms}, which are unique differential forms $\tau_0 \in \Omega^0$, $\tau_1 \in \Omega^1$, $\tau_2 \in \Omega^2_{14}$ and $\tau_3 \in \Omega^3_{27}$ such that 
\begin{align*}
d\g2 &= \tau_0\psi + 3\tau_1\wedge \g2 + *_{\g2}\tau_3,  \\
d\psi &= 4\tau_1\wedge \psi + *_{\g2} \tau_2.
\end{align*}
\cites{bryantrmks,skflow}.
The full intrinsic torsion $\tau$ lives in the $49$-dimensional space 
\begin{align*}
    T^*M\otimes \mathfrak{g}_2^\perp \cong \mathcal{X}_0\oplus \mathcal{X}_1\oplus \mathcal{X}_2\oplus\mathcal{X}_3
\end{align*}
which has four irreducible $\G2$-modules. This gives rise to the sixteen classes of $\G2$-structures, and $\tau=0$ if and only if $d\g2=d\psi=0$ \cites{Fernandez-Gray, classification_G2}. A nearly parallel $\G2$-structure is characterized by
$0\ne\tau\in\mathcal{X}_0$, which is equivalent to
\begin{defn}
A $\G2$-structure $\g2$ is {\bf{nearly parallel }} if and only if there exists $\lambda\neq 0$ such that
\begin{align}
d\g2=\lambda\psi \ \ \ \ \textup{and} \ \ \ \ d\psi=0. \label{eq:ng2defn}
\end{align} 
\end{defn}

 The Lie group $\G2$ is the group of automorphisms of the octonions $\mathbb{O}$ that preserve the splitting $\mathbb{O}\cong \R+\rm{Im}\mathbb O$. One can define $\SU(3)$ as the subgroup of $\G2$ that preserves a fixed imaginary unit octonion. This fact indicates the presence of an $\SU(3)$-structure on a hypersurface of a manifold with a $\G2$-structure, and led Calabi \cite{Calabi-hypersurface} and Gray \cite{Gray-hypersurface} to study induced $\SU(3)$-structures on orientable hypersurfaces of $\rm{Im} \mathbb{O}$.   

An $\SU(3)$-structure on a $6$-dimensional manifold $M$ is defined by a pair $\omega,\Omega$ where $\omega$ is a non-degenerate  2-form and $\Omega=\Omega_++i\Omega_-$ is a complex $(3,0)$ form. These forms satisfy
\begin{enumerate}
    \item $\omega\wedge\Omega_\pm =0$,
    \item $\Omega_+\wedge \Omega_-=\frac{2}{3}\omega^3$.
\end{enumerate}
Let $I\subset\R$ be an interval. Given a one-parameter family of $\SU(3)$-structures $(\omega(t),\Omega(t))$ on $N$, one can define a $\G2$-structure $\g2$ on $I\times N$ by 
 \begin{align}
     \begin{split}\label{g2-eqns}
         \g2&=dt\wedge\omega(t)+\Omega_+(t)\\
         \psi=*_\g2\g2&=\textstyle\frac12\omega^2(t)-dt\wedge \Omega_-(t).
     \end{split}
 \end{align}

We are interested in the situation where the above $\G2$-structure on $I\times N$ is nearly parallel, that is $d\g2=\lambda\psi$ for some non-zero scalar constant $\lambda$. 
The $\G2$-structure $\g2$  defines a nearly $\G2$-structure if and only if 
\begin{align}
\begin{split}\label{SU3FORNG2}
d\Omega_+(t)&=\frac{\lambda}{2}\omega^2(t)\\
       d\omega(t)&=\Omega_+'(t)+{\lambda} \Omega_-(t),\\
        d\Omega_-(t)&=-\frac{(\omega^2)'}{2}.  
\end{split}
    \end{align}
The first condition describes the intrinsic torsion of the SU(3)-structure on the hypersurface $N$. This $\SU(3)$-structure was first described in \cite{nearlyhypo-nhf} and was called a nearly half-flat $\SU(3)$-structure.

\begin{defn}
  An $\SU(3)$-structure $(\omega,\Omega)$ on a 6-manifold $N$ is nearly half-flat if for some non-zero real constant $\lambda$
    \begin{align}\label{eqn:nhf}
        d\Omega_+&= \frac{\lambda}{2}\omega^2.
    \end{align}
\end{defn}

   If the scalar $\lambda=0$ in the above definition and $d\omega^2=0$, then the SU(3)-structure is called \textit{half-flat}. A 6-manifold $N$ with an SU(3)-structure that is (nearly) half-flat can (at least
if it is real analytic) be embedded in a manifold with (nearly) parallel $\G2$-structure. The simplest instance of this construction occurs when $N$ is nearly K\"ahler, in which case the
conical metric over $N$ has holonomy $\G2$. In fact, nearly half-flat is a slight generalisation of the half-flat situation to one in which the $3$-form $\Omega_+$ is not closed. 
    
\subsection{Cohomogeneity-one actions and group diagrams}\label{section:coho-1action}

Suppose that a compact Lie group $G$ acts with cohomogeneity one on a
connected compact manifold $M$. When the orbit space is a closed
interval, the action is described by a group diagram
\[
        H\subset K^-,K^+\subset G,
        \qquad K^\pm/H\simeq S^{\ell_\pm}.
\]
Here $H$ is the principal isotropy group and $K^\pm$ are the
isotropy groups of the two singular orbits. Conversely, such a diagram
determines a $G$-manifold, up to the usual equivalences of group
diagrams, by
\[
 M\simeq
 G\times_{K^-}D^{\ell_-+1}
 \mathop{\cup}_{G/H}
 G\times_{K^+}D^{\ell_++1}.
\]
Thus the topology of $M$ depends on the embeddings of $H$ and
$K^\pm$ in $G$, and not only on the abstract diffeomorphism types
of the singular orbits.

In this article, we focus on the cohomogeneity-one action by $\SO(4)$ on the principal $\rm{Sp}(1)$-bundles over $S^4$ as described in \cites{Verdiani-Podesta,GroveZiller2000,ziller,GWZ}. The bundles are classified by an element in $\pi_3(S^3)= \mathbb{Z}$. To describe the cohomogeneity-one action of $SO(4)$ on the $\Sp(1)$-bundle over $S^4$, we begin by describing the well-known
cohomogeneity-one action by SO(3) on $S^4$. 

Let $V\cong \mathrm{Sym}^2_0(\R^3)$ be the five-dimensional irreducible $\SO(3)$-representation on the space of traceless symmetric $3\times 3$ real matrices. The group $\SO(3)$ acts by conjugation on $V$, and the action descends to $S^4\subset V$. Every point in $S^4$ is conjugate to a matrix in $F = \{\operatorname{diag} (\lambda_1, \lambda_2, \lambda_3) \ | \ \sum_i\lambda_i=0, \ \sum_i \lambda_i^2=1\}$, and hence the quotient space is 1-dimensional. If the eigenvalues of $B\in V$ are distinct, then the stabilizer of $B$ is conjugate to the group $H\cong \Z_2\times \Z_2$ of diagonal matrices in $\SO(3)$. The exceptional orbits occur when two eigenvalues coincide. The singular orbits consist of those matrices for which the coincident eigenvalues are both positive or both negative, respectively. The singular isotropy group is $K\cong S(O(1)\times O(2))$, so each singular orbit is diffeomorphic to $\SO(3)/\operatorname{O}(2)\cong \mathbb{RP}^2$ and is a Veronese surface in $S^4$. 

Under the two-fold covering $\operatorname{Sp}(1)\to \SO(3)$, the group $H$ lifts to the quaternionic group $ Q_8 = \{\pm 1,\pm i, \pm j,\pm k\}$.  Writing
\[
S^1_i=\{e^{i\theta}:\theta\in\R\},
\qquad
S^1_j=\{e^{j\theta}:\theta\in\R\},
\]
the lifted group diagram of the action on $S^4$ is
\[
Q_8
\subset
\left\{S^1_iQ_8,S^1_jQ_8\right\}
\subset
\mathrm{Sp}(1).
\]
Here $S^1_{i,j}Q_8/Q_8\cong S^1$, corresponding to the two-dimensional normal
discs of the singular orbits.

We can now define the SO(4)-action on the bundle. We work with the double cover
\begin{align}\label{eqn:G_tilde}
    \pi\colon\widetilde G=\Sp(1)\times\Sp(1)\longrightarrow \SO(4).
\end{align}
For $u\in\{i,j,k\}$ and relatively prime integers $p,q$, set
\begin{align}\label{def:Cu_pq}
    C^u_{p,q}&=
 \left\{(e^{pu\theta},e^{qu\theta}):\theta\in\mathbb R\right\}
 \subset \widetilde G.
\end{align}
The superscript $u$ records the unit imaginary quaternion determining the maximal circle $S^1_u\subset\Sp(1)$, whereas the pair
$(p,q)$ records the weights in the two $\operatorname{Sp}(1)$-factors. We denote by
\begin{align}\label{eqn:K_pq}
    K^i_{p,q}&:=\pi(C^i_{p,q} H).
\end{align}

In all the examples considered below, the principal isotropy group is a copy of the quaternion group $Q_8$, and the identity components of the singular isotropy groups are circles of the form in \eqref{def:Cu_pq}.

\paragraph{The Grove--Wilking--Ziller family \cite{GWZ}.}
Set
\begin{align*}
    H_P&=\langle(i,-i),(j,j)\rangle
 =
 \{\pm(1,1),\pm(i,-i),\pm(j,j),\pm(k,-k)\}.
\end{align*}

The manifolds $P_k$ are defined by
\begin{align*}
    H_P\subset
 &\left\{
 C^i_{1,-1}H_P,\,
 C^j_{2k-1,\,-2k-1}H_P
 \right\}
 \subset\Sp(1)\times\Sp(1),
 \qquad k\geq 1.
\end{align*}
The two collapsing circles lie on different quaternionic axes; hence, the two singular isotropy
groups are not contained in a common proper subgroup of
$\Sp(1)\times\Sp(1)$ and the action is primitive. The manifolds $P_k$ are 2-connected and
are equivariantly diffeomorphic to the universal covers of the
3-Sasakian manifolds associated with the Hitchin orbifolds with a transverse $\Z_{2k-1}$ singularity. When $k=1$, we obtain the $3$-Sasakian $S^7$. 

\paragraph{The Berger space $B^7$.}
For the Berger space, let
\begin{align*}
    H_B&=\Delta Q_8=\langle(i,i),(j,j)\rangle
 =
 \{\pm(1,1),\pm(i,i),\pm(j,j),\pm(k,k)\}.
\end{align*}

The lifted diagram of its cohomogeneity-one $\SO(4)$-action is
\begin{align*}
     H_B & \subset
 \left\{
 C^i_{3,-1}H_B,\,
 C^j_{1,-3}H_B
 \right\}
 \subset\Sp(1)\times\Sp(1).
\end{align*}

Thus
\begin{align*}
    B^7&=\frac{\SO(5)}{\SO(3)}
 \simeq
 M\bigl(H_B;C^i_{3,-1}H_B,C^j_{1,-3}H_B\bigr).
\end{align*}

In \cite{Hoelscher2010}, Hoelscher described all the possible cohomogeneity-one actions on compact simply connected 7-manifolds. The family $P_k$ and the Berger space appear there as primitive manifolds, but the author also describes a family of non-primitive manifolds $N^7_A$ \cite{Hoelscher2010}. 

In Proposition \ref{prop:certified-tau23} we rediscover a complete nearly parallel $\G2$-structure on a member of this family, which we describe next. 

\paragraph{The non-primitive manifold $M_{23}$.}
Consider the simply connected cohomogeneity-one manifold defined by the
group diagram
\[
 \Delta\mathbb Z_4
 \subset
 \left\{C^i_{3,-1},C^i_{1,-3}\right\}
 \subset S^3\times S^3.
\]
This is a member of the non-primitive family $N_A^7$ in Hoelscher's classification \cite{Hoelscher2010}, equivariantly diffeomorphic to the Sasaki--Einstein manifold $A^{1,3,2}$ of \cite{TomasielloZaffaroni2011}.

The manifold $A^{1,3,2}$ admits a free involution $\iota$ compatible with the cohomogeneity-one action. We denote its quotient by
\[
 M_{23}\coloneqq A^{1,3,2}/\langle\iota\rangle.
\]
It is a rational homology sphere with fundamental group $\mathbb Z_2$, and the quotient action has
effective group diagram
\[
 \mathbb Z_2^2
 \subset
 \left\{K^i_{3,-1},K^i_{1,-3}\right\}
 \subset\SO(4).
\]
Topologically, it is an $\mathbb{RP}^3$-bundle over the Grassmannian $\operatorname{Gr}_2(\mathbb R^4)\cong (S^2\times S^2)/\Delta\Z_2$ of non-oriented subspaces.

The above diagram should be distinguished from the Berger diagram: in the Berger case the two collapsing circles $S_i^1$ and $S_j^1$ are not contained in a common maximal torus, whereas for $M_{23}$ they lie on the same maximal torus $S^1_i\times S^1_i$.

\begin{rem}
    We use a different embedding of $Q_8\in \Sp(1)\times \Sp(1)$ from \cite{GWZ}. One can obtain the conventions used in \cite{GWZ} by conjugating the second $\Sp(1)$ factor by $k$. This sends $H_P \to \Delta Q_8$ and $H_B \to \langle (i,-i),(j,-j)\rangle$. For the singular isotropy groups the conjugation sends
    \[
        C^u_{p,q}\longmapsto C^u_{p,-q}, \qquad u\in \operatorname{Im}\mathbb{H}.
        \]
\end{rem}

\section{Invariant $\G2$-structures on the regular part}
\label{sec:open-orbit}

An oriented hypersurface in a $7$-manifold with a $\G2$-structure
inherits an $\SU(3)$-structure. Thus, in order to describe the invariant $\G2$-structure on the open orbit $\SO(4)/\Z_2^2\times I$ for an open interval $I\subset \R$, we begin by describing the invariant
$\SU(3)$-structures on the principal orbit $\SO(4)/\mathbb Z_2^2$.
As seen in \Cref{section:g2-su3-prelim}, an oriented hypersurface of a nearly parallel $\G2$-manifold $(M^7,\g2)$ is induced by a nearly half-flat SU(3)-structure $(\omega,\Omega)$ defined in \eqref{eqn:nhf}. We can lift a family of real analytic nearly half-flat SU(3)-structures on $N^6$ to a nearly parallel $\G2$-structure on $N^6\times I$ if they satisfy \eqref{SU3FORNG2}.

If we denote by $S_{i,j}$ the symmetric $3\times 3$ matrix with 1 in the $(i,j)$ and $(j,i)$ entries and 0 elsewhere, then \begin{align*}
    E_1&:=\frac{1}{\sqrt{6}}\rm{diag}(1,1,-2), \quad
     E_2:=\frac{1}{\sqrt{2}}\rm{diag}(1,-1,0), \quad
    E_3:=S_{12}, \quad E_4:=S_{13},\quad E_5:=S_{23}
\end{align*} defines a basis of $\mathrm{Sym}^2_0(\R^3)\cong \R^5$. Let $A\in\mathrm{Sym}^2_0(\R^3)$. By \cite{Hitchin_Einstein}, if the eigenvalues of $A$ are distinct, then the stabilizer of $A$ is conjugate to the group $\{a,b\in\Z_2 \ | \ \rm{diag}(a,b,ab)\}\subset\rm{SO}(3)$. We can define the subgroup $\rm{SO}(4)_{E_1}\subset \rm{SO}(5)$ as the subgroup preserving the $E_1$ direction in $\R^5$. The generic stabilizer group for $\SO(4)_{E_1}$ is also given by $\Z_2^2\cong \rm{diag}(1,1,ab,b,a)$, which preserves the $E_1,E_2$ directions in $\R^5$. The generic orbit is therefore given by $\rm{SO}(4)/\Z_2^2$. The stabilizer of the singular orbit is the group $K^\pm\cong \rm{O}(2)$ such that $K_0^\pm \cong \rm{SO}(2)$. 

The tangent space at a point of a principal orbit $\SO(4)/\Z_2^2$ is identified with \[\mathfrak{so}(4)=\mathfrak{su}(2)_+\oplus\mathfrak{su}(2)_-
\cong\Lambda^2_+\oplus\Lambda^2_-.\] Following \cite{Madsen-Salamon}, we choose bases $\{e^1,e^3,e^5\}$ and $\{e^2,e^4,e^6\}$ for $\mathfrak{su}(2)_+$ and $\mathfrak{su}(2)_-$ respectively.
Let $f^1,\dots,f^4$ be the dual
basis on $\operatorname{span}\{E_2,E_3,E_4,E_5\}$, define 
\begin{align}\label{basis6}
   e^1&:=f^{12}+f^{34}, \quad e^3:=f^{13}-f^{24}, \quad e^5:=f^{14}+f^{23},\\
   e^2&:=f^{12}-f^{34}, \quad e^4:=f^{13}+f^{24}, \quad e^6:=-f^{14}+f^{23}
\end{align}
where $f^{ij}=f^i\wedge f^j$. The Maurer--Cartan equations give
\begin{align*}
    de^1=e^{35},\ de^2=e^{46},\ de^3=-e^{15},\ de^4=-e^{26},\ de^5=e^{13},\ de^6=e^{24}.
\end{align*}
The action of $\mathbb Z_2^2$ on $\mathbb R^4$ induces one on
$\Lambda^*(\mathbb R^4)$.  The invariant $2$-forms are spanned by
$e^{12},e^{34},e^{56}$, while the invariant $3$-forms are spanned by
$e^{135},e^{136},e^{145},e^{146},e^{235},e^{236},e^{245},e^{246}$.

An $\SU(3)$-structure on a 6-dimensional manifold is determined by a stable real
$3$-form $\Omega_+$ and a non-degenerate real $2$-form $\omega$ with
$\omega\wedge\Omega_+=0$.  The form $\Omega_+$ determines an almost
complex structure $J$ and the companion form
$\Omega_-=J\Omega_+$, so that $\Omega=\Omega_++i\Omega_-$ has type
$(3,0)$.  The volume matching and positivity conditions are
\begin{align*}
    3\Omega_+\wedge \Omega_-&=2\omega^3, \quad \omega(\cdot,J\cdot)>0.
\end{align*}
On a principal orbit $\{t\}\times \SO(4)/\Z_2^2$, a $\Z_2^2$-invariant $\SU(3)$-structure is given by
\begin{align}
    \label{eqn:omega_nhf}\omega&=p_1 e^{12}+p_2e^{34}+p_3e^{56},\\
   \label{eqn:omega_plus_nhf} \Omega_+&=q_1 e^{135}+q_2 e^{136}+q_3 e^{145}+q_4 e^{146}+q_5 e^{235}+q_6 e^{236}+q_7 e^{245}+q_8 e^{246}
\end{align}
where $p_1,p_2,p_3$ and $q_1,\dots,q_8$ are functions of $t$.  In
dimension six the pair $(\omega,\Omega_+)$ determines the full
$\SU(3)$-structure.  For notational convenience, set
\begin{align}\begin{split}\label{eq:alpha_beta_defn}
    \alpha_i&\coloneqq q_iq_{9-i}, i=1\dots 4,\\
    \beta_1&\coloneqq q_1q_4q_6q_7,\\
    \beta_2&\coloneqq q_2q_3q_5q_8.
\end{split}
    \end{align}
The condition $\omega\wedge\Omega=0$ is already satisfied and  the normalization condition $\Omega_+\wedge\Omega_-=2/3\omega^3$ implies that 
\begin{align}\label{eqn:normalization}
        \alpha&=-4(p_1p_2p_3)^2
    \end{align}
where 
\[\alpha:=\sum_{i,j=1,i<j}^4 (\alpha_i-\alpha_j)^2 - 2\sum_{i=1}^4\alpha_i^2+4(\beta_1+\beta_2).\]

The $3$-form $\Omega_-$ is then given by
\begin{align}
\begin{split}\label{eq:omega_minus_nhf}
    \Omega_-=& \frac{p_1p_2p_3}{\alpha}\Big(\frac{2\alpha_1(2\alpha_1-\4)+4\beta_2}{ q_8}e^{135}-\frac{2\alpha_1(2\alpha_1-\4)+4\beta_1}{ q_1}e^{246} \\
    &+\frac{2\alpha_2(2\alpha_2-\4)+4\beta_2}{ q_2}e^{245}-\frac{2\alpha_2(2\alpha_2-\4)+4\beta_1}{ q_7}e^{136}\\
    &+\frac{2\alpha_3(2\alpha_3-\4)+4\beta_2}{ q_3}e^{236}-\frac{2\alpha_3(2\alpha_3-\4)+4\beta_1}{ q_6}e^{145}\\
    &+\frac{2\alpha_4(2\alpha_4-\4)+4\beta_2}{ q_5}e^{146}-\frac{2\alpha_4(2\alpha_4-\4)+4\beta_1}{ q_4}e^{235}\Big).
\end{split}
\end{align}
where $\4$ is an abbreviation for the sum $\sum_{i=1}^4\!\alpha_i$.\\

Using the exterior derivatives of the one-forms on $\{t\}\times M$, we can compute 
\begin{align*}
    d\Omega_+&=-(q_2+q_7)e^{1234}+-(q_3+q_6)e^{1256}+-(q_4+q_5)e^{3456},
\end{align*}
thus the nearly half-flat condition \eqref{eqn:nhf} implies
\begin{align}
    \begin{split}\label{eqn:NHFconstraints}
    q_2+q_7&=- \lambda p_1p_2,\\
     q_3+q_6&= -\lambda p_1p_3,\\
        q_4+q_5&= -\lambda p_2p_3.
        \end{split}
\end{align}

\begin{rem}
    Left invariant nearly half-flat $\SU(3)$-structures on $S^3\times S^3$ were characterized in \cite{SINGHAL-NHF}. Any nearly half-flat $\SU(3)$-structure on $S^3\times S^3$ can be obtained by two $3\times 3$ real matrices $P,Q \in {\rm{Mat}_{3,3}(C^\infty(S^3\times S^3))}$ and two functions $a,b$.  The enhanced symmetry in the above system implies that the matrices $P,Q$ are diagonal and hence we get an $8$-parameter system. 
\end{rem}

Using the $\SU(3)$-structure on the principal orbit $M:=\SO(4)/\Z_2^2$ one can obtain  a $\G2$-structure $(\g2,\psi)$ on $(0,\pi/3)\times M$ by
\begin{align}
  \label{eq:phifromsu3}  \g2&=\Omega_++dt\wedge \omega,\\
 \label{eq:psifromsu3}   \psi=*\g2&= \frac{1}{2}\omega^2-dt\wedge \Omega_-.
\end{align}
From \eqref{eqn:omega_nhf},\eqref{eqn:omega_plus_nhf}, 
\begin{equation}\label{eq:gen-phi}
  \begin{aligned}
         \g2=&(p_1e^{12}+p_2e^{34}+p_3e^{56})\wedge dt \\ &+ q_1e^{135}+q_2e^{136}+q_3e^{145}+q_4e^{146}+q_5e^{235}+q_6e^{236}+q_7e^{245}+q_8e^{246}
         \end{aligned}
     \end{equation} 
where the coefficients $p_i,q_i$ are functions of $t\in (0,\pi/3)$.
Using the $3$-form above, one can compute the symmetric bilinear tensor 
\[(B_{\g2})_{ij}\vol=-\frac{1}{6}(e_i\lrcorner\g2)\wedge(e_j\lrcorner\g2)\wedge\g2.\]
The $3$-form $\g2$ defines a $\G2$-structure if and only if $B_\g2$ is positive definite. If it is positive definite, we obtain the metric $g_\g2$ induced by $\g2$ as 
\[g_\g2=\frac{B_\g2}{\det{B_\g2}^\frac{1}{9}}.\]

Using the notation in \eqref{eq:alpha_beta_defn}, the symmetric bilinear form 
$B_\g2={\rm{diag}}(B_1,B_2,B_3,-p_1p_2p_3)$ is given by
\begin{align*}
B_1&=\left(\begin{array}{c c}
    p_1(q_1q_4-q_2q_3) &\cfrac{p_1(\alpha_1+\alpha_4-\alpha_2-\alpha_3)}{2}\\
    \cfrac{p_1(\alpha_1+\alpha_4-\alpha_2-\alpha_3)}{2}& p_1(q_5q_8-q_6q_7)
\end{array}\right),\\ B_2&=\left(\begin{array}{c c}
    p_2(q_1q_6-q_2q_5) &-\cfrac{p_2(\alpha_2+\alpha_4-\alpha_1-\alpha_3)}{2}\\
     -\cfrac{p_2(\alpha_2+\alpha_4-\alpha_1-\alpha_3)}{2}&  p_2(q_3q_8-q_4q_7)
\end{array}\right), \\ B_3&=\left(\begin{array}{c c}
    p_3(q_1q_7-q_3q_5) &-\cfrac{p_3(\alpha_3+\alpha_4-\alpha_1-\alpha_2)}{2}\\
     -\cfrac{p_3(\alpha_3+\alpha_4-\alpha_1-\alpha_2)}{2}&  p_3(q_2q_8-q_4q_6)
\end{array}\right).
    \end{align*}
The determinant of $B_\g2$ is given by 
    \[\det(B_\g2)=-\left(\frac{p_1p_2p_3\alpha}{4}\right)^3.\]
    The normalisation condition \eqref{eqn:normalization} on the $\SU(3)$-structure on $\SO(4)/\Z_2^2$ implies that 
       \begin{align*}
        \det(g_\g2)&= \det(B_\g2)^{1/9}=|p_1p_2p_3|.
    \end{align*}

In the orientation and square-root convention used throughout the
shooting argument, the Riemannian chamber is characterised by
$p_1p_2p_3<0$, or equivalently by
\begin{equation}\label{eq:riemannian-chamber}
 q_2+q_7>0,\qquad q_3+q_6>0,\qquad q_4+q_5<0.
\end{equation}
Reversing the transverse orientation changes all three $p_i$ and gives
the equivalent convention $p_1p_2p_3>0$.  The Hodge
dual $4$-form $\psi=*_{\varphi}\varphi$ is
\begin{equation}
\begin{aligned}
    \label{eq:gen-psi}\psi\;=&\;p_1p_2 e^{1234}+p_1p_3e^{1256}+p_2p_3e^{3456}\\
    &-\frac{\alpha_1(2\alpha_1-\4)+2\beta_2}{2p_1p_2p_3q_8}e^{1357}+\frac{\alpha_1(2\alpha_1-\4)+2\beta_1}{2p_1p_2p_3 q_1}e^{2467} \\
    &-\frac{\alpha_2(2\alpha_2-\4)+2\beta_2}{2p_1p_2p_3 q_2}e^{2457}+\frac{\alpha_2(2\alpha_2-\4)+2\beta_1}{2p_1p_2p_3 q_7}e^{1367}\\
    &-\frac{\alpha_3(2\alpha_3-\4)+2\beta_2}{2p_1p_2p_3 q_3}e^{2367}+\frac{\alpha_3(2\alpha_3-\4)+2\beta_1}{2p_1p_2p_3 q_6}e^{1457}\\
    &-\frac{\alpha_4(2\alpha_4-\4)+2\beta_2}{2p_1p_2p_3q_5}e^{1467}+\frac{\alpha_4(2\alpha_4-\4)+2\beta_1}{2p_1p_2p_3 q_4}e^{2357}.
\end{aligned}
\end{equation}

The torsion of the $\G2$-structure is given by
\begin{align}
\begin{split}\label{eq:dphi}
    d\g2&=d\Omega_++dt\wedge(\dot{\Omega}_+ -  d\omega),\\
    &=(q_7-q_2)e^{1234}+(q_6-q_3)e^{1256}+(q_4-q_5)e^{3456}-q_1'e^{1357}-q_8'e^{2467}-(-p_3+q_2')e^{1367}\\
    &-(p_3+q_7')e^{2457}-(-p_2+q_3')e^{1457}-(p_2+q_6')e^{2367}-(-p_1+q_5')e^{2357}-(p_1+q_4')e^{1467}.
    \end{split}
    \end{align}
    \begin{align}
    \begin{split}\label{eq:dpsi}
        d\psi&= d\omega\wedge\omega+dt\wedge ( \dot\omega\wedge\omega+d\Omega_-)\\
        &=\left\{\left(\frac{\alpha p_1p_2p_3 }{4}\right)^{1/3}\left(-\frac{p_3'}{p_3^2}+\frac{2(q_2-q_7)\left(2\alpha_2-\4\right)}{\alpha}+\frac{2(\beta_1-\beta_2)}{\alpha_2}\right)+\frac{(\alpha p_1p_2p_3)'}{3p_3(2\alpha p_1p_2p_3)^{2/3}}\right\}e^{12347}\\
    &+\left\{\left(\frac{\alpha p_1p_2p_3 }{4}\right)^{1/3}\left(-\frac{p_2'}{p_2^2}+\frac{2(q_3-q_6)\left(2\alpha_3-\4\right)}{\alpha}+\frac{2(\beta_1-\beta_2)}{\alpha_3}\right)+\frac{(\alpha p_1p_2p_3)'}{3p_2(2\alpha p_1p_2p_3)^{2/3}}\right\}e^{12567}\\
    &+\left\{\left(\frac{\alpha p_1p_2p_3 }{4}\right)^{1/3}\left(-\frac{p_1'}{p_1^2}+\frac{2(q_5-q_4)\left(2\alpha_4-\4\right)}{\alpha}+\frac{2(\beta_1-\beta_2)}{\alpha_4}\right)+\frac{(\alpha p_1p_2p_3)'}{3p_1(2\alpha p_1p_2p_3)^{2/3}}\right\}e^{34567}.
    \end{split}
    \end{align}

The $\G2$-structure $(\g2,\psi)$ is torsion-free if it satisfies $d\g2=0$ and $d\psi=0$. Assuming $d\g2=0$, we obtain
\begin{subequations}\label{eqns:gen_tfg2}
\begin{equation}
p_1=-q_4',\ \ p_2=q_3', \ \ p_3=q_2'
\end{equation}
\begin{equation}
q_1'=0, \ \ q_8'=0,\ \ q_5=-q_4,\ \ q_6=-q_3,\  \ q_7=-q_2.
\end{equation}
\end{subequations}
The equation $d\psi=0$ is a non-linear second order ODE. 

\subsection{The nearly parallel evolution system}
If $(\g2,\psi)$ satisfies $d\g2=\lambda\psi$ for a non-zero scalar
$\lambda$, then the $\SU(3)$-structure on each principal orbit
$\{t\}\times M$ is nearly half-flat and the family satisfies
\eqref{SU3FORNG2}.

For $\alpha_i=q_iq_{9-i}, i=1\dots 4$, $\beta_1=q_1q_4q_6q_7, \beta_2=q_2q_3q_5q_8$, and $\4=\sum_{i=1}^4\!\alpha_i$, the ODEs for $q_i$ are given by 
\begin{subequations}\label{eqns:gen_ng2}
\begin{equation}
    q_1'=\frac{\lambda}{2p_1p_2p_3q_8} (\alpha_1(2\alpha_1-\4)+2\beta_2)
\end{equation}
\begin{equation}
    q_8'=\frac{-\lambda}{2p_1p_2p_3q_1} (\alpha_1(2\alpha_1-\4)+2\beta_1)
\end{equation}
\begin{equation}
    q_2'=p_3-\frac{\lambda}{2p_1p_2p_3q_7} (\alpha_2(2\alpha_2-\4)+2\beta_1)
\end{equation}
\begin{equation}
    q_7'=-p_3+\frac{\lambda}{2p_1p_2p_3q_2} (\alpha_2(2\alpha_2-\4)+2\beta_2)
\end{equation}
\begin{equation}
    q_3'=p_2-\frac{\lambda}{2p_1p_2p_3q_6} (\alpha_3(2\alpha_3-\4)+2\beta_1)
\end{equation}
\begin{equation}
    q_6'=-p_2+\frac{\lambda}{2p_1p_2p_3q_3} (\alpha_3(2\alpha_3-\4)+2\beta_2)
\end{equation}
\begin{equation}
    q_4'=-p_1+\frac{\lambda}{2p_1p_2p_3q_5} (\alpha_4(2\alpha_4-\4)+2\beta_2)
\end{equation}
\begin{equation}
    q_5'=p_1-\frac{\lambda}{2p_1p_2p_3q_4} (\alpha_4(2\alpha_4-\4)+2\beta_1).
\end{equation}
\end{subequations}

Let the constraint set be
\begin{equation}\label{eq:constraint-set}
 \mathcal C=\Bigl\{(p,q)\in\mathbb R^3\times\mathbb R^8:
 q_2+q_7=-\lambda p_1p_2,\quad q_3+q_6=-\lambda p_1p_3,
 \quad q_4+q_5=-\lambda p_2p_3\Bigr\}.
\end{equation}
On the open set where $p_1p_2p_3\neq0$, and in the chamber used
throughout this paper, the $p_i$ are recovered from $q$ by
\begin{equation}\label{eq:p-from-q}
p_1=-\sqrt{\cfrac{-(q_2+q_7)(q_3+q_6)}{\lambda(q_4+q_5)}},\quad
p_2=\sqrt{\cfrac{-(q_2+q_7)(q_4+q_5)}{\lambda(q_3+q_6)}},\quad
p_3=\sqrt{\cfrac{-(q_3+q_6)(q_4+q_5)}{\lambda(q_2+q_7)}}.
\end{equation}
The evolution system also has the discrete symmetries listed next.  The
first two are general, while the last two reflect the geometry of the
present principal and singular orbits.

\begin{prop}\label{prop:discrete_symm}
    The system \eqref{eqns:gen_ng2} is invariant under the following symmetries:
    \begin{enumerate}
        \item $\tau_1 \colon$ time translation, $t\leftrightarrow t+t_0$ for some fixed $t_0\in\R$. 
        \item $\tau_2 \colon$ \textit{time reversal}, $t\leftrightarrow -t$, and $p_i\leftrightarrow -p_i$ for $i=1,2,3$. 
        \item $\tau_3 \colon$ \textit{$\rm{Sp}(1)$-flip}, $p_i\leftrightarrow-p_i$ for $i=1,2,3$, and $q_{i}\leftrightarrow q_{9-i}$ for $i=1,\dots,8$. 
        \item $\tau_4 \colon$ \textit{$Q_8$-axis flip}, $t\leftrightarrow -t$,  $p_1\leftrightarrow-p_3$, $q_1\leftrightarrow q_8$, $q_3\leftrightarrow q_6$, $q_2\leftrightarrow-q_4$, $q_7\leftrightarrow -q_5$.
    \end{enumerate}
\end{prop}

\noindent Each map preserves $\mathcal C$.  The involutions $\tau_2,\tau_4$ are
time-reversing, while $\tau_1,\tau_3$ preserve time orientation.
Geometrically, $\tau_4$ sends 
$C^i_{p,q}\to C^j_{-q,-p}$, whereas $\tau_3$ exchanges the two $S^3$
factors in the principal orbit and sends $C^u_{p,q} \to C^u_{q,p}$ for $u\in\{i,j,k\}$.
\smallskip

There are some known solutions to \eqref{eqns:gen_ng2} satisfying $\mathcal{C}$. We discuss them briefly here. 
    
\paragraph{Sine cone with a nearly K\"ahler link.} An $\SU(3)$-structure $(\omega,\Omega)$ is strictly nearly K\"ahler if for some non-zero scalar $\mu\in\R$, $d\omega=3\mu\Omega_-, d\Omega_+=-2\mu\omega^2.$
Thus nearly K\"ahler is a distinguished subclass of nearly half-flat. 
    The canonical $\SO(4)/\Z_2^2$-invariant nearly K\"ahler $\SU(3)$-structure is given by 
    \begin{align*}
        \omega&= e^{12}+e^{34}+e^{56},\\
        \Omega_+&= e^{135}+\frac{1}{2}(e^{136}-e^{145}+e^{146}+e^{235}-e^{236}+e^{245})+e^{246},\\
        \Omega_-&= \frac{\sqrt3}{2}(e^{136}-e^{145}-e^{146}+e^{235}+e^{236}-e^{245}).
    \end{align*}
 The conical metric on $C(N)=(0,\infty)\times N$ over a nearly K\"ahler manifold $N$ has holonomy $\G2$, and (unless $N=S^6$) there is a conical vertex at $t=0$. The sine cone $SC(N)=(0,\pi)\times N$ carries a nearly parallel $\G2$-structure which 
 (unless $N=S^6$) has two vertices at the two end points.

For the nearly K\"ahler structure on $N=\SO(4)/\Z_2^2$, the nearly parallel $\G2$-structure on $SC(N)$ satisfying $d\g2_{sc}=\lambda *_{sc}\g2_{sc}$ is given by 
\begin{align}\label{eqn:sine cone}
    \varphi_{sc}&= \frac{32\sqrt{3}}{9\lambda^3}\sin^2(t) dt\wedge\omega + \frac{64}{27\lambda^3}\sin^3(t) (\sin(t)  \Omega_++\cos(t) \Omega_-)
\end{align}

\paragraph{The homogeneous nearly parallel $\G2$-structure on 
squashed S$^7$.}

In \cite{3-Sasakian_paper} the orthonormal basis for the invariant nearly parallel $\G2$-structure on squashed $S^7$ satisfying $d\g2=\frac{6}{\sqrt{5}}\psi$ is given by $\{F_i,i=1\dots 6,dt\}$ where 
\[
\left(\ns\ba{c} F_{2k-1}\y F_{2k}\ea\ns\right) =
\left(\ba{cc} -\frac{1}{\sqrt{5}} & -\frac{1}{\sqrt{5}}(2c_k+1)\yy 0 & -2s_k\ea\right)
\left(\ns\ba{c} e_{2k-1}\y e_{2k}\ea\ns\right)\]
where $c_k=\cos(t+\phi_k), s_k=\sin(t+\phi_k)$ and $\phi_k=(2k-2)\pi/3$ for $k=1,2,3$. The $\SU(3)$-structure $(\omega,\Omega)$ on the hypersurface is given by
\begin{align}
\begin{split}\label{eqn:S7-ortho}
     \omega&= F_1\wedge F_2+F_3\wedge F_4+F_5\wedge F_6,\\
    \Omega&= (F_1+iF_2)\wedge(F_3+iF_4)\wedge (F_5+iF_6).
\end{split}
   \end{align}

The nearly parallel $\G2$-structure is then given by
\begin{align*}
    \g2&=dt\wedge\omega+{\rm{Re}}\Omega.
\end{align*}
The singularities occur at $t=0, \pi/3$, and the singular isotropy groups are $K^i_{3,-1},K^j_{1,-1}\cong S^1\rtimes \Z_2$, respectively. 

\paragraph{The homogeneous nearly parallel $\G2$-structure on Berger.} 
This example motivates our construction of complete $\SO(4)$-invariant cohomogeneity-one
nearly parallel manifolds.  The nearly parallel $\G2$-structure on $B=SO(5)/SO(3)$ is described in detail in
\cite{Ball-Madnick,3-Sasakian_paper,SINGHAL-NHF}; hence, here we record only the 3-form $\g2$ in our notation. The $\G2$-structure below satisfies $d\g2=\frac{6}{\sqrt{5}}\psi$.

\begin{align}
\begin{split}\label{eq:hom_ng2}
\g2&=(\sin(t)e^{12}-\sin(t-2\pi/3)e^{34}-\sin(t+2\pi/3)e^{56})\wedge e^7\\
&+\frac{1}{20}(8\cos(t-2\pi/3)\cos(t)\cos(t+2\pi/3)(e^{135}+e^{246})-7(e^{135}-e^{246}))\\
&+\frac{1}{5}\sin(t)\sin(t-2\pi/3)
  \bigl(-2\cos(t+2\pi/3)(e^{136}-e^{245})-3(e^{136}+e^{245})\bigr)\\
&+\frac{1}{5}\sin(t)\sin(t+2\pi/3)
  \bigl(-2\cos(t-2\pi/3)(e^{145}-e^{236})-3(e^{145}+e^{236})\bigr)\\
&+\frac{1}{5}\sin(t-2\pi/3)\sin(t+2\pi/3)
  \bigl(-2\cos(t)(e^{146}-e^{235})+3(e^{146}+e^{235})\bigr).
\end{split}
\end{align}

The volume $\sqrt{\det g_\g2}$ is proportional to $\sin{3t}$, which vanishes at $t=0,\pi/3$, corresponding to the singular orbits.

\section{Smooth extension over the singular orbits}
\label{sec:smoothness}
In this section, we identify which solutions of \eqref{eqns:gen_ng2} extend to the singular orbits described in \Cref{section:coho-1action}. We apply the smoothness criterion of Eschenburg--Wang and refer the reader to \cite{Eschenburg-Wang, Haskins-Lorenzo} for more details. 

Let $\Sigma=G/K$ be a codimension-two singular orbit, let
$q=eK\in\Sigma$, and let $V\cong\mathbb R^2$ be its normal slice.
The group $K$ acts transitively on the unit circle in $V$, with
principal stabiliser $H$.  A tubular neighbourhood of $\Sigma$ is the
associated bundle
\begin{equation*}
 E=G\times_K V.
\end{equation*}
Choose an $\operatorname{Ad}(K)$-invariant complement
$\mathfrak g=\mathfrak k\oplus\mathfrak n$.  Restriction to the slice
identifies a $G$-invariant $3$-form on $E$ with a $K$-equivariant map
\begin{equation*}
 V\longrightarrow\Lambda^3(V\oplus\mathfrak n).
\end{equation*}
Let $W$ denote the space of equivariant maps from the unit circle to
$\Lambda^3(V\oplus\mathfrak n)$.  Evaluation at a fixed
$u_0\in S^1$ gives an isomorphism
    \begin{align*}
        \epsilon\colon W&\to \Lambda^3(V\oplus\n)^{\Z_2^2}\\
        w &\mapsto w(u_0).
    \end{align*}
Let $W_p\subset W$ be the subspace obtained by restricting homogeneous
equivariant polynomial maps of degree $p$ on $V$.  If the radial
restriction of an invariant form has Taylor series
$h_t\sim\sum_{p\ge0}h_pt^p$, then \cite[Lemma~1.1]{Eschenburg-Wang}
states that it extends smoothly across the zero section if and only if
\begin{equation}\label{eq:EW-smoothness}
 h_p\in\epsilon(W_p)\qquad\text{for every }p\ge0.
\end{equation}

See \Cref{section:coho-1action} for the cohomogeneity-one action of $\SO(4)$. On the double cover
$\Sp(1)\times\Sp(1)$ the singular isotropy group is of the form 
\begin{align*}
    C^u_{m,n}=\{(e^{u m\theta},e^{u n\theta}):\theta\in\R\},
\end{align*}
where $mn<0$ and $u\in {i,j}$.

If we denote by $K$ the singular isotropy group for the SO(4)-action, then
$K\cong S(O(2)O(1))\cong S^1\rtimes\mathbb Z_2$.  On
$\mathbb R^4=\langle E_1\rangle^\perp
=\operatorname{span}\{E_2,E_3,E_4,E_5\}$ its action is
\begin{align*}
   \left[\begin{array}{cccc}
         \cos r\theta&-\sin r\theta&0&0\\
         a\sin r\theta&a\cos r\theta&0&0\\
         0&0&\cos s\theta&-\sin s\theta\\
         0&0&a\sin s\theta&a\cos s\theta
    \end{array}\right]
\end{align*} 
for $\theta\in[0,\pi]$, $a\in\mathbb Z_2$, $r+s=m,r-s=-n$.

 The singular orbit is the homogeneous space 
 \[\Sigma=\SO(4)/K\cong \frac{\rm{SO}(3)\times\rm{SO}(3)}{\rm{SO}(2)\times\Z_2}\cong(S^2\times\mathbb S^3)/\Z_2.\] Here $V=T_qM/T_q\Sigma\cong \R^2$ is the normal space spanned by $v_1,v_2$. The space $T_q\Sigma\cong \mathfrak{n}$ is the orthogonal reductive complement of $\mathfrak{k}:=\rm{Lie}(K)$ in $\mathfrak{so}(4)\cong\mathfrak{su}(2)\oplus\mathfrak{su}(2)$.
Let $E_{ij}$ denote the anti-symmetric $5\times 5$ matrix with $1$ at the $(i,j)$ entry and $0$ elsewhere. Then $\mathfrak{k}$ as a sub-algebra of $\mathfrak{so}(4)_{E_1}$ is given by $\rm{Span}\{2E_{23}+E_{45}\}$ and 
$\n={\rm{Span}}\{u_1,u_2,u_3,u_4,u_5\}$ where 
\begin{align*}
    u_1&:= E_{23}-2E_{45}, \quad u_2:= E_{24}-E_{35},\quad u_3:=E_{24}+E_{35},
    u_4:=E_{25}+E_{34},\quad u_5:=E_{25}-E_{34}.
\end{align*}

The isotropy group $K$ acts transitively on the unit circle in $V$, and the stabilizer of a unit normal vector is $H$. Hence the kernel of the $K^0$-action on $V$ is $K^0\cap H$. For the effective $\SO(4)$-diagram, $H\cong\mathbb Z_2^2$ and $K^0\cap H\cong\mathbb Z_2$. Since the kernel here has order two, the normal-slice representation has weight $2$. 
The action of $K$ on the 7-dimensional space $V\oplus\n=\rm{Span}\{v_1,v_2,u_1,u_2,u_3,u_4,u_5\}$ is given by
\begin{align*}
   v_1&\mapsto \cos2\theta v_1-a\sin2\theta v_2, \ \
   v_2\mapsto  \sin2\theta v_1+a\cos2\theta v_2,\ \
   u_1\mapsto a u_1,\\
   u_2&\mapsto \cos m\theta u_2-a\sin m\theta u_4,\ \ u_4\mapsto \sin m\theta u_2+a\cos m\theta u_4, \\
    u_3&\mapsto \cos n\theta u_3+a\sin n\theta u_5,\ \  u_5\mapsto  -\sin n\theta u_3+a\cos n\theta u_5 
     \end{align*}
For the above $S^1$-action to preserve a positive $\G2$-form, these must satisfy a maximal-torus relation in the Lie group $\G2$ and thus 
\begin{align*}
    \pm2\pm m\pm n=0.
\end{align*}
Since $m,n$ are odd \cite[Lemma 7.2]{GWZ} and have opposite signs, up to exchanging the factors and changing both signs, there are two cases:
\begin{align*}
    (m,n)=(1,-1), \ \ \ (m,n)=(2k+1,1-2k),k\geq 1.
\end{align*}
These are precisely the two orbits that occur as the singular orbits for the family $P_k$, $k\geq 1$, in \cite{GWZ}. Note that this also covers the Berger space diagram, as the singular orbits there have weights $(3,-1)$ and $(1,-3)$. 

We now describe the parity conditions explicitly in the form of \eqref{eq:EW-smoothness} for the above singular isotropy groups. 

Let $V\cong\R^2$ denote the normal slice and let $\n$ be an
$\operatorname{Ad}(K)$-invariant complement of $\mathfrak k$ in $\fg$.
For $m\neq n$ along a fixed radial line in $V$ the change of basis from $\{dt,e^i,i=1\dots 6\}$ to $V$ is given by
\begin{align}
\begin{split}
     dt&=v^1, \quad e^1= \frac{2}{m-n}\left(\frac{mv^2}{t}-u^1\right), \quad e^2=\frac{2}{m-n}\left(-\frac{nv^2}{t}+u^1\right)\\
 e^3&= u^2, \quad e^4=u^3,\quad 
   e^5=u^4, \quad e^6=u^5.
\end{split}\label{eq:mn_adapted_coframe}
   \end{align}

The space $\Lambda^{3}(V\oplus \n)^{H}$ of $H$-invariant $3$-forms on $V\oplus \n$ is an $11$-dimensional space spanned by

\[v^{12}u^1, v^1u^{23},v^1u^{45},v^2u^{24},v^2u^{25},v^2u^{34},v^2u^{35},u^{124},u^{125},u^{134},u^{135}.\]

Every $\mathbb Z_2^2$-invariant $3$-form on $V\oplus\mathfrak n$
is represented by a $K_-$-equivariant polynomial on the normal unit
circle.  The identity component acts on
$\operatorname{span}\{v^1,v^2\}$ with weight $2$.  Hence, for
$h\in W$ and $(e^{i\theta},a)\in S^1\rtimes\mathbb Z_2$,
$K_-$-equivariance gives
\begin{align*}
    h(e^{i\theta})&=e^{i\theta}\cdot h(1).
\end{align*} 
Using this, one can compute the polynomial $h$ corresponding to any $\alpha \in \Lambda^3(V\oplus \n)^{\Z_2^2}$ by assuming $h(1)=\alpha$. The action of $K$ on $\Lambda^3(V\oplus\mathfrak{n})^H$ decomposes into 11 irreducible representations. If $K$ acts by $e^{il\theta}$ on  $V_l$,
\begin{align}\label{eqn:3-form-decompose-K}
    \Lambda^3(V\oplus\mathfrak{n})^H=3V_0\oplus 2V_2\oplus V_{2+m+n}\oplus V_{-2+m+n}\oplus V_{2+m-n}\oplus V_{-2+m-n} \oplus V_{m+n}\oplus V_{m-n},
\end{align} where $K$ acts
\begin{itemize}
    \item trivially on $v^{12}u^{1},u^{124},u^{135}$, 
    \item by weight $2$ on $v^2u^{24},v^2u^{35}$,
    \item  by weight $2+m+n$ on $v^1(u^{23}+u^{45})+v^2(u^{25}+u^{34})$, 
    \item by weight $-2+m+n$ on $v^1(u^{23}+u^{45})-v^2(u^{25}+u^{34})$,
    \item  by weight $2+m-n$ on
    $v^1(u^{23}-u^{45})-v^2(u^{25}-u^{34})$,
    \item  by weight $-2+m-n$ on
    $v^1(u^{23}-u^{45})+v^2(u^{25}-u^{34})$,
    \item by weight $m-n$ on $u^{125}-u^{134}$, and $m+n$ on $u^{125}+u^{134}$.
\end{itemize}
Substituting \eqref{eq:mn_adapted_coframe} into \eqref{eq:gen-phi}, we obtain the $\G2$-structure in terms of the new variables on $V\oplus \mathfrak{n}$, 
\begin{align}\label{eqn:phi_t0}
 \g2={}&\frac{4p_1}{(m-n)t}v^{12}u^1
 +p_2v^1u^{23}+p_3v^1u^{45}
 +\frac{2mq_1-2nq_5}{(m-n)t}v^2u^{24}
 +\frac{2mq_2-2nq_6}{(m-n)t}v^2u^{25}
 +\frac{2mq_3-2nq_7}{(m-n)t}v^2u^{34}\notag\\
 &+\frac{2mq_4-2nq_8}{(m-n)t}v^2u^{35}
 +\frac{2q_5-2q_1}{m-n}u^{124}
 +\frac{2q_6-2q_2}{m-n}u^{125}
 +\frac{2q_7-2q_3}{m-n}u^{134}
 +\frac{2q_8-2q_4}{m-n}u^{135}.
\end{align}
We can now determine the parity of the coefficients for $(m,n)=(1,-1)$ and $(m,n)=(2k+1,1-2k)$ and the conditions for extension of the $\G2$-structure to the singular orbit. Note that, since all the singular orbits $K^u_{m,n}$ are conjugate by the action of an element in the quaternionic group, the extension does not depend on $u$ but only on the weights $(m,n)$. 

We also record a more general version of the ODE result \cite[Theorem 4.7]{Haskins-Lorenzo} that we use to show the existence of a family of solutions of \eqref{eqns:gen_ng2} that extends to the singular orbit. See
\cite[Chapters 5 and 6]{Eschenburg-Wang} and
\cite[Section 2]{Rong2013} for further details.

\begin{thm}\label{thm:IVPsoln}
    Consider the initial value problem for $y\in\R^k$ satisfying
    \begin{equation}\label{eq:IVP}
\begin{aligned}
    \dot{y}(t)&=\frac{1}{t}M_{-1}(y)+M(t,y), \qquad y(0)=y_0,
\end{aligned}
    \end{equation}
    where $M_{-1}\colon \R^k\to\R^k$ is a smooth function in a neighbourhood of $y_0$ and $M:\R\times\R^k\to\R^k$ is a smooth function in a neighbourhood of $(0,y_0)$. Then there exists a unique solution of \eqref{eq:IVP} if 
    \begin{enumerate}
        \item[\textup{(i)}]  $M_{-1}(y_0)=0$,
        \item[\textup{(ii)}] $h\textup{Id}-{\rm d}_{y_0}M_{-1}$ is invertible for all $h\in\mathbb N$.
    \end{enumerate}
The solution depends continuously on $y_0$. Moreover, suppose that condition~\emph{(ii)} fails at exactly one integer $N\geq1$, such that
\[
  \dim\ker\bigl(N\Id-d_{y_0}M_{-1}\bigr)=1,
\]
and the recursion for the Taylor coefficients is solvable at order $N$. Then the initial-value problem admits a one-parameter family of smooth solutions. More precisely, if
$v$ spans $\ker(N\Id-d_{y_0}M_{-1})$, then the coefficient of $t^N$
is of the form
\[
  y_N=y_N^{0}+\nu v,\qquad \nu\in\mathbb R,
\]
while all the remaining Taylor coefficients are determined
recursively by $\nu$. For every fixed $\nu$, the corresponding
solution is defined on a sufficiently small interval $[0,T_\nu)$.
\end{thm}

Thus, in order to find a solution of the system of ODEs, one needs to find an initial point $y_0\in\R^k$ for which $M_{-1}(y_0)=0$.  Moreover, if for some $h\in\mathbb N$, $h\textup{Id}-{\rm d}_{y_0}M_{-1}$ is not invertible, then the solution also depends on the $h$th-order derivatives of $y_i$ at $t=0$.

\subsection{Singular orbit with slope (1,$-$1).} \label{section:sing_11}
We start with the case $(m,n)=(1,-1)$. From the decomposition in \eqref{eqn:3-form-decompose-K} for $m=1,n=-1$, we find that the coefficients of 
\begin{itemize}
    \item[] $v^{12}u^{1},u^{124},u^{135},v^1(u^{23}-u^{45})-v^2(u^{25}-u^{34}),v^1(u^{23}-u^{45})+v^2(u^{25}-u^{34}),u^{125}+u^{134}$ are even, and
    \item[] $v^2u^{24},v^2u^{35},v^1(u^{23}+u^{45})+v^2(u^{25}+u^{34}),v^1(u^{23}+u^{45})-v^2(u^{25}+u^{34}),u^{125}-u^{134}$ are odd.
\end{itemize}
Up to discrete symmetries, the parity conditions for extension imposed by Eschenburg--Wang on the functions $p_i,q_i$ in \eqref{eqn:phi_t0} are
\begin{enumerate}[label=(\roman*)]   
\item $p_1, p_2+p_3$ are odd, $p_2-p_3$ is even,
\item $q_1, q_4, q_5, q_8$ are even,
\item $q_2-q_3, q_6-q_7$ are odd, $q_2+q_3, q_6+q_7$ are even.
\end{enumerate}
These conditions, together with the well-definedness of the $\G2$-structure at $t=0$, impose the following form on the local solutions of \eqref{eqns:gen_ng2}. For real parameters $a,b,c$ and smooth functions $y_i(t)$, $i=1,\dots,8$, we have
\begin{align}\begin{split}\label{eqn:initial_11}
    q_1(t)&=a+t^2y_1(t),\ \
q_2(t)=b+t y_2(t)+t^2 y_3(t),\ \
q_3(t)=b-ty_2(t)+t^2y_3(t),\\
q_4(t)&=c+t^2 y_4(t),\ \
q_5(t)=-a+t^2 y_5(t),\ \
q_6(t)=-b+ty_6(t)+t^2y_7(t),\\
q_7(t)&=-b-t y_6(t)+t^2y_7(t),\ \
q_8(t)=-c+t^2y_8(t).
\end{split}
\end{align}
With respect to the $y_i$, equation \eqref{eqns:gen_ng2} takes the form 
\begin{align*}
    y'(t)=\frac{1}{t}M_{-1}(y)+M(t,y),
\end{align*} as shown explicitly below. Define
 \begin{align*}
A&=\frac{q_2+q_7}{t}
   =y_2-y_6+t(y_3+y_7),\\
B&=\frac{q_3+q_6}{t}
   =-y_2+y_6+t(y_3+y_7),\\
C&=q_4+q_5
   =c-a+t^2(y_4+y_5),\\
    Q&=-ABC=-\bigl(y_2-y_6+t(y_3+y_7)\bigr)
 \bigl(-y_2+y_6+t(y_3+y_7)\bigr)
 \bigl(c-a+t^2(y_4+y_5)\bigr),\\
    N_i& =\alpha_i (2\alpha_i-\4)+2\beta_2, \ i=\{1,4,6,7\},\\
    N_i& = \alpha_i (2\alpha_i-\4)+2\beta_1, \ i=\{2,3,5,8\}.
\end{align*}
Direct substitution of \eqref{eqn:initial_11} shows that
$$N_1,N_4,N_5,N_8=O(t^2),
\qquad
N_2,N_3,N_6,N_7=O(t).$$  For 
    \begin{align*}
       R_i &= \frac{N_i}{t^2},  \ i=\{1,4,5,8\}, \ \ \ 
    R_i = \frac{N_i}{t},  \ i=\{2,3,6,7\},
\end{align*}
the equations \eqref{eqns:gen_ng2} become for
\begin{align*}
    P_1={}&\sqrt{-\frac{AB}{\lambda C}},\ \ 
P_2=-\sqrt{-\frac{AC}{\lambda B}},\ \
P_3=\sqrt{-\frac{BC}{\lambda A}},
\end{align*}
$$ 
\begin{aligned}
t y_1'
={}&-2y_1+\frac{\lambda R_1}{2\sqrt{Q} q_8},
\\
t y_2'
={}&-y_2+\frac12(P_3-P_2)
-\frac{\lambda R_2}{4\sqrt{Q} q_7}
+\frac{\lambda R_3}{4\sqrt{Q} q_6},
\\
t y_3'
={}&-2y_3+
\frac{1}{2t}\left(
P_2+P_3
-\frac{\lambda R_2}{2\sqrt{Q} q_7}
-\frac{\lambda R_3}{2\sqrt{Q} q_6}
\right),
\\
t y_4'
={}&-2y_4-P_1+\frac{\lambda R_4}{2\sqrt{Q} q_5},
\end{aligned}
$$
$$
\begin{aligned}
 t y_5'
={}&-2y_5+P_1-\frac{\lambda R_5}{2\sqrt{Q} q_4},
\\
t y_6'
={}&-y_6+\frac12(P_3-P_2)
+\frac{\lambda R_6}{4\sqrt{Q} q_3}
-\frac{\lambda R_7}{4\sqrt{Q} q_2},
\\
t y_7'
={}&-2y_7-\frac{1}{2t}\left(
P_2+P_3
-\frac{\lambda R_6}{2\sqrt{Q} q_3}
-\frac{\lambda R_7}{2\sqrt{Q} q_2}
\right),
\\
t y_8'
={}&-2y_8-\frac{\lambda R_8}{2\sqrt{Q} q1}.
\end{aligned}
$$

Note that the coefficient of $1/t$ in $y_3'(t)$ and $y_7'(t)$ vanishes at $t=0$, so they still have the form required by Theorem \ref{thm:IVPsoln}.
    
  Let $y_0=\{r_1,r_2,\dots,r_8\}$ be the initial condition for which $M_{-1}(y_0)=0$ in Theorem~\ref{thm:IVPsoln}. 

When $c>a$ and $ac-b^2>0$ the smoothness conditions determine the initial values $y_i(0)=r_i$ where $\mu=\pm 1$ 
\begin{align}\label{eqn:y_intial_11}
r_1&=\frac{\lambda^2(5a^2-3ac-2b^2)}{8(c-a)}
+\frac{\mu a}{2}\sqrt{\frac{\lambda(c-a)}{ac-b^2}}+\frac{(1-\mu)a b^2\lambda^2}{4(ac-b^2)},\notag \\
r_2&=\mu\lambda\sqrt{ac-b^2}-\sqrt{\frac{c-a}{\lambda}},\notag \\
r_3&=\frac{b(c-a)}{4\lambda (ac-b^2)}
+\frac{b\lambda^2(ac-5b^2)}{8 (ac-b^2)}
-\frac{b\sqrt{\lambda (c-a)}}{2\sqrt{ac-b^2}}, \notag\\
r_4&=-\frac{\lambda^2(5c^2-3ac-2b^2)}{8(c-a)}
+\frac{\mu(ac+c^2-2b^2)}{2(c-a)}
\sqrt{\frac{\lambda(c-a)}{ac-b^2}}
+\frac{(1-\mu)c b^2\lambda^2}{4(ac-b^2)},\\
r_5&=-\frac{\lambda^2(5a^2-3ac-2b^2)}{8(c-a)}
-\frac{\mu(a^2+ac-2b^2)}{2(c-a)}
\sqrt{\frac{\lambda(c-a)}{ac-b^2}}+\frac{(\mu-1)a b^2\lambda^2}{4(ac-b^2)},\notag \\
r_6&=-\mu\lambda\sqrt{ac-b^2}-\sqrt{\frac{c-a}{\lambda}},\notag \\
r_7&=-\frac{b(c-a)}{4\lambda (ac-b^2)}
-\frac{b\lambda^2(ac-5b^2)}{8 (ac-b^2)}
-\frac{b\sqrt{\lambda (c-a)}}{2\sqrt{ac-b^2}},\notag \\
r_8&=\frac{\lambda^2(5c^2-3ac-2b^2)}{8(c-a)}
+\frac{\mu c}{2}\sqrt{\frac{\lambda(c-a)}{ac-b^2}}+\frac{(\mu-1)c b^2\lambda^2}{4(ac-b^2)}. \notag
\end{align}
One can compute that for all natural numbers $n$
\begin{align}\label{eqn:det_11}
\det\left(n{\rm{Id}}-\d_{y_0}M_{-1}\right)
=
n^3(n+1)(n+2)^4.
\end{align}

Theorem \ref{thm:IVPsoln} therefore gives two three-parameter families,
$\xi_{a,b,c}$ for $\mu=1$, and $\widetilde\xi_{a,b,c}$ for
$\mu=-1$. Each solution is uniquely determined by its initial data
and depends continuously on $a,b,c$.
The Taylor coefficients of the  solution are
therefore determined recursively and uniquely by the parameters
$a,b,c$. Namely
\begin{itemize}
\item $\xi_{a,b,c}(t)$ is the branch for which $\mu=1$, $c>a$, $ac-b^2>0$;
\item $\tilde\xi_{a,b,c}(t)$ is the branch for which $\mu=-1$, $c>a$, $ac-b^2>0$.
\end{itemize}
Both choices define three-parameter families of local solutions which
extend smoothly across the $(1,-1)$ singular orbit.  The homogeneous
nearly parallel structure on squashed $S^7$ lies in the
$\xi_{a,b,c}$ family, with
\[ (a,b,c)=(\sqrt5/25,\>-2\sqrt5/25,\>19\sqrt5/25).\]

When $c-a<0, ac-b^2<0$, the metric is not positive definite; hence, we disregard it. The induced metric on the singular orbit $\operatorname{Span}\{u_i,i=1,\dots,5\}$ is given by
\begin{align}\label{eqn:metric-11}
    \frac{1}{\rho}\begin{pmatrix}
    \rho^{3}&0&0&0&0\\
    0&a&b&0&0\\
    0&b&c&0&0\\
    0&0&0&a&b\\
    0&0&0&b&c
\end{pmatrix}, \qquad \rho=\sqrt{\frac{ac-b^2}{c-a}}.
\end{align}
The parameters $a$ and $c$ control the sizes of the two
distinguished $2$-dimensional directions in the singular orbit,
while $b$ is an off-diagonal mixing parameter.  The parameter $b$ can be regarded as a
\textit{symmetry-breaking parameter}, as its vanishing indicates the presence of an enhanced
$\U(1)$-symmetry:

\begin{thm}\label{thm:complete-11} For $a,b,c\in\R$, $c-a>0$, $ac-b^2>0$, there are two three-parameter families of solutions $\xi_{a,b,c}$ and $\tilde{\xi}_{a,b,c}$ of \eqref{eqns:gen_ng2} that extend smoothly to the singular orbit with slope $(1,-1)$ with initial conditions
\begin{align*}
    q_1(0)=a,  \ q_2(0)=q_3(0)=-q_6(0)=-q_7(0)=b, \ q_4(0)=c, \ q_5(0)=-a, \ q_8(0)=-c.
\end{align*}
If $b=0$, the solutions have an additional $\U(1)$-symmetry given by $q_2=-q_3, q_7=-q_6$. 
\end{thm}

\begin{proof} Consider the circle action
\[
 e^3+i e^5\longmapsto e^{i\theta}(e^3+i e^5),\qquad
 e^4+i e^6\longmapsto e^{-i\theta}(e^4+i e^6),
\]
with $e^1,e^2$, and $dt$ fixed.
It follows directly from the expression for $\varphi$ that its
fixed-point set is
\[
 \mathcal P_{\U(1)}
 =
 \left\{
 p_2+p_3=0,\quad q_2+q_3=0,\quad q_6+q_7=0
 \right\}.
\]
For the smooth $(1,-1)$ initial data,
\[
 q_2(0)+q_3(0)=2b,
 \qquad
 q_6(0)+q_7(0)=-2b.
\]
Thus the restriction of $\varphi$ to the singular orbit is fixed by the
circle precisely when $b=0$.

When $b=0$, the initial values in
\eqref{eqn:y_intial_11} also satisfy
\[
 r_2+r_3=0,\qquad r_6+r_7=0.
\]
The algebraic nearly half-flat relations then give $p_2+p_3=0$.
Hence the full regularized initial data belong to
$\mathcal P_{\U(1)}$. Since the nearly parallel equations are equivariant
under the circle action, $\mathcal P_{\U(1)}$ is invariant under the
associated vector field. Uniqueness of the singular initial-value
problem therefore implies that the solutions $\xi_{a,0,c}$ and
$\widetilde\xi_{a,0,c}$ remain in $\mathcal P_{\U(1)}$ throughout their
maximal intervals of existence.

If $b\neq0$, the restriction to the singular orbit does not belong to
$\mathcal P_{\U(1)}$, so the corresponding solution is not invariant
under this circle.
\end{proof}

\subsection{Singular orbit with slope (3,$-$1)}\label{section:sing_t_0}

Before tackling the general $(2k+1,1-2k)$ case, we first treat $k=1$, namely the $(3,-1)$ case. This case is important both for the purposes of this paper, since these are the weights of the singular orbits for the Berger space, and computationally, since it gives better insight into the ODE system for the general $k$ case. 
The Eschenburg--Wang criterion here requires the coefficients of
$v^{12}u^1$, $v^1(u^{25}+u^{34})-v^2(u^{23}+u^{45})$,
$u^{124}$, $u^{135}$, $u^{123}-u^{145}$ and
$v^1(u^{25}+u^{34})+v^2(u^{23}+u^{45})$ to be even.  The
coefficients of $v^2u^{24}$, $v^2u^{35}$, $u^{123}+u^{145}$ and
$v^1(u^{25}-u^{34})\pm v^2(u^{23}-u^{45})$ must be odd.  Therefore
the coefficient functions $p_i,q_i$ satisfy the following parity
conditions:
\begin{enumerate}[label=\textup{(\roman*)}]
    \item $p_1$ is odd
    \item $p_2+p_3$ is even, $p_2-p_3$ is odd,
    \item $q_1,q_4, q_5,q_8$ are even,
    \item $q_2-q_3$ is even, $q_2+q_3$ is odd,
    \item $q_6-q_7$ is even, $q_6+q_7$ is odd.
 \end{enumerate}

Imposing the above parity relations and the well-definedness of $\g2$ at $t=0$ we can express the coefficients $q_i,i=1\dots 8$ as the following Taylor series expansions around $t=0$ for some real constants $a,b,c$ and smooth functions $y_i(t),i=1\dots 8$,

\begin{align}\label{eqns:q_in_y_t0}
\begin{split}
     q_1(t)&=a+t^2y_1(t),\ \ q_2(t)=ty_2(t),\ \  q_3(t)=ty_3(t),\ \ q_4(t)=c+t^2y_4(t),\\
    q_5(t)&=-3a+t^2y_5(t),\ \
    q_6(t)=ty_6(t),\ \
    q_7(t)=ty_7(t),\ \
    q_8(t)=-3c+t^2y_8(t).
\end{split}
    \end{align}
Let \begin{align*}
    Q&=-(y_2+y_7)(y_3+y_6)(c-3a+t^2(y_4+y_5))
\end{align*}
    then we can write the differential equations for the $\G2$-structure to be nearly parallel \eqref{eqns:gen_ng2} as follows
    \begin{align}
        \label{eqns:yprime}
        y_1'(t)=&\frac{1}{t}\left(-2y_1- \frac{a\lambda^{5/2}}{2\sqrt{Q}}(a(3y_4+y_8)-c(3y_1+y_5)-y_2y_7-y_3y_6-6y_2y_3)\right) \notag \\
        &-\frac{\lambda^{5/2}}{2\sqrt{Q}}t(a(y_1y_8-y_4y_5)+y_1(a(3y_4+y_8)-c(3y_1+y_5)-y_2y_7-y_3y_6)+2y_2y_3y_5 \notag\\
        & \ \ \ \ \ \ \ \ \ \ \ \ \ +t^2(y_1^2y_8-y_1y_4y_5)) \notag \\
        y_2'(t)=&\frac{1}{t}\left(-y_2+\sqrt{\frac{-(y_3+y_6)(c-3a+t^2(y_4+y_5)}{\lambda(y_2+y_7)}}+\frac{\lambda^{5/2} ac }{\sqrt{Q}}(3y_2+y_6)\right) \notag\\
        &+\frac{\lambda^{5/2}}{2\sqrt{Q}}t((y_2^2y_7-y_2y_3y_6-a(y_2y_8-3y_2y_4-2y_4y_6)-c(y_2y_5-3y_1y_2-2y_1y_6))\notag \\
        & \ \ \ \ \ \ \ \ \ \ \ \ \ -t^2(y_1y_2y_8+y_2y_4y_5-2y_1y_4y_6)) \notag\\
        y_3'(t)=&\frac{1}{t}\left(-y_3+\sqrt{\frac{-(y_2+y_7)(c-3a+t^2(y_4+y_5)}{\lambda(y_3+y_6)}}+\frac{\lambda^{5/2} ac }{\sqrt{Q}}(3y_3+y_7)\right) \notag\\
        &+\frac{\lambda^{5/2}}{2\sqrt{Q}}t((y_3^2y_6-y_2y_3y_7-a(y_3y_8-3y_3y_4-2y_4y_7)-c(y_3y_5-3y_1y_3-2y_1y_7))\\
        & \ \ \ \ \ \ \ \ \ \ \ \ \ -t^2(y_1y_3y_8+y_3y_4y_5-2y_1y_4y_7)) \notag\\
        y_4'(t)=&\frac{1}{t}\left(-2y_4+\sqrt{\frac{-(y_2+y_7)(y_3+y_6)}{\lambda(c-3a+t^2(y_4+y_5))}}+\frac{c\lambda^{5/2}}{2\sqrt{Q}}(a(3y_4+y_8)-c(3y_1+y_5)+y_2y_7+y_3y_6+6y_2y_3)\right) \notag\\
        &-\frac{\lambda^{5/2}}{2\sqrt{Q}}t(c(y_4y_5-y_1y_8)+y_4(-a(3y_4+y_8)+c(3y_1+y_5)-y_2y_7-y_3y_6)+2y_2y_3y_8 \notag\\
         & \hspace{1.25cm}-t^2(y_1y_4y_8-y_4^2y_5)) \notag\\
        y_5'(t)=&\frac{1}{t}\left(-2y_5-\sqrt{\frac{-(y_2+y_7)(y_3+y_6)}{\lambda(c-3a+t^2(y_4+y_5))}}+\frac{a\lambda^{5/2}}{2\sqrt{Q}}(3a(3y_4+y_8)-3c(3y_1+y_5)+3y_2y_7+3y_3y_6+2y_6y_7)\right) \notag\\
        &+\frac{\lambda^{5/2}}{2\sqrt{Q}}t(-3a(y_4y_5-y_1y_8)+y_5(-a(3y_4+y_8)+c(3y_1+y_5)-y_2y_7-y_3y_6)+2y_1y_6y_7 \notag\\
         & \hspace{1.25cm}-t^2(y_1y_5y_8-y_4y_5^2)) \notag\\
        y_6'(t)=&\frac{1}{t}\left(-y_6-\sqrt{\frac{-(y_2+y_7)(c-3a+t^2(y_4+y_5)}{\lambda(y_3+y_6)}}-\frac{3\lambda^{5/2} ac }{\sqrt{Q}}(y_6+3y_2)\right) \notag\\
        &-\frac{\lambda^{5/2}}{2\sqrt{Q}}t((y_3y_6^2-y_2y_6y_7-a(y_6y_8-3y_4y_6+6y_2y_8)-c(y_5y_6-3y_1y_6+6y_2y_5)) \notag\\
         & \ \ \ \ \ \ \ \ \ \ \ \ \ -t^2(y_1y_6y_8+y_4y_5y_6-2y_2y_5y_8)) \notag\\
          y_7'(t)=&\frac{1}{t}\left(-y_7-\sqrt{\frac{-(y_3+y_6)(c-3a+t^2(y_4+y_5)}{\lambda(y_2+y_7)}}-\frac{3\lambda^{5/2} ac }{\sqrt{Q}}(3y_3+y_7)\right) \notag\\
        &-\frac{\lambda^{5/2}}{2\sqrt{Q}}t((y_2y_7^2-y_3y_6y_7-a(y_7y_8-3y_4y_7+6y_3y_8)-c(y_5y_7-3y_1y_7+6y_3y_5)) \notag\\
        & \ \ \ \ \ \ \ \ \ \ \ \ \ -t^2(y_1y_7y_8+y_4y_5y_7-2y_3y_5y_8)) \notag\\
        y_8'(t)=&\frac{1}{t}\left(-2y_8+\frac{c\lambda^{5/2}} {2\sqrt{Q}}(-3a(3y_4+y_8)+3c(3y_1+y_5)+3(y_2y_7+y_3y_6)+2y_6y_7)\right) \notag\\
        &+\frac{\lambda^{5/2}}{2\sqrt{Q}}t(-3c(y_1y_8-y_4y_5)+y_8(a(3y_4+y_8)-c(3y_1+y_5)-y_2y_7-y_3y_6)+2y_4y_6y_7 \notag \\
        & \ \ \ \ \ \ \ \ \ \ \ \ \ +t^2(y_1y_8^2-y_4y_5y_8))  \notag     
\end{align}

Let $y_0=\{b_1,b_2,\dots,b_8\}$ be the initial condition for which $M_{-1}(y_0)=0$ in Theorem~\ref{thm:IVPsoln}. 

For $a>0,c<0$ and $\mu\in\{\pm 1\}$, we obtain
\begin{align}
\begin{split}\label{eqn:initialcond_0}
    b_1&=-\frac{3\lambda^2a \ (5a-c)}{8 (3a-c)}+\frac{\mu}{2c}\sqrt{-\lambda ac(3a-c)},\\
    b_2=b_3&=\lambda\mu\sqrt{-ac}+\sqrt{\frac{3a-c}{\lambda}},\\
    b_4&=-\frac{\lambda^2c \ (9a-5c)}{8(3a-c)}+\frac{\mu}{2a}\sqrt{-\lambda ac(3a-c)}-\mu\sqrt{\frac{-\lambda ac}{3a-c}},\\
    b_5&=\frac{9\lambda^2a \ (5a-c)}{8(3a-c)}-\frac{\mu}{2c}\sqrt{-\lambda ac(3a-c)}+\mu\sqrt{\frac{-\lambda ac}{3a-c}},\\
     b_6=b_7&=-3\lambda\mu\sqrt{-ac}-\sqrt{\frac{3a-c}{\lambda}},\\
      b_8&=\frac{3\lambda^2c \ (9a-5c)}{8(3a-c)}-\frac{\mu}{2a}\sqrt{-\lambda ac(3a-c)}.
      \end{split}
\end{align}

For $ac>0, 3a-c>0$ and $\mu\in\{\pm 1\}$ we have
\begin{align}
\begin{split} \label{eqn:iniital_ac_pos}
    b_1&=\frac{\lambda^2a \ (9a-5c)}{8 (3a-c)}-\frac{\mu}{2c}\sqrt{\lambda ac(3a-c)},\\
    b_2=b_3&=-\lambda\mu\sqrt{ac}+\sqrt{\frac{3a-c}{\lambda}},\\
    b_4&=\frac{3\lambda^2c \ (5a-c)}{8(3a-c)}-\frac{\mu}{2a}\sqrt{\lambda ac(3a-c)}-\mu\sqrt{\frac{\lambda ac}{3a-c}},\\
    b_5&=-\frac{3\lambda^2a \ (9a-5c)}{8(3a-c)}+\frac{\mu}{2c}\sqrt{\lambda ac(3a-c)}+\mu\sqrt{\frac{\lambda ac}{3a-c}},\\
     b_6=b_7&=3\lambda\mu\sqrt{ac}-\sqrt{\frac{3a-c}{\lambda}},\\
      b_8&=-\frac{9\lambda^2c \ (5a-c)}{8(3a-c)}+\frac{\mu}{2a}\sqrt{\lambda ac(3a-c)}.
      \end{split}
\end{align}

We now need to check whether the solution with the above initial conditions is unique. For this, we need to compute the determinant $\det(n {\rm{Id}}-d_{y_0}M_{-1})$ in both cases ($ac>0$ and $ac<0$). 

We compute that if $\mu=1$ 
\begin{align*}
    \det( n \rm{Id}-d_{y_0}M_{-1})&=n(n+1)(n+2)^4(n^2+2n+5),
\end{align*} 
hence $n\operatorname{Id}-d_{y_0}M_{-1}$ is invertible for every
$n\in\mathbb N$.  The singular IVP therefore has a unique solution
for each admissible pair $(a,c)$, giving a two-parameter family
$\Psi_{a,c}$ which extends smoothly across $t=0$ and satisfies
$3a-c>0$.

When $\mu=-1$ we get 
\begin{align}\label{eqn:det_M}
    \det( n \rm{Id}-d_{y_0}M_{-1})&=(n-1)(n+1)(n+2)^5(n+3),
\end{align} 
which vanishes for $n=1$.  Thus Theorem~\ref{thm:IVPsoln} does not give
uniqueness in this case.  The kernel is one-dimensional, so the solution is
determined by one additional first-order datum.  We take
$\nu\coloneqq y_3'(0)$ and denote the resulting three-parameter family
by $\eta_{a,c,\nu}$.

The homogeneous Berger structure and the squashed-sphere structure
both extend smoothly across this singular orbit.  In both cases
$\mu=-1$, so the zeroth-order data alone do not determine the local
solution uniquely.

For the Berger $a=\frac{\sqrt{5}}{20}, c=-\frac{3\sqrt{5}}{100}$, and $\nu=\frac{\sqrt{5}}{50}$ with $\lambda=\frac{6}{\sqrt{5}}$.

For the homogeneous nearly parallel $\G2$-structure on the squashed $S^7$, one of the singular ends has the same weight as the Berger metric, so we can compute the parameters $a,c,\nu$. For the squashed $S^7$, the parameters $a=\frac{\sqrt{5}}{25}, c=-\frac{3\sqrt{5}}{5}$, and $\nu=\frac{\sqrt{5}}{50}$ with $\lambda=\frac{6}{\sqrt{5}}$. 

\begin{rem}
    The square root is real only if
    $Q=-(y_2+y_7)(y_3+y_6)(c-3a+t^2(y_4+y_5))>0$.
    At $t=0$ the identities $y_2=y_3$ and $y_6=y_7$ reduce this to
    $-(y_2(0)+y_6(0))^2(c-3a)>0$, and therefore $3a-c>0$, as assumed.
\end{rem}
Up to finite quotients, the singular orbit is an $S^1$-bundle over
$S^2\times S^2$ with weights $3$ and $-1$, respectively. Thus we have
a fibration
\begin{align*}
    S^1 \to T^{3,-1} \to S^2\times S^2,
\end{align*}
and since $\gcd(3,-1)=1$, $T^{3,-1}\cong S^2\times S^3$. 

For both families of solutions, the metric induced from the $\G2$-metric on the singular orbit at $t=0$ is given by
\begin{align*}
    g_{\Sigma_-}&=\frac{1}{\rho} \operatorname{diag}\left(25\rho^3,a,-c,a,-c\right), \qquad \rho=\sqrt{\frac{\lambda|ac|}{3a-c}}.
    \end{align*}
Thus we get a Riemannian metric only when $a>0,c<0$. Accordingly, in this article we focus on this case and assume $a>0,c<0$. 
If we denote by $\Sigma^1_-,\Sigma^2_-$ the planes spanned by $\{u^1,u^2,u^4\}$ and $\{u^1,u^3,u^5\}$, respectively, then 
\begin{align*}
\sqrt{\det(g_{T^{3,-1}}|_{\Sigma^1_-})}&=5a, \qquad \sqrt{\det(g_{T^{3,-1}}|_{\Sigma^2_-})}= -5c.
\end{align*}
Thus $a$ and $c$ are intrinsic singular-orbit parameters: they determine the relative sizes of the two distinguished $2$-plane distributions inside $T\Sigma_-$.  The $\G2$-structure \eqref{eqn:phi_t0} at the singular orbit for both families is given by 
\begin{align*}
    \g2|_{\Sigma_-}=& -2\sqrt{\frac{-\lambda ac}{3a-c}}v^{12}u^1+\sqrt{\frac{3a-c}{\lambda}} \ (v^1u^{23}+v^1u^{45}+v^2u^{25}+v^2u^{34})
    -2a u^{124}-2c u^{135}.
\end{align*}
Thus, for both families, the $\G2$-structure on the singular orbit coincides. 

Around $t=0$, for $a>0,c<0$ and $\mu=1$ in \eqref{eqn:initialcond_0} we get the 2-parameter family of solutions $\Psi_{a,c}=\{q_1(t),\dots,q_8(t)\}$, where
\begin{align*}
q_1(t)&=a+t^2\left(-\frac{3\lambda^2a \ (5a-c)}{8 (3a-c)}+\frac{\sqrt{-\lambda ac (3a-c)}}{2c}\right)+{\mathcal{O}}(t^4),\\
q_2(t)=q_3(t)&= t\left(\lambda\sqrt{-ac}+\sqrt{\frac{3a-c}{\lambda}}\right)+{\mathcal{O}}(t^3),\\
q_4(t)&=c+t^2 \left(-\frac{\lambda^2c \ (9a-5c)}{8(3a-c)}+\frac{\sqrt{-\lambda ac(3a-c)}}{2a}-\sqrt{\frac{-\lambda ac}{3a-c}}\right)+{\mathcal{O}}(t^4),\\
q_5(t)&=-3a+t^2 \left(\frac{9\lambda^2a \ (5a-c)}{8(3a-c)}-\frac{\sqrt{-\lambda ac(3a-c)}}{2c}+\sqrt{\frac{-\lambda ac}{3a-c}}\right)+{\mathcal{O}}(t^4),\\
q_6(t)=q_7(t)&=  t\left(-3\lambda\sqrt{-ac}-\sqrt{\frac{3a-c}{\lambda}}\right)+{\mathcal{O}}(t^3), \\
q_8(t)&=-3c+t^2 \left(\frac{3\lambda^2c \ (9a-5c)}{8(3a-c)}-\frac{\sqrt{-\lambda ac(3a-c)}}{2a}\right)+{\mathcal{O}}(t^4).
\end{align*}

\begin{rem}
    For all permissible values of $a,c$ tested numerically, the solution $\Psi_{a,c}$ becomes degenerate after a short time.  We therefore do not expect complete solutions in this family.
\end{rem}

For $\mu=-1$ and $a>0,c<0$, we obtain a 3-parameter family $\eta_{a,c,\nu}$. The Taylor series expansion of $\eta_{a,c,\nu}=\{q_1(t),\dots,q_8(t)\}$ has the following form around $t=0$:
\begin{align*}
    q_1(t)&=a+t^2\left(-\frac{3\lambda^2a \ (5a-c)}{8 (3a-c)}-\frac{\sqrt{-\lambda ac(3a-c)}}{2c}\right)+{\mathcal{O}}(t^4),\\
q_2(t)&=t\left(-\lambda\sqrt{-ac}+\sqrt{\frac{3a-c}{\lambda}}\right)-t^2\nu+{\mathcal{O}}(t^3),\\
q_3(t)&= t\left(-\lambda\sqrt{-ac}+\sqrt{\frac{3a-c}{\lambda}}\right)+t^2\nu+{\mathcal{O}}(t^3),\\
q_4(t)&=c+t^2 \left(-\frac{\lambda^2c \ (9a-5c)}{8(3a-c)}-\frac{\sqrt{-\lambda ac(3a-c)}}{2a}+\sqrt{\frac{-\lambda ac}{3a-c}}\right)+{\mathcal{O}}(t^4),
\end{align*}
\vspace{-10pt}
\begin{align}
  q_5(t)&=-3a+t^2 \left(\frac{9\lambda^2a \ (5a-c)}{8(3a-c)}+\frac{\sqrt{-\lambda ac(3a-c)}}{2c}-\sqrt{\frac{-\lambda ac}{3a-c}}\right)+{\mathcal{O}}(t^4), \label{init_eta_ac} \\
q_6(t)&= t\left(3\lambda\sqrt{-ac}-\sqrt{\frac{3a-c}{\lambda}}\right)-t^2 \left(\frac{\sqrt{3a-c}+3\lambda^{3/2}\sqrt{-ac}}{\sqrt{3a-c}-\lambda^{3/2}\sqrt{-ac}}\right)\nu +{\mathcal{O}}(t^3),\notag \\
q_7(t)&=  t\left(3\lambda\sqrt{-ac}-\sqrt{\frac{3a-c}{\lambda}}\right)+t^2 \left(\frac{\sqrt{3a-c}+3\lambda^{3/2}\sqrt{-ac}}{\sqrt{3a-c}-\lambda^{3/2}\sqrt{-ac}}\right)\nu+{\mathcal{O}}(t^3), \notag \\
q_8(t)&=-3c+t^2 \left(\frac{3\lambda^2c \ (9a-5c)}{8(3a-c)}+\frac{\sqrt{-\lambda ac(3a-c)}}{2a}\right)+{\mathcal{O}}(t^4). \notag
\end{align}

\noindent
Sending $\nu$ to $-\nu$ in $\eta_{a,c,\nu}$ swaps two isotropy summands, namely $q_2\leftrightarrow q_3$ and $q_6 \leftrightarrow q_7$. Thus we can always assume $\nu\geq 0$. If $\nu=0$, we have $q_2=q_3,q_6=q_7$ up to order 1. Since \eqref{eqn:det_M} has no other positive eigenvalues, invariance of \eqref{eqns:gen_ng2} under this involution implies that $q_2=q_3,q_6=q_7$ at all higher orders and hence throughout its maximal interval of existence. Thus, when $\nu=0$, the family $\eta_{a,c,0}$ has an extra $\U(1)$-invariance.  

\begin{thm}\label{thm:complete-31} For $a>0$, $c<0$, and $\nu\in \R$ the family
$\eta_{a,c,\nu}$ is a real-analytic three-parameter family of solutions
of \eqref{eqns:gen_ng2}
extending smoothly across the singular orbit with weights $(3,-1)$. The initial conditions for $\eta_{a,c,\nu}$ are
\begin{align*}
    q_1(0)&=a, \ q_2(0)=q_3(0)=q_6(0)=q_7(0)=0, \ q_4(0)=c, \ q_5(0)=-3a, \ q_8(0)=-3c, \ q_3''(0)=-q_2''(0)=2\nu.
\end{align*}
The parameter $\nu$ is a symmetry-breaking parameter. If $\nu=0$, the solution has an additional $\U(1)$ symmetry and $q_2=q_3,q_6=q_7$. 
\end{thm}

\begin{rem}\label{inv_ca}
    The system \eqref{eqns:gen_ng2} is invariant under the homothety $$t\mapsto st,\quad \lambda\mapsto\lambda/s,\quad q\mapsto s^3q,\quad p\mapsto s^2p,\quad (a,c)\mapsto s^3(a,c),\quad \nu\mapsto s\nu .$$ Fixing $\lambda$ removes this freedom; the invariant shape coordinates are $\rho=c/a$, $\sigma=\nu/a^{1/3}$.
\end{rem}

\subsection{Singular orbit with weights $(2k+1,1-2k)$}\label{section:sing_k}
For $(m,n)=(2k+1,1-2k)$, $k\in\mathbb N$ we can express the coefficients $q_i$, $i=1\dots 8$ as the following Taylor series expansions around $t=0$ for some real constants $a,c$ and smooth functions $y_i(t)$, $i=1\dots 8$,
\begin{align}
\begin{split}\label{eq:q_k}
  q_1(t)&=(2k-1)a+t^2y_1(t),\ \
    q_2(t)=ty_2(t),\ \
    q_3(t)=ty_3(t),\ \
    q_4(t)=(2k-1)c+t^2y_4(t),\\
    q_5(t)&=-(2k+1)a+t^2y_5(t),\ \
    q_6(t)=ty_6(t),\ \
    q_7(t)=ty_7(t),\ \
    q_8(t)=-(2k+1)c+t^2y_8(t).     
\end{split}
\end{align}

We can carry out computations similar to those for the $(3,-1)$ case and obtain families of solutions that extend over the singular orbit with weights $(2k+1,1-2k)$. Again, the system of ODEs in \eqref{eqns:gen_ng2} reduces to the form in Theorem \ref{thm:IVPsoln} for $y_i$. We find that the solutions in \eqref{eq:q_k} extend to the singular orbit $(2k+1,1-2k)$ if $y_i(0)=b_i$ for $i=1,\dots,8$, where, for $a>0$, $c<0$, $\mu\in\{\pm1\}$, and $\Delta_k=(2k+1)a-(2k-1)c$,
\begin{align*}
b_1
&=-\frac{(2k-1)\lambda^2a
\bigl(5(2k+1)a-3(2k-1)c\bigr)}
{8\Delta_k}
+\frac{\mu}{2c}\sqrt{-\lambda ac\Delta_k},\\
b_2=b_3
&=(2k-1)\lambda\mu\sqrt{-ac}
+\sqrt{\frac{\Delta_k}{\lambda}},\\
b_4
&=-\frac{(2k-1)\lambda^2c
\bigl(3(2k+1)a-5(2k-1)c\bigr)}
{8\Delta_k}
+\frac{\mu}{2a}\sqrt{-\lambda ac\Delta_k}
-\mu\sqrt{\frac{-\lambda ac}{\Delta_k}},\\
b_5
&=\frac{(2k+1)\lambda^2a
\bigl(5(2k+1)a-3(2k-1)c\bigr)}
{8\Delta_k}
-\frac{\mu}{2c}\sqrt{-\lambda ac\Delta_k}
+\mu\sqrt{\frac{-\lambda ac}{\Delta_k}},\\
b_6=b_7
&=-(2k+1)\lambda\mu\sqrt{-ac}
-\sqrt{\frac{\Delta_k}{\lambda}},\\
b_8
&=\frac{(2k+1)\lambda^2c
\bigl(3(2k+1)a-5(2k-1)c\bigr)}
{8\Delta_k}
-\frac{\mu}{2a}\sqrt{-\lambda ac\Delta_k}.
\end{align*}

Hence, we again obtain two families. To determine the uniqueness of solutions for the above initial conditions, we compute for $\mu=1$ 
\begin{align*}
    \det( n \rm{Id}-d_{y_0}M_{-1})&=n(n+1)(n+2)^4((n+1)^2+4k^2),
\end{align*} 
hence $n\operatorname{Id}-d_{y_0}M_{-1}$ is invertible for every
$n\in\mathbb N$.  The singular IVP therefore has a unique solution
for each admissible pair $(a,c)$, giving a two-parameter family
$\Psi^k_{a,c}$ which extends smoothly across the singular orbit $(2k+1,1-2k)$.

When $\mu=-1$ we get 
\begin{align}\label{eqn:det_k}
    \det( n \rm{Id}-d_{y_0}M_{-1})&=(n-(2k-1))(n+1)(n+2)^5(n+(2k+1)),
\end{align} 
which vanishes for $n=2k-1$.  Thus Theorem~\ref{thm:IVPsoln} does not give
uniqueness in this case.  The kernel is one-dimensional, so the solution is
determined by one additional datum of order $2k-1$. We get a unique solution $\eta^k_{a,c,\nu}$ only after fixing $\nu=y_3^{(2k-1)}(0)$. 
The metric on the 5-dimensional singular orbit spanned by $u^i,i=1,\dots,5$ is given by 
\begin{align*}
   \frac{1}{\rho} \ {\rm{diag}}\left((4k+1)^2\rho^3,a,-c,a,-c\right),  \qquad \rho=\sqrt{\frac{\lambda |ac|}{(2k+1)a-(2k-1)c}}
\end{align*}
Therefore, we again consider only the branch where $a>0,c<0$; otherwise, the metric is not positive-definite.

Consider the involution
\[
\sigma:
(p_2,p_3;q_2,q_3,q_6,q_7)
\longmapsto
(p_3,p_2;q_3,q_2,q_7,q_6),
\]
with all the remaining coefficients fixed. 
If $\nu=0$, then for orders less than $2k-1$, uniqueness of
the Taylor recursion implies that every coefficient is
$\sigma$-invariant. At order $2k-1$, setting $\nu=0$ removes
the only $\sigma$-anti-invariant kernel component. At all subsequent
orders, the operators $n\operatorname{Id}-\d_{y_0}M_{-1}$
are invertible, and the $\sigma$-equivariance of the nearly parallel
equations implies inductively that all Taylor coefficients remain
$\sigma$-invariant. Hence $q_2=q_3, q_6=q_7$ throughout its maximal interval of existence.

\begin{thm}\label{thm:nu_k_symmetry_breaking}
For $k\geq1$, the family $\eta^{k}_{a,c,\nu}$ for $a>0,c<0, \nu\in\R$ defines a family of solutions of \eqref{eqns:gen_ng2} extending smoothly over the
singular orbit with weights $(2k+1,1-2k)$.
The parameter $\nu$ is the symmetry-breaking parameter for a distinguished
additional $\U(1)$-action commuting with the $\SO(4)$-action. 

\end{thm}

\section{Reflection and compact solutions}
\label{sec:reflection}

In this section, we aim to turn a solution that is smooth at one singular orbit into a
compact solution by reflection.  This is the analogue of the Doubling Lemma of \cite[Lemma 5.19]{Haskins-Lorenzo} in the six-dimensional nearly K\"ahler case. We establish a sufficient condition for a solution of \eqref{eqns:gen_ng2} that is complete at one end to close up smoothly at the other end after being reflected by one of the involutive symmetries listed in Proposition \ref{prop:discrete_symm}. 

With $u=(q_1,\ldots,q_8)$, write the system \eqref{eqns:gen_ng2} as
\begin{align}\label{eqn:system_F}
    \dot u=F(u),\qquad F:\;\mathcal U\subset\mathbb R^8\to\mathbb R^8.
\end{align}
It is defined and real‑analytic on the open set
  $$\mathcal U = \big\{ u\mid (q_2+q_7),(q_3+q_6)>0,\ (q_4+q_5)<0 \big\}.$$
  
Let $\tau$ be an involutive, time‑reversing symmetry of the system. We use the same notation to denote its action on $\mathbb R^8$. Let the fixed locus $\operatorname{Fix}(\tau)$ be the codimension-$k$ linear subspace defined by the vanishing of the linear functionals $g_i$, $i=1,\dots,k$, and let $L_\tau$ be the constant $k\times8$ matrix whose $i$th row is given by $g_i$. Thus 
\begin{equation*}
\operatorname{Fix}(\tau)=\big\{u\mid g_i(u)=0,\ i=1,\ldots,k\big\}.
\end{equation*}
From Proposition \ref{prop:discrete_symm} when $\tau=\tau_4$, the shooting parameters are
\[
        z=(a,c,\nu,T)\in\mathbb R^4.
\]
Here by Theorem \ref{thm:complete-31}, the parameters $a,c,\nu$ determine the local solution
$\eta_{a,c,\nu}$ emanating from the singular orbit at $t=0$,
while $T$ is the proposed reflection time. When $\tau=\tau_2\circ\tau3=:\tau_{23}$ the solution we find is
$U(1)$-invariant so we impose $\nu=0$ (Theorem \ref{thm:complete-31}), and
the shooting parameters reduce to $z=(a,c,T)\in\mathbb R^3$.

Let $V$ denote the coefficient space of invariant $3$-forms on a
principal orbit. For a time-reversing involution $\tau$, write
\[
        \mathcal F_\tau=\operatorname{Fix}(\tau)\subset V
\]
for its fixed locus. Since $\tau$ acts linearly on $V$,
$\mathcal F_\tau$ is defined by a finite collection of linear
equations. We denote by
\[
        L_\tau:V\longrightarrow\mathbb R^{r_\tau}
\]
the constant linear map formed by a maximal independent collection of
these equations, so that $\mathcal F_\tau=\ker L_\tau$. The corresponding shooting map is
\begin{align}\label{eq:shooting-map}
    \mathcal S_\tau(a,c,\nu,T)
   =L_\tau\bigl(\eta_{a,c,\nu}(T)\bigr).
\end{align}
Thus
\[
        \mathcal S_\tau(a,c,\nu,T)=0
        \quad\Longleftrightarrow\quad
        \eta_{a,c,\nu}(T)\in\mathcal F_\tau.
\]
Its zeros are precisely the trajectories which meet
$\operatorname{Fix}(\tau)$ at time $T$. 

Let $\mathfrak P$ denote the solution
space of \eqref{eqns:gen_ng2} for $t\in(0,T]$. We now state a sufficient condition under which $\Phi\in \mathfrak{P}$ extends smoothly to a singular orbit $K^u_{p,q}$ and closes smoothly at the other singular end $\tau(K^u_{p,q})$, so as to obtain a compact cohomogeneity-one solution. 

\begin{lemma}\label{lemma:timereversing}
    \textup{(Reflection principle).} Let $\tau$ be a time-reversing
    symmetry of \eqref{eqns:gen_ng2}, and suppose that
    \begin{enumerate}
        \item[(i)] $\Phi(t)\in\mathfrak P$ for all $t\in(0,T]$ and closes up smoothly on the singular orbit at $t=0$,
        \item[(ii)] $\mathcal S_\tau(a,c,\nu,T)=0$.
    \end{enumerate}
    
    \noindent
    Then $\widetilde\Phi(t):=\tau(\,\Phi(2T-t))$ is a nearly $G_2$ trajectory defined on $[T,2T]$ with $\widetilde\Phi(T)=\Phi(T)$. The solution
    $$\Psi(t)=\begin{cases}\Phi(t),&0\le t\le T,\\ \tau(\,\Phi(2T-t)),&T\le t\le 2T,\end{cases}$$
    determines a smooth complete nearly $G_2$-structure.
\end{lemma} 

\begin{proof}
Since the system \eqref{eqn:system_F} is autonomous, meaning that its
right-hand side has no explicit dependence on $t$, and reversible with
respect to $\tau$, that is, 
 $F(\tau u)=-D\tau_u F(u)$, the curve $t\mapsto\tau(\Phi(2T-t))$ is also a solution.  It agrees with $\Phi$
at $t=T$, and uniqueness at this regular orbit shows that the two
solutions agree wherever both are defined.  Near $t=2T$ the reflected
solution is the image under $\tau$ of the smooth cap at $t=0$, so it
extends over the second singular orbit. Moreover, since the orbital volume $V$ is even about $T$, $V'(T)=0$, and hence the maximal-volume orbit for $\Phi,\tau(\Phi)$ occurs at $t=T$.
\end{proof}

\noindent

The topology of the resulting cohomogeneity-one manifold is governed by the singular orbits at both ends. 
\begin{itemize}
    \item The $\tau_4$ reflection takes the singular isotropy group $K^i_{3,-1}$ at $t=0$ to $K^j_{1,-3}$ at $t=2T$ and one recovers
the Berger diagram. 
\item For $\tau_{23}\coloneqq \tau_2\circ\tau_3$, the two collapsing circles lie on
the same imaginary quaternion axis, as the involution takes the singular isotropy group $K^i_{3,-1}$ at $t=0$ to $K^i_{1,-3}$, and the resulting manifold is $M_{23}$, as described in Section \ref{section:coho-1action}.
\end{itemize}

\begin{rem} The Berger homogeneous solution can be obtained via the reflection principle using the involution $\tau_4$. It is given by 
$$\Psi_B(t)=\begin{cases}\eta_{a,c,\nu}(t),&0\le t\le \pi/6,\\ \tau_4(\,\eta_{a,c,\nu}(\pi/3-t)),&\pi/6\le t\le \pi/3,\end{cases}$$
 for 
\[(a,c,\nu)=(\sqrt{5}/{20},\>-{3\sqrt{5}}/{100},\>{\sqrt{5}}/{50})\]
and $\lambda=6/\sqrt{5}$.
\end{rem}

Our aim in the subsequent sections is to find parameters $(a,c,\nu,T)$ such that $\eta_{a,c,\nu}(T)\in \operatorname{Fix}(\tau)$ for $\tau=\tau_4$ and $\tau=\tau_{23}$ in order to obtain complete $\SO(4)$-invariant nearly parallel $\G2$-structures on $B^7$ and $M_{23}$, respectively. The main tool we use to prove the existence of such parameters is the Krawczyk inclusion theorem, as stated in Theorem \ref{thm:krawczyk} below. We first use numerical computation to obtain approximate parameters $(a_0,c_0,\nu_0,T_0)$ in $C$ for which $\|\mathcal S_\tau(a_0,c_0,\nu_0,T_0)\|$ is sufficiently close to zero. Then Krawczyk's theorem implies that there exists a genuine fixed point in $(a_0,c_0,\nu_0,T_0)+[-\epsilon,\epsilon]^4$ for some $\epsilon>0$. 

\subsection{Validated solutions}
\label{sec:validation}

We now explain the computer-assisted part of the proof. The aim is to
turn an approximate numerical solution of the shooting problem into a
genuine solution of the nearly parallel $\G2$-equations.

We have already seen from \eqref{eq:shooting-map} that a zero of $\mathcal S_\tau$ gives a trajectory which meets the fixed locus of the time-reversing symmetry $\tau$.  By the reflection principle Lemma \ref{lemma:timereversing}, this
trajectory then extends smoothly across the hypersurface $t=T$.

Let $\eta_{a,c,\nu}$ be the family from
Theorem \ref{thm:complete-31}, smooth at the singular orbit with slope
$(3,-1)$.  The two involutions from Proposition \ref{prop:discrete_symm} used here are $\tau_4$ and $\tau_{23}=\tau_2\tau_3$.  Their fixed sets are
given by
\begin{align}
 \operatorname{Fix}(\tau_4)
 &=\{q_1=q_8,\ q_2=-q_4,\ q_3=q_6,\ q_5=-q_7\},
 \label{eq:section5-fix-tau4}\\
 \operatorname{Fix}(\tau_{23})
 &=\{q_1=q_8,\ q_2=q_7,\ q_3=q_6,\ q_4=q_5\}.
 \label{eq:section5-fix-tau23}
\end{align}

For the $\tau_4$ reflection, define
\begin{equation}\label{eq:section5-shoot-tau4}
 \mathcal S_4(a,c,\nu,T)=
 \begin{pmatrix}
 q_1(T)-q_8(T)\\
 q_2(T)+q_4(T)\\
 q_3(T)-q_6(T)\\
 q_5(T)+q_7(T)
 \end{pmatrix},
 \qquad q(t)=\eta_{a,c,\nu}(t).
\end{equation}
Hence $\mathcal S_4(a,c,\nu,T)=0$ precisely when the trajectory meets
$\operatorname{Fix}(\tau_4)$ at time $T$.

In the $\tau_{23}$ case, we obtain that $\nu=0$. Thus, from Theorem \ref{thm:complete-31}, we have an additional $\U(1)$-symmetry, and
\begin{equation}\label{eq:section5-u1-identities}
        q_2=q_3,\qquad q_6=q_7
\end{equation}
throughout the interval of existence.  Thus the six variables $(q_1,q_2,q_4,q_5,q_6,q_8)$ form an invariant subsystem, on which the four equations in
\eqref{eq:section5-fix-tau23} reduce to three equations.  We therefore
define
\begin{equation}\label{eq:section5-shoot-tau23-u1}
 \mathcal S_{23}^{\U(1)}(a,c,T)=
 \begin{pmatrix}
 q_1(T)-q_8(T)\\
 q_2(T)-q_6(T)\\
 q_4(T)-q_5(T)
 \end{pmatrix},
 \qquad q(t)=\eta_{a,c,0}(t).
\end{equation}
A zero of this reduced map satisfies all four equations defining
$\operatorname{Fix}(\tau_{23})$.

Validation of the zeros of the fixed-set maps $\mathcal{S}_4, \mathcal{S}^{\rm{U}(1)}_{23}$ is achieved via Theorem \eqref{thm:krawczyk} below.  This theorem was formulated by Krawczyk in \cite{Krawczyk1969} and modified into an existence test by Moore \cite{Moore77}. Moreover, Rump \cite{RUMP198351} proved uniqueness of the solution within a neighbourhood of the approximate zero. The inclusion was further refined in \cite{neumaier1990interval, kearfott1996rigorous,Rump2010, krawczyk-19}. Recently, it was used by \cite{Vasquez} to obtain new cohomogeneity-one Einstein metrics on $\mathbb{CP}^n$ for $3\leq n\leq 7$.

We begin by explaining the basic idea. Let $F : \mathbb{R}^n \rightarrow \mathbb{R}^n$ be a $C^1$ function. We would like to solve the equation
\begin{equation}
    F(x)=0.
\end{equation}
Newton's method uses the operator
\begin{equation*}
N(x) = x - DF(x)^{-1}F(x). 
\end{equation*}
It is well known that if $F(x^*)=0$ and $DF(x^*)$ is nonsingular, then
$x^*$ is an attracting fixed point for $N$.
It turns out that the same is true if we replace $DF(x)^{-1}$ by a fixed matrix $Y$, which is sufficiently close to $DF(x^*)^{-1}$. The modified Newton operator is given by
\begin{equation}\label{eq:Nm}
N_m(x) = x - YF(x).
\end{equation}
If $U$ is homeomorphic to a closed finite-dimensional ball and if
$N_m(U) \subset U$,
then, by the Brouwer fixed-point theorem, there exists a solution $x_0 \in U$ of $N_m$, and the solution is unique if $N_m$ is a contraction on $U$. Since $Y$ is invertible, we obtain that $F(x_0) = 0$.

Now let $[x]\in \mathbb{IR}^n$ be a box in $\R^n$. Since
\[\operatorname{diam}([x]-YF([x])) \ge \operatorname{diam}([x])+\operatorname{diam}(YF([x]))\]
it is impossible to verify \eqref{eq:Nm} in a single interval evaluation. This shortcoming can be overcome by the Krawczyk operator. Let $x_0\in[x]$, and let $Y\in \mathbb{R}^{n\times n}$ be an invertible matrix. Then this operator is defined by
\begin{equation}
K(x_0, [x], F) := x_0 - YF(x_0) + (\mathrm{Id} - Y[DF([x])])([x] - x_0). 
\end{equation}
Observe that $N_m([x])\subset K(x_0,[x],F)$ and thus if $F(x^*)=0$ for some $x^*\in [x]$ then 
\[x^*=N_m(x^*)\subset K(x_0,[x],F).\] 
A more interesting result is the converse:

\begin{thm}[Krawczyk inclusion]\label{thm:krawczyk}
Let $F:\mathbb R^n\to\mathbb R^n$ be continuously differentiable
on the box $[x]\in \mathbb{IR}^n$. Let $x_0\in [x]$ and $Y$ be any invertible $n\times n$ matrix. If $$K(x_0, [x], F) \subset \mathrm{int}([x]),$$ then there exists a unique zero of $F$ in $[x]$.

\end{thm}

This method is
well suited to this problem because the shooting map is finite-dimensional,
but its evaluation requires solving an ODE.  Our computation therefore has
two steps.

\begin{enumerate}
\item First, produce rigorous enclosures for the singular start and for
the numerical flow. Since the ODE is singular at $t=0$, Taylor integration needs to start from $t=\delta$ for some $\delta>0$. We cannot numerically evaluate the exact values of the $q_i(\delta)$ by a finite Taylor expansion, but we prove a bound on the Taylor remainder by a singular Gronwall estimate and then use interval arithmetic on the entire interval to obtain rigorous enclosures. 
\item Second, use these rigorous enclosures in a finite-dimensional interval Krawczyk test. Once we obtain the rigorous enclosures, we can apply the Krawczyk interval test on these intervals and show that the approximate solutions pass the Krawczyk inclusion test, which proves the existence of a genuine solution in a neighbouring box of each approximate solution. 
\end{enumerate}

Let $z_0\in \mathbb R^n$ be an approximate zero of the shooting map.  We
validate a zero in a small box
\begin{align*}
z=z_0+ \Lambda\xi,
\qquad
\xi\in X:=[-1,1]^n,
\end{align*}
where $\Lambda$ is a diagonal scaling matrix given by $\Lambda=\operatorname{diag}(r_1,\ldots,r_n)$.
The scaled shooting map is
\begin{align*}
\widehat{\mathcal S}(\xi)
=
\mathcal S(z_0+\Lambda\xi).
\end{align*}
Thus the validation box is always the fixed unit cube $X=[-1,1]^n$.

Let $[\,\widehat{\mathcal S}(0)\,]$ be an interval enclosure of the
shooting residual at the centre and let
$[\,D\widehat{\mathcal S}(X)\,]$ be an interval enclosure of the Jacobian
on the whole cube.  If $Y$ is a numerical approximate inverse of the
midpoint Jacobian, then we can define the Krawczyk operator as
\begin{align*}
K(X)
=
-Y[\,\widehat{\mathcal S}(0)\,]
+
\bigl(I-Y[\,D\widehat{\mathcal S}(X)\,]\bigr)X.
\end{align*} By Theorem \ref{thm:krawczyk}  if
\begin{align*}
K(X)
\subset
\operatorname{int}(X),
\end{align*}
then $\widehat{\mathcal{S}}$ has a zero in $X$.

The Taylor arithmetic gives
a polynomial approximation $P(\xi)$ to $\widehat{\mathcal S}(\xi)$ on the
whole cube.  The remaining analytic work is to control the difference
between the true shooting map and this polynomial approximation, which we do in the following section.

\subsection{Singular start: $0\leq t\leq \texorpdfstring{\delta}{delta}$}
At $t=0$ the regularized variables $y$ satisfy a regular-singular
system
\begin{equation}\label{eq:section5-fuchsian-system}
        t y'=G(t,y;a,c,\nu).
\end{equation}
For the $\U(1)$-invariant problem we impose $\nu=0$ and use the reduced
six-dimensional system. Let

$$ \alpha= \begin{cases} (a,c,\nu),&\text{for }\tau_4,\\ (a,c),&\text{for }\tau_{23}\text{ on }\nu=0. \end{cases} $$

The terminal time $T$ is not included in $\alpha$, because it does not affect the local solution at the singular orbit.

Let $\alpha_0$ be the centre of the parameter box and write

$$ \alpha(\xi)=\alpha_0+\Lambda\xi, \qquad \xi\in[-1,1]^k, $$

where $\Lambda$ is the diagonal matrix containing the parameter radii. Here $k=3,2$ for $\tau_4,\tau_{23}$ respectively.
In the regularized variables, the singular initial-value problem has the form

$$ t\frac{\partial y}{\partial t} = G(t,y;\alpha). $$

Theorem \ref{thm:complete-31} shows that, for every admissible $\alpha$, this equation has a unique real-analytic solution $y(t,\alpha)$.

We want rigorous enclosures not only for $y(\delta,\alpha)$, but also for its derivatives with respect to the parameters. These derivatives are needed to enclose the Jacobian of the fixed-locus map. The derivative matrix
$$ W(t,\xi) = D_\xi y(t,\alpha(\xi)) $$
is an $n\times k$ matrix, where $n=8$ for $\tau_4$ and $n=6$ for the $\U(1)$-invariant solution for $\tau_{23}$.
Differentiating $ t\,\partial_t y=G(t,y;\alpha(\xi)) $ with respect to $\xi$ gives

$$ t\frac{\partial W}{\partial t} = D_yG(t,y;\alpha)\,W + D_\alpha G(t,y;\alpha)\,\Lambda. $$
We denote the augmented system of equations on $Z\coloneqq (y,W)$ by 
\begin{align*}
    tZ'&=  \mathcal{G}(t,Z;\alpha),  
\end{align*} where  $\mathcal{G}(t,Z;\alpha):=(G(t,y,\alpha),D_yG(t,y;\alpha)\,W + D_\alpha G(t,y;\alpha)\,L)$. We define the norm on $\R^n\times \operatorname{Mat}_{n\times k}(\R)$ by $$ \|(y,W)\|_\infty = \max\left\{ \|y\|_\infty,\max_{i,j}|W_{ij}| \right\}.$$
By \eqref{init_eta_ac}, the solution $\eta_{a,c,\nu}(t)$ is determined uniquely by $\alpha$; thus, by recursive differentiation, we can construct a polynomial
$$ \overline Z_N(t) = \sum_{j=0}^{N}Z_jt^j $$
which satisfies the augmented equation up to order $N$. 
For the $\tau_4$ computation we take $N=24$. For the $\U(1)$-invariant $\tau_{23}$ computation we take $N=8$. Since the polynomial $\overline Z_N$ is not the exact solution, the $O(t^{N+1})$ error is given by
$$ \mathcal R_N(t) = t\overline Z_N'(t) - \mathcal G(t,\overline Z_N(t);\alpha). $$
We require a bound on $R_N(t)$ that is uniform both in $t$ and in the parameters. More precisely, we find a constant $C_R$ such that
$$ \|\mathcal R_N(t)\|_\infty \leq C_Rt^{N+1}, \qquad 0\leq t\leq\delta, \ \alpha\in\alpha(\xi).$$

The existence of such a constant follows from the analyticity of $\mathcal G$. The important point in the computer-assisted argument is that we compute an explicit rigorous value of $C_R$.
Let $Z(t)$ denote the exact augmented solution and put

$$ e(t)=Z(t)-\overline Z_N(t). $$
If we denote by
\begin{align*}
     A_e(t) = \int_0^1 D_Z\mathcal G \bigl(t,\overline Z_N(t)+\theta e(t);\alpha\bigr) \,d\theta,
\end{align*} we obtain
\begin{align}\label{eqn:error-eqn}
    te'&=A_e(t)e-\mathcal R_N(t).
\end{align}

In addition to the coefficients $Z_0,\ldots,Z_N$, we compute an
interval enclosure of the next coefficient $Z_{N+1}$. Let
$C_{N+1}$ be an upper bound for its infinity norm. Fix a trial radius $H>0$ and consider the tube
\[
\mathcal T_H
=
\left\{
(t,Z,\alpha):
0<t\leq\delta,\ 
\alpha\in[\alpha],\
\|Z-\overline Z_N(t;\alpha)\|_\infty
\leq Ht^{N+1}
\right\}.
\]
Since
\[
Z(t)-\overline Z_N(t)
=
Z_{N+1}t^{N+1}+O(t^{N+2}),
\]
the inequality $C_{N+1}<H$ implies that the exact solution lies strictly inside
$\mathcal T_H$ for all sufficiently small positive $t$. We shall prove below, by a first-exit argument, that it remains in $\mathcal T_H$ for every $0<t\leq\delta$.

If we let 
\begin{align*}
    L\coloneqq D_yG(0,y_0;\alpha_0) = D_yM_{-1}(0,y_0;\alpha_0),
\end{align*}
then for $ C_j u = D_y\!\left(D_yG(0,y;\alpha_0)w_{j,0} +D_\alpha G(0,y;\alpha_0)\Lambda e_j\right)_{y=y_0}u$, consider the augmented vector field
$$ D_Z\mathcal G(0,Z_0) \begin{pmatrix} u\\ v_1\\ \vdots\\ v_k \end{pmatrix} = \begin{pmatrix} Lu\\ C_1u+Lv_1\\ \vdots\\ C_ku+Lv_k \end{pmatrix}. $$
Consequently,
$$ A_0:=D_Z\mathcal G(0,Z_0) = \begin{pmatrix} L&0&0&\cdots&0\\ C_1&L&0&\cdots&0\\ C_2&0&L&\cdots&0\\ \vdots&\vdots&&\ddots&\vdots\\ C_k&0&0&\cdots&L \end{pmatrix}, $$ and $ \det(m \Id-A_0)= \det(m\Id-L)^{k+1}$.

Hence by \eqref{eqn:det_M}, all eigenvalues of $A_0-m\Id$ are negative for $m>1$. By the Lyapunov theorem there is a symmetric positive-definite matrix $P$ satisfying
\begin{align}\label{eqn:P_matrix}
    -\bigl((A_0-mI)^TP+P(A_0-mI)\bigr)>0 .
\end{align}
By continuity, the same strict inequality, with a possibly smaller positive lower bound, holds in a sufficiently small neighbourhood of the singular initial data; hence, the same $P$ works for all $A\in  D_Z\mathcal G(\mathcal T_H)$. For the particular tube $\mathcal T_H$, this is verified directly by interval arithmetic. The positivity of $P$ and of the matrix in
\eqref{eqn:P_matrix} is proved by interval $LDL^T$
decomposition of $Q=-\bigl((A-mI)^TP+P(A-mI)\bigr)$ and $P$.  For the full parameter boxes the verified lower bounds
for the two smallest diagonal entries in the $LDL^T$ decomposition are
\begin{center}
\begin{tabular}{c|cc}
 & $P$ & $Q$\\ \hline
$\tau_4$ & $0.07825$ & $0.85567$\\
$\tau_{23}$  & $0.07230$ & $0.99982$.
\end{tabular}
\end{center}

Let $U(t,s)$, for $0<s\le t\le\delta$, be the solution operator of the homogeneous part of the error equation \eqref{eqn:error-eqn}
\begin{equation}\label{eqn:error-homo}
       t v'(t)=A_e(t)v(t),
       \qquad U(s,s)=I.
\end{equation}
Then 
\begin{align*}
    e(t)=-\int_0^t U(t,s) \frac{R^{N}}{s} ds.
\end{align*}
To bound $\|e(t)\|_\infty$, we need an estimate of the form
\begin{align}\label{eqn:bound_U}
    \|U(t,s)\| \leq C_L\left(\frac ts\right)^m,
\end{align} for $1<m<N+1$.
We use the matrix $P$ from \eqref{eqn:P_matrix} to prove this bound. Define the $P$-norm by
$$ |v|_P^2=v^TPv. $$
Let $ w(t)=t^{-m}v(t)$. Then, from \eqref{eqn:error-homo}, we have $$ tw'=(A_e(t)-mI)w. $$
Therefore
$$ \frac{d}{dt}|w|_P^2 = \frac1t w^T\left( (A_e-mI)^TP+P(A_e-mI) \right)w. $$
The verified matrix inequality \eqref{eqn:P_matrix} implies
$$ \frac{d}{dt}|w|_P^2\leq0. $$
Hence, for $0<s\leq t$,
\begin{align*}
    |w(t)|_P\leq |w(s)|_P \implies |v(t)|_P \leq \left(\frac ts\right)^m|v(s)|_P .
\end{align*}
Converting the $P$-norm to the infinity norm gives \eqref{eqn:bound_U} where the code uses the rigorous bound
$$  C_L= \sqrt{ n_Z\, \|P\|_\infty\, \|P^{-1}\|_\infty }, $$
and $n_Z$ is the dimension of the augmented variable $Z$. The matrix $P^{-1}$ and all norms are enclosed by outward-rounded interval arithmetic.

We now close the tube estimate by a first-exit argument. On any
interval $(0,t_1]$ on which the exact solution remains in
$\mathcal T_H$, variation of constants gives, for $0<s<t\leq t_1$,
\[
e(t)
=
U(t,s)e(s)
-
\int_s^t
U(t,r)\frac{\mathcal R_N(r)}{r}\,dr.
\]
Since $e(s)=O(s^{N+1})$ as $s\to0$ and
\[
\|U(t,s)\|_\infty
\leq
C_L\left(\frac ts\right)^m,
\qquad m<N+1,
\]
we have
\[
\|U(t,s)e(s)\|_\infty
=
O\bigl(s^{N+1-m}\bigr)
\longrightarrow0
\qquad\text{as }s\to0.
\]
Consequently,
\[
e(t)
=
-\int_0^t
U(t,s)\frac{\mathcal R_N(s)}{s}\,ds.
\]
Using $\|\mathcal R_N(s)\|_\infty
\leq C_Rs^{N+1}$, we obtain
\begin{align*}
\|e(t)\|_\infty
&\leq
\int_0^t
C_L\left(\frac ts\right)^m
C_Rs^N\,ds  =
\frac{C_LC_R}{N+1-m}\,t^{N+1}.
\end{align*}
Set
\[
H_0:=\frac{C_LC_R}{N+1-m}.
\]
The interval calculation verifies the two strict inequalities
\begin{equation}\label{eq:tube-closing-inequalities}
C_{N+1}<H,
\qquad
H_0<H.
\end{equation}

Suppose, for contradiction, that the exact solution leaves
$\mathcal T_H$ before $t=\delta$, and let $t_*\leq\delta$ be its
first exit time. The solution remains in $\mathcal T_H$ on
$(0,t_*]$, so the preceding estimate applies at $t=t_*$. Hence
\[
\|e(t_*)\|_\infty
\leq
H_0t_*^{N+1}
<
Ht_*^{N+1},
\]
contradicting the definition of $t_*$, according to which
\[
\|e(t_*)\|_\infty=Ht_*^{N+1}.
\]
Therefore no first exit occurs, and
\[
\|Z(t)-\overline Z_N(t)\|_\infty
\leq
H_0t^{N+1},
\qquad
0<t\leq\delta.
\]
Thus the exact solution remains in $\mathcal T_H$ throughout the
singular-start interval, and its error has the same order
$t^{N+1}$ as the Taylor defect.

\paragraph{Numerical error bounds.}
For the full $\tau_4$ parameter box, we use
\[
N=24,\qquad m=4,\qquad \delta=\frac18.
\]
The interval computation gives
\[
R
=
3.0981467332626974\times 10^7,
\qquad
C_L
=
11.4718403931572.
\]
Consequently,
\[
H_0
=
\frac{C_LC_R}{N+1-m}
=
\frac{C_LC_R}{21}
=
1.6924497542176703\times 10^7.
\]
The larger radius used to verify the Lyapunov inequality on the
whole tube is
\[
H=3.384899508435341\times 10^7.
\]
Although these constants are large, they multiply by $t^{25}$, thus, at
$t=\delta=1/8$, the actual error satisfies
\[
\left\|Z(\delta)-\overline Z_{24}(\delta)\right\|_\infty
\leq
H_0\delta^{25}
=
H_0\left(\frac18\right)^{25}
<
4.49\times 10^{-16}.
\]

For the full $\U(1)$-invariant $\tau_{23}$ parameter box, we use
\[
N=8,\qquad m=4,\qquad \delta=\frac1{1024}.
\]
The interval computation gives
\[
R
=
1.3555331515188007,
\qquad
C_L
=
8.73364532077527.
\]
It follows that
\[
H_0
=
\frac{C_LC_R}{N+1-m}
=
\frac{C_LC_R}{5}
=
2.3677491531835867.
\]
The trial radius used in the tube verification is
\[
H=4.735498306367173.
\]
The computed enclosure of $Z_{N+1}$ also satisfies
$C_{N+1}<H$, while the displayed values give $H_0<H$.
Thus \eqref{eq:tube-closing-inequalities} holds on the full parameter box.
At $t=\delta=1/1024$, the corresponding error satisfies
\[
\left\|Z(\delta)-\overline Z_8(\delta)\right\|_\infty
\leq
H_0\delta^9
=
H_0\left(\frac1{1024}\right)^9
<
1.92\times 10^{-27}.
\]
All displayed bounds are obtained with outward-rounded interval
arithmetic and hold uniformly over the corresponding full parameter
boxes.
\subsection{Regular flow: $\delta\leq t\leq T$}
The Taylor argument supplies a rigorous enclosure at the positive time $t=\delta$. From there onward, the ODE is regular.
At the end of the singular-start calculation we have $q(\delta;\beta)\in[q_\delta] $ and an enclosure of $ D_\beta q(\delta;\beta), $
uniformly over the parameter box. Here $\beta$ denotes the parameters other than the reflection time $T$.

Since $T$ also varies over an interval, it is convenient to replace $t\in[\delta,T]$ by
$$ t=\delta+s(T-\delta), \qquad 0\le s\le1. $$
The rescaled system \eqref{eqn:system_F} becomes
\begin{equation}\label{eq:section5-rescaled-flow}
 \frac{dq}{ds}=(T-\delta)F(q),
 \qquad
 \frac{dT}{ds}=0.
\end{equation}
Let $\Phi_s(q_\delta,T)$ denote the flow of this rescaled system. Then
$$ \Phi_1(q_\delta,T) = \bigl(q(T),T\bigr). $$
CAPD propagates the interval initial set $ [q_\delta]\times[T] $ from $s=0$ to $s=1$. It rigorously encloses:
\begin{enumerate}
    \item the solution $q(s)$ on every integration step,
    \item the endpoint $q(T)$,
    \item the derivative of the endpoint map with respect to both $q_\delta$ and $T$.
\end{enumerate}
More precisely, the validated $C^1$ solver produces an interval enclosure $[M]$ for every $(q_\delta,T)\in [q_\delta]\times[T]$ such that
$$ D_{(q_\delta,T)}\Phi_1(q_\delta,T)\subset [M] .  $$
The derivative with respect to the original scaled parameters is then obtained by the chain rule. CAPD also verifies on every integration step that the constraints 
$$ q_2+q_7>0,\qquad q_3+q_6>0,\qquad -(q_4+q_5)>0, $$ are satisfied throughout the interval. The output of this section is therefore a rigorous enclosure over the complete parameter box, of
$$ q(T;\xi) \qquad\text{and}\qquad D_\xi q(T;\xi). $$

For $\tau_4$ the parameter centre and radii are
\begin{align*}
 z_4={}&(0.042329982765393195,-0.06469940434457604,
       0.006759263945977987,0.7656914313714904),\\
 r_4={}&(10^{-6},10^{-6},10^{-6},10^{-6}),
 \qquad (\delta,N)=(1/8,24).
\end{align*}
For $\tau_{23}$ on $\nu=0$ they are
\begin{align*}
 z_{23}^{\U(1)}={}&(0.025880416406250467,-0.07764124921873236,0.8578017401911906),\\
 r_{23}^{\U(1)}={}&(10^{-7},3\cdot10^{-7},5\cdot10^{-7}),
 \qquad (\delta,N)=(1/1024,8).
\end{align*}
Here $N$ is the degree of the singular Taylor polynomial.  The regular
flow uses order $24$ in both cases.

\subsection{Compact inhomogeneous solutions}

We now use the endpoint enclosures obtained from the validated regular
flow to prove that the fixed-locus maps have zeros in the stated
parameter boxes.

Recall that for $\tau_4$, the fixed-locus map is
\[
\mathcal S_4(a,c,\nu,T)
=
\begin{pmatrix}
q_1(T)-q_8(T)\\
q_2(T)+q_4(T)\\
q_3(T)-q_6(T)\\
q_5(T)+q_7(T)
\end{pmatrix}.
\]
For the $\U(1)$-invariant $\tau_{23}$ problem, where
$q_3=q_2$ and $q_7=q_6$, we use
\[
\mathcal S_{23}(a,c,T)
=
\begin{pmatrix}
q_1(T)-q_8(T)\\
q_2(T)-q_6(T)\\
q_4(T)-q_5(T)
\end{pmatrix}.
\]
Thus $\mathcal S_\tau=0$ is precisely the condition that the endpoint belongs to the corresponding reflection fixed locus.

For each of the two fixed-locus maps, CAPD provides an interval
enclosure $[f_0]$ of
$\widehat{\mathcal S}_\tau(0)$ and an interval matrix
$[J_\tau]$ containing
\[
D\widehat{\mathcal S}_\tau(\xi)
\qquad
\text{for every }\xi\in X_\tau:=[-1,1]^{k}.
\]
The invertible matrix $A_\tau$ in Theorem \ref{thm:krawczyk} is a fixed matrix obtained from a
numerical approximation to the inverse of the midpoint of
$[J_\tau]$.  All subsequent operations involving $A_\tau$ are performed with outward-rounded interval arithmetic.

The verified bounds are
\[
\begin{array}{c|cc}
&
\|[ f_0]\|_\infty
&
\|I-A_\tau [J_\tau]\|_\infty

\\ \hline
\tau_4
&
<1.8\cdot10^{-14}
&
<0.003706

\\
\tau_{23}
&
<1.5\cdot10^{-14}
&
<0.79901

\end{array}
\]

More explicitly, for the $\tau_4$ problem the computed Krawczyk image
satisfies
\[
\begin{split}
K_4\subset{}&
[-5.9\cdot10^{-4},5.9\cdot10^{-4}]
\times[-5.2\cdot10^{-4},5.2\cdot10^{-4}]\\
&\times[-4.5\cdot10^{-4},4.5\cdot10^{-4}]
\times[-3.8\cdot10^{-3},3.8\cdot10^{-3}].
\end{split}
\]
In particular,
\begin{align}\label{eqn:S4_inclusion}
    K_4\subset\operatorname{int}[-1,1]^4.
\end{align}

For the $\U(1)$-invariant $\tau_{23}$ problem,
\[
K_{23}
\subset
[-0.322,0.322]
\times[-0.322,0.322]
\times[-0.800,0.800],
\]
and hence
\begin{align}\label{eqn:S23_inclusion}
    K_{23}\subset\operatorname{int}[-1,1]^3.
\end{align}

The hypotheses of Theorem \ref{thm:krawczyk} are therefore satisfied in
both cases.  It follows that $\widehat{\mathcal S}_4$ has a unique zero
in $[-1,1]^4$ and that $\widehat{\mathcal S}_{23}$ has a unique zero
in $[-1,1]^3$.  Equivalently, $\mathcal S_4$ and
$\mathcal S_{23}$ each have a unique zero in their respective certified parameter boxes.

\subsubsection{The inhomogeneous solution on $B^7$}
By \eqref{eqn:S4_inclusion}, the map $\mathcal S_4$ has a unique zero in
\begin{equation}\label{eq:section5-tau4-box}
        z_4+10^{-6}[-1,1]^4.
\end{equation}
By the reflection principle, the corresponding trajectory reflects
smoothly across $\operatorname{Fix}(\tau_4)$.  The singular orbit at
$t=2T$ is the $\tau_4$-image of the singular orbit at $t=0$, and the
resulting compact manifold is $SO(5)/SO(3)$.

It remains to show that the new complete solution thus obtained is not isometric to the known homogeneous Berger solution. We do so by comparing the interval lengths for the two metrics. In our normalisation $|dt|^2=1$, so $t$ is the arclength parameter. 

\begin{prop}\label{prop:not-homogeneous-berger}
The new metric obtained by the $\tau_4$ reflection is not
isometric, even up to homothety, to the standard homogeneous
Berger metric.
\end{prop}

\begin{proof}
Let $g$ denote the new metric obtained, and let $g_B$
be the homogeneous Berger metric, both normalised by
$\lambda=6/\sqrt5$. Then
\[
    \operatorname{Ric}(g)=\frac{27}{10}g,\ \ \operatorname{Ric}(g_B)=\frac{27}{10}g_B.
\]
Consequently, any homothety between these metrics must be an isometry.

Suppose that such an isometry
$f\colon(B,g)\to(B,g_B)$ exists. It conjugates the given
$G=\SO(4)$ action to a subgroup of
$\operatorname{Isom}_0(B,g_B)=\SO(5)$.
Every subgroup of $\SO(5)$ isomorphic to $\SO(4)$ is conjugate
to the standard subgroup fixing a vector in $\R^5$; see
\cite[Sections~2.2 and~2.6]{Ball-Madnick}.
After composing $f$ with an isometry of $g_B$, we may therefore
assume that $f$ maps the $G$-orbits to the standard
$\SO(4)$-orbits. In particular, it induces an isometry of the
metric orbit spaces
\[
    (B,g)/G \longrightarrow (B,g_B)/\SO(4).
\]

These orbit spaces are closed intervals, whose lengths we now
compare. For the reflected solution, $t$ is arclength along a
normal geodesic and the two singular orbits occur at $t=0$
and $t=2T$. Hence $L(g)=2T$. The certified parameter enclosure gives
\[
    \left|T-0.7656914313714904\right|\leq 10^{-6},
\]
and therefore
\[
    L(g)\in[1.531380,\,1.531386].
\]
In the same normalisation $\lambda=6/\sqrt{5}$ from \eqref{eq:hom_ng2}, $$L(g_B)=\frac{\pi}{3}.$$

Since $\pi/3\notin[1.531380,\,1.531386]$, the two metric orbit
spaces cannot be isometric. This contradicts the existence
of $f$.
\end{proof}

We can now state the main result of the section. 

\begin{thm}\label{thm:certified-berger}
There exists a smooth inhomogeneous complete nearly parallel $G_2$-structure on
$SO(5)/SO(3)$, invariant under the cohomogeneity-one $SO(4)$-action.

\end{thm}

\subsubsection{The $\U(1)$-invariant solution on $M_{23}$}

Let
\[
 z_{23}^{\U(1)}=
 (0.025880416406250467,-0.07764124921873236,
 0.8578017401911906)
\]
and
\[
 D_{23}=\operatorname{diag}(10^{-7},3\cdot10^{-7},5\cdot10^{-7}).
\]
By \eqref{eqn:S23_inclusion}, the reduced map
$\mathcal S_{23}^{\U(1)}$ has a unique zero in
\begin{equation}\label{eq:section5-tau23-u1-box}
 z_{23}^{\U(1)}+D_{23}[-1,1]^3.
\end{equation}
Note that this is a uniqueness statement within the $\U(1)$-invariant subfamily. No uniqueness among all non-invariant nearby solutions is asserted.

The identities \eqref{eq:section5-u1-identities} hold along the whole
trajectory, and the three matching equations imply all four equations
in \eqref{eq:section5-fix-tau23}.  The reflection principle therefore
gives a smooth complete solution on $M_{23}$.

\begin{prop}\label{prop:certified-tau23}
There exists a smooth complete nearly parallel $\G2$-structure on
$M_{23}$, invariant under the cohomogeneity-one $\SO(4)$-action and
under the additional $\U(1)$-action.  Its shooting parameters satisfy
\begin{align*}
 a&\in 0.025880416406250467+10^{-7}[-1,1],\\
 c&\in -0.07764124921873236+3\cdot10^{-7}[-1,1],\\
 \nu&=0,\\
 T&\in 0.8578017401911906+5\cdot10^{-7}[-1,1].
\end{align*}
\end{prop}
We compare the certified solution with the explicit Sasaki--Einstein
metric \(A^{1,3,2}\) in \cite{TomasielloZaffaroni2011} with \eqref{eqn:phi_t0} for $m=3,n=-1$. 
After renormalising the Sasaki--Einstein metric in \cite[(4.16)--(4.18)]{TomasielloZaffaroni2011}
so that nearly parallel constant $\lambda=\frac6{\sqrt5}$, we obtain 
\[
\begin{aligned}
\varphi_A\big|_{\Sigma_-}
={}&-\frac5{18}v^{12}u^1
+\frac{5\sqrt3}{36}
 \left(v^1u^{23}+v^1u^{45}+v^2u^{25}+v^2u^{34}\right)-\frac{5\sqrt5}{216}u^{124}
+\frac{5\sqrt5}{72}u^{135}.
\end{aligned}
\]
This directly implies that $c=-3a$, and
\[
 a_A =\frac{5\sqrt5}{432},
 \qquad
 c_A=-3a_A=-\frac{5\sqrt5}{144}, \qquad \nu_A=0.
\]
Since \(t\) is arclength, the corresponding shooting time is the the length of radial part of the Sasaki–Einstein metric
\[
\begin{split}
 T_A&=\frac{\sqrt5}{3}
   \int_0^{\pi/2}
   \sqrt{\frac{4-\cos^2\theta}{7-\cos^2\theta}}\,d\theta,
\end{split}
\]
A direct interval evaluation of this one-dimensional integral gives
\[
T_A\in(0.8578016920,\,0.8578017884).
\]
Consequently,
\[
 (a_A,c_A,T_A)
 \in
 z_{23}^{\U(1)}+D_{23}(-1,1)^3.
\]
Since the Krawczyk argument shows that
\(\mathcal S_{23}^{\U(1)}\) has a unique zero in
\eqref{eq:section5-tau23-u1-box}.  Therefore the certified zero is
exactly \((a_A,c_A,T_A)\), and the lifted metric is the explicit
Sasaki--Einstein metric \(A^{1,3,2}\).

\paragraph{Computer codes to reproduce the numerical results.} The singular Taylor coefficients and their parameter derivatives are computed in Julia with $256$-bit outward-rounded interval arithmetic.
The regular flow and its first derivative are enclosed by the multiprecision $C^1$ interval solver (CAPD).  The final interval matrix products and Krawczyk inclusions are again carried out in Julia.  

The scripts \texttt{approximate\_S4.jl} and
{\texttt{approximate\_S23.jl}} locate the approximate zeros of the two shooting maps, while {\texttt{certify\_tau4\_capd.jl}} and {\texttt{certify\_tau23\_u1\_capd.jl}} rigorously certify them.  The supporting files {\texttt{SingularStart.jl}}, \texttt{CapdFlow.jl}, \texttt{G2CapdCertificates.jl}, and \texttt{IntervalKrawczyk.jl}
validate the regular-singular Taylor start, the CAPD enclosure of the
solution and its variational equation, preservation of the
$\G2$-chamber, and the strict Krawczyk inclusion.  Finally, \texttt{approximate\_S7\_gluing.jl}, together with
{\texttt{ApproximateS7Gluing.jl}}, performs a non-rigorous search for two-ended solutions on $S^7$. 

The source code accompanying this paper, including the programs used for the rigorous certificates and the numerical experiments, is available at
\url{https://github.com/raginisinghalmath/G2ReflectionCertificates}.

\medskip

Nearly parallel $\G2$-structures are classified \cite{Bar_realkilling,friedkath} as proper, Sasaki-Einstein, and 3-Sasakian relative to the dimension of their space of Killing spinors $\mathcal K$.  If the simply connected complete
manifold $\widetilde M$ is not round $S^7$, then\[
 \dim\mathcal K\in\{1,2,3\}.
\]
The round case does not occur here.  Indeed, $H^4(B^7;\mathbb Z)\cong\mathbb Z_{10}$, whereas the universal cover of $M_{23}$ is an $\mathbb RP^3$-bundle over $S^2\times S^2$ and has
$b_2=2$.  Neither universal cover is therefore diffeomorphic to
$S^7$. 

If $\Psi_1$, $\Psi_2$ are independent unit Killing spinors on a nearly $\G2$-manifold $M$, then there exists a nowhere-vanishing unit vector field $V$ such that $\Psi_2=V\cdot\Psi_1$, where $\cdot$ denotes Clifford multiplication. Since the structures in Theorem \ref{thm:certified-berger} and Proposition \ref{prop:certified-tau23} are invariant under the cohomogeneity-one action, the vector field $V$ must also be invariant under the same action. However, since the singular isotropy group does not fix any vector field, any smooth invariant vector field must vanish at the singular set, contradicting the fact that $V$ is nowhere vanishing. Hence no such $V$ can exist and $\dim\mathcal{K}=1$ for the nearly parallel $\G2$-structures constructed on $B^7$ and $M_{23}$. Moreover, since $B^7$ is simply connected, we have the following result. 

\begin{prop}\label{prop:proper-solutions}
The nearly parallel $\mathrm G_2$-structure constructed in
Theorem \ref{thm:certified-berger} is proper. Equivalently, the metric cone
associated with the structure has full holonomy
$\operatorname{Spin}(7)$.
\end{prop}

\begin{rem}[The $S^7$ matching problem]
    Using the families of solutions $\xi_{a,b,c}$ and $\eta_{a,c,\nu}$ extending to the singular orbits $K^i_{1,-1}$ and $K^j_{1,-3}$, respectively, one can obtain a complete solution on $S^7$ by solving a matching problem. If there exist parameters $(a_1,b_1,c_1)$ such that $c_1>a_1$, $a_1c_1-b_1^2>0$, and $(a_2,c_2,\nu_2)$ such that $a_2>0,c_2<0,\nu_2\geq 0$, with $\xi_{a_1,b_1,c_1,}(T_1)=\tau_4(\eta_{a_2,c_2,\nu_2}(T_2) ) $ for some $T_1,T_2>0$, then 
    \begin{align*}
       \Phi(t)&= \begin{cases}
            \xi_{a_1,b_1,c_1}(t) & 0\leq t\leq T_1\\
            \tau_4(\eta_{a_2,c_2,\nu_2}(T_1+T_2-t))& T_1\leq t \leq T_1+T_2
        \end{cases}
    \end{align*}
    defines a complete $\SO(4)$-invariant nearly parallel $\G2$-structure on $S^7$. For $\lambda=6/\sqrt{5}$, the homogeneous nearly parallel $\G2$-structure on the squashed $S^7$ described in \eqref{eqn:S7-ortho} can be obtained via this construction for
    \begin{align*}
     \left(a_1,b_1,c_1,T_1\right)&= \left(\frac{\sqrt{5}}{25},\frac{-2\sqrt{5}}{25},\frac{19\sqrt{5}}{25},\frac{\pi}{6}\right),\qquad
        \left(a_2,c_2,\nu_2,T_2\right)= \left(\frac{\sqrt{5}}{25},\frac{-3\sqrt{5}}{5},\frac{\sqrt{5}}{50},\frac{\pi}{6}\right) .
        \end{align*}
  \end{rem}

   In our numerical search, we could not find any other set of parameters satisfying the matching condition, which motivates the following conjecture.

\begin{conj}
Every complete nearly parallel $\mathrm G_2$-metric on $S^7$
invariant under the cohomogeneity-one $\SO(4)$-action considered
here is homogeneous. 
\end{conj}

\begin{rem}[The family $A^{3,1,k+1}$]
For $k\geq 1$, let $A^{3,1,k+1}$ \cite{TomasielloZaffaroni2011} denote the total space of the
principal circle bundle over
$\mathbb{CP}^{2}\#\overline{\mathbb{CP}}^{\,2}$ with Euler class
\[
    e=3x+y+(k+1)z,
    \qquad
    z^{2}+2(x+y)z=0.
\]
Equivalently, up to reversing the parametrization of one of the
singular circles, it is the cohomogeneity-one manifold with group
diagram
\[
 \Delta\mathbb Z_{4}
 \ \subset\
 \left\{
     C^{\,i}_{1,-3},
     C^{\,i}_{2k+1,\,1-2k}
 \right\}
 \ \subset\
 S^{3}\times S^{3},
\]

The matching problem between
$\eta_{a_{1},c_{1},\nu_{1}}$ and
$\eta^{k}_{a_{k},c_{k},\nu_{k}}$ has precisely the above group diagram. For every $k=2,\ldots,8$ considered numerically, this matching
problem admitted a solution with residual of order $10^{-12}$. In every case the matching parameters satisfied, to numerical
precision,
\[
    \nu_{1}=\nu_{k}=0,
    \qquad
    \frac{c_{1}}{a_{1}}=-3,
    \qquad
    \frac{c_{k}}{a_{k}}
       =-\frac{2k+1}{2k-1}.
\]
In particular, the computation recovers the $\mathrm{SU}(2)^2\times \mathrm{U}(1)$ metrics on $A^{3,1,k+1}$, but gives no numerical evidence within this ansatz for a second, genuinely symmetry-breaking nearly parallel $\mathrm{G}_{2}$-metric on these manifolds.
\end{rem}

\section{Local sine cone desingularisation}
\label{sec:desingularisation}

The purpose of this section is to replace the singular end of the sine cone by a small smooth singular orbit.  The model for the smooth singular orbit is a finite quotient of the asymptotically conical (AC) torsion-free $\mathrm G_2$-metric in the 
$C_7$ family constructed by Foscolo--Haskins--Nordstr\"om \cite{FHN}. Let $g_N$ be the invariant nearly Kähler metric on $N=SO(4)/\Z_2^2$. The AC member of the $C_7$ family is asymptotic to the torsion-free $\mathrm G_2$-cone over $(S^3\times S^3)/\mathbb Z_4$, equipped with the nearly Kähler structure descended from $S^3\times S^3$. After taking an additional free $\mathbb Z_2$-quotient of this, we obtain the asymptotic cone $C(N)=(0,\infty)\times N$ with the metric $g_C=dr^2+r^2g_N$, while the sine cone on $N$ is $SC(N)=(0,\pi)\times N$ with $ g_{\mathrm{sc}}=dt^2+(\sin t)^2g_N$. The cone $C(N)$ therefore occurs in two different ways:
\begin{itemize}
    \item it is the tangent cone of the sine cone over the nearly K\"ahler structure on $N$ at $t=0$,
    \item it is the
asymptotic cone of $C_7$ at infinity.
\end{itemize}

We now explain how these models arise from the smooth nearly parallel
$\G2$-structures constructed in Theorem \ref{thm:complete-11}.  In
Subsection~\ref{subsection:inner-limit} we introduce a one-parameter subfamily
$$
\varphi_{\epsilon,\beta(\epsilon)}
=
\xi_{a,0,c},
\qquad \epsilon>0,\quad \beta(\epsilon)>0.
$$
Here $\epsilon$ measures the scale of the singular orbit, while
$\beta$ is the parameter in the singular initial data coming from the $C_7$ family described in Subsection \ref{subsection:C7-family}.  The purpose of the matching argument is to choose a function $\beta=\beta(\epsilon)$, with
$\beta(\epsilon)\to\beta_{\mathrm{ac}}$, for which the resulting
one-parameter family $\varphi_\epsilon:= \varphi_{\epsilon,\beta(\epsilon)}$ locally desingularises the sine cone.

There are two different limits of this same family, corresponding to
two different spatial scales.  In the outer region, $t$ is kept
fixed away from the collapsing singular orbit.  No rescaling is made, and we prove that
$$
\varphi_\epsilon\longrightarrow\varphi_{\mathrm{sc}}
\qquad\text{as }\epsilon\longrightarrow0
$$
on compact subsets of the smooth part of the sine cone.

To see the geometry near the singular orbit, we instead introduce the inner coordinate $\widehat t=\frac{t}{\epsilon}$
and the dilation
$$
D_\epsilon(\widehat t,x)=(\epsilon\widehat t,x).
$$
The rescaled structure and metric are
$$
\widehat\varphi_{\epsilon}
  =\epsilon^{-3}D_\epsilon^*\varphi_\epsilon,
\qquad
\widehat g_{\epsilon}
  =\epsilon^{-2}D_\epsilon^*g_\epsilon.
$$
If $d\varphi_\epsilon
  =\lambda*_{\varphi_\epsilon}\varphi_\epsilon$, then
$$
d\widehat\varphi_\epsilon
  =
  \epsilon\lambda
  *_{\widehat\varphi_\epsilon}\widehat\varphi_\epsilon.
$$
Thus, as $\epsilon\to0$, the nearly parallel constant tends to zero
on the inner scale.  The inner limit is therefore torsion-free.  The
choice $\beta(\epsilon)\to\beta_{\mathrm{ac}}$ gives
$$
\widehat\varphi_\epsilon
\longrightarrow
\varphi_{C_7}
\qquad\text{as }\epsilon\longrightarrow0,
$$
where on the $\SO(4)$-manifold, $\varphi_{C_7}$ denotes the finite quotient described in Proposition \ref{prop:c7-finite-quotient}.

The outer and inner descriptions are related through the intermediate region $\epsilon\ll t\ll1$. Indeed, along any choice of $t=t(\epsilon)$ satisfying these
inequalities,
$$
t\longrightarrow0,
\qquad
\widehat t=\frac{t}{\epsilon}\longrightarrow\infty.
$$
Thus, in the intermediate region, the outer sine cone solution approaches its tangent cone, while
the inner $C_7$ solution approaches its asymptotic cone.  These are the same torsion-free cone as demonstrated in Figure \ref{fig:local-sine cone-desingularisation}.  It is this common cone, rather than the two limiting solutions directly, that provides the matching region.

In Subsection \ref{subsection:cone-fixed-point}, we explain that the common cone is a hyperbolic fixed point of
the scale-invariant dynamical system.  The $C_7$ trajectory approaches this fixed point along its stable direction. There are two independent outgoing directions: one is tangent to the sine cone trajectory, while the other is a faster-growing unstable direction that obstructs the desired matching. We vary the parameter $\beta$ to eliminate this unwanted component. The resulting solution then leaves a neighbourhood of the cone close to the sine cone trajectory, giving the desired local desingularisation illustrated in Figure \ref{fig:local-passage-xC}.

\begin{figure}[H]
\centering
\scalebox{0.85}{
\begin{tikzpicture}[
    x=1.25cm,
    y=1cm,
    >={Latex[length=2.5mm]},
    every node/.style={font=\small}
]

\draw[thick,smooth]
plot coordinates {
(0,0.48)
(0.5,0.55)
(1.0,0.68)
(1.5,0.82)
(2.0,0.97)
(2.5,1.10)
(3.0,1.20)
(3.5,1.32)
(4.0,1.48)
(4.5,1.62)
(5.0,1.70)
(5.5,1.667)
(6.0,1.571)
(6.5,1.414)
(7.0,1.202)
(7.5,0.944)
(8.0,0.651)
(8.5,0.332)
(9.0,0)
};

\draw[thick,smooth]
plot coordinates {
(0,-0.48)
(0.5,-0.55)
(1.0,-0.68)
(1.5,-0.82)
(2.0,-0.97)
(2.5,-1.10)
(3.0,-1.20)
(3.5,-1.32)
(4.0,-1.48)
(4.5,-1.62)
(5.0,-1.70)
(5.5,-1.667)
(6.0,-1.571)
(6.5,-1.414)
(7.0,-1.202)
(7.5,-0.944)
(8.0,-0.651)
(8.5,-0.332)
(9.0,0)
};

\draw[thick] (0,0) ellipse [x radius=0.16,y radius=0.48];

\draw[dashed,gray] (2.55,-1.35) -- (2.55,1.35);
\draw[dashed,gray] (3.65,-1.48) -- (3.65,1.48);

\node[align=center] at (1.15,1.25)
  {scaled AC $C_7$ cap\\$\hat t=t/\epsilon$};

\node[align=center] at (3.10,2.00)
  {Cone region \\ $g_C=ds^2+s^2g_N$};

\node[align=center] at (6.00,2.10)
  {sine cone region\\
   $g_{\mathrm{sc}}=dt^2+\sin^2(t)g_N$};

\node[align=right,anchor=east] at (-0.25,0)
  {smooth singular orbit\\$\hat t=0$};

\node[align=left,anchor=west] at (9.15,0)
  {sine cone singularity\\$t=\pi$};

\draw[->,thick] (0.35,-2.15) -- (3.35,-2.15)
  node[midway,below]
  {$\hat t\to\infty$};

\draw[->,thick] (2.85,-2.65) -- (8.85,-2.65)
  node[midway,below]
  {$t:0\longrightarrow\pi$};

\node[align=center] at (3.10,-1.75)
  {$\epsilon\ll t\ll1$\\
   $1\ll \hat t\ll\epsilon^{-1}$};

\end{tikzpicture}}
\caption{Schematic diagram of the local desingularisation.  The AC $C_7$ metric replaces the singular end $t=0$ of
the sine cone.  In the overlap region, the $C_7$ metric is close to
its asymptotic cone and the sine cone metric is close to its tangent
cone.}
\label{fig:local-sine cone-desingularisation}
\end{figure}
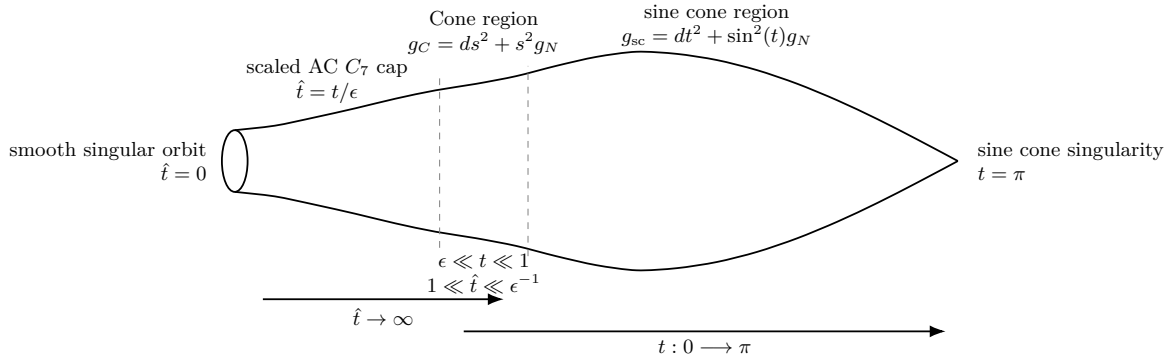

\subsection{The $C_7$ family of torsion-free $\G2$-structures}\label{subsection:C7-family}

We first describe the group action and the cohomogeneity-one torsion-free $\mathrm G_2$-structure of the ALC $C_7$ family obtained in \cite{FHN}. The action is closely related to the $\SU(2)\times \SU(2)$-action described in Section \ref{section:coho-1action}. For more details, see \cite[Proposition 4.1]{FHN}. We describe the one-parameter family of smooth $\SU(2) \times \SU(2) \times  \U(1)$–invariant torsion-free $\G2$–structures using conventions that differ from those of FHN by an embedding of the principal isotropy group in $\SU(2)\times \SU(2)$.

The group $\widetilde G$ defined in \eqref{eqn:G_tilde} acts on itself by left multiplication. For coprime positive integers $m,n$, recall the circle subgroup contained in the
maximal torus generated by the imaginary quaternion $i$,
$$
C^i_{m,-n}
=
\left\{
(e^{im\theta},e^{-in\theta}):
\theta\in\mathbb R
\right\}
\subset
\Sp(1)\times\Sp(1).
$$
Let $C^i_{m,-n}$ act on $\mathbb C$ by $
(e^{im\theta},e^{-in\theta})\cdot z=e^{2i(m+n)\theta}z$. Thus the slice representation has weight $2(m+n)$.  The corresponding
cohomogeneity-one manifold is
$$
M_{m,-n}
=
\widetilde G\times_{C^i_{m,-n}}\mathbb C,
$$
where for $h\in C^i_{m,-n}$, $
(g,z)\sim
\bigl(gh^{-1},h\cdot z\bigr)$. The action of $\widetilde G$ is induced by left multiplication $
g_0\cdot[g,z]=[g_0g,z]$.

At $z=0$, the isotropy group is $C^i_{m,-n}\subset T^2$. There is a fibration
$$
S^1
\simeq
T^2/C^i_{m,-n}
\longrightarrow
\widetilde G/C^i_{m,-n}
\longrightarrow
\widetilde G/T^2
\simeq
S^2\times S^2.
$$
Thus the singular orbit is a circle bundle over $S^2\times S^2$. For coprime $m,n$, it is diffeomorphic to $S^2\times S^3$.

For $z\neq0$, the isotropy group is the kernel of the slice
representation \[\Gamma_{m+n}=
\left\{(e^{im\theta},e^{-in\theta})\in C^i_{m,-n}: e^{2i(m+n)\theta}=1\right\}\cong
\mathbb Z_{2(m+n)}.\]
The group diagram is therefore
$$
\Gamma_{m+n}
\subset
C^i_{m,-n}
\subset
\widetilde G.
$$
For every pair of coprime positive integers $m,n$, the authors of \cite[Theorem C]{FHN} construct a one-parameter family $\{\varphi^{m,n}_\beta\}$ of local torsion-free $\mathrm G_2$-structures closing smoothly on the singular orbit $\widetilde G/C^i_{m,-n}$. In our conventions, for $\beta>0$, 
\begin{align}\label{eqn:phi_mn}
    \g2^{m,n}_\beta& = n^2r_0^3e^{135}-m^2r_0^3e^{246}+d(
    A^{m,n}_\beta (e^{34}-e^{56})-B^{m,n}_\beta e^{12}).
\end{align}
The functions $A^{m,n}_\beta, B^{m,n}_\beta$ satisfy the asymptotics
\begin{align*}
     A^{m,n}_\beta(t) &= r_0^2\beta t+O(t^3), \qquad
     B^{m,n}_\beta(t) = mn\,r_0^3 + \frac{\sqrt{mn}(m+n)r_0}{2\beta}\,t^2 + O(t^4).
\end{align*}
The constant $r_0$ fixes the scale of the metric. In \cite[Theorem D]{FHN}, the authors prove that for every coprime positive pair $(m,n)$, there is a unique critical value $\beta_{\mathrm{ac}}^{m,n}$ such that
\begin{itemize}
    \item for $\beta>\beta_{\mathrm{ac}}^{m,n}$ the metric is complete and ALC,
    \item for $\beta=\beta_{\mathrm{ac}}^{m,n}$ the metric is complete and AC,
    \item for $\beta<\beta_{\mathrm{ac}}^{m,n}$ the metric is incomplete.
\end{itemize}
The AC member is asymptotic, with rate $-3$, to the cone over $S^3\times S^3/\Z_{2(m+n)}$.

The $C_7$ family corresponds to the case $m=n=1$.  The group diagram for the $M_{1,1}$ manifold is
$$
\mathbb Z_4
\subset
C^i_{1,-1}
\subset
\Sp(1)\times\Sp(1).
$$
We denote the $\SU(2)^2\times \U(1)$-invariant family $\g2^{1,1}_\beta$ of torsion-free $\G2$-structures on $M_{1,1}$ by $\g2_\beta $ for convenience. From \eqref{eqn:phi_mn}
\begin{align}\label{eqn:phi_c7}
    \g2_\beta&= r_0^3e^{135}-r_0^3e^{246}+d(
    A_\beta (e^{34}-e^{56})-B_\beta e^{12}),
\end{align}
with
\begin{align}\label{eqn:AB_beta}
    A_\beta(t)&= r_0^2\beta t+O(t^3), \qquad
     B_\beta(t) = r_0^3 + \frac{r_0}{\beta}\,t^2 + O(t^4).
\end{align}

For the $\g2_\beta$ metric to have the required principal orbit $\SO(4)/\mathbb Z_2^2$,  we take a free finite quotient that enlarges the principal isotropy from $\mathbb Z_4$ to the quaternion group.

\begin{prop}[Extension of the finite action over the $C_7$ singular orbit]\label{prop:c7-finite-quotient}

Let $b=(j,j)$. Define $\sigma:M_{1,1}\longrightarrow M_{1,1},
\ \sigma[g,z]=[gb,\overline z]$ to be the free involution preserving the
torsion-free $C_7$ structure. The quotient $X_{C_7}:=M_{1,1}/\langle\sigma\rangle$ has group diagram
\begin{equation*}
 \widehat H=\langle a,b\rangle\cong Q_8
 \subset
 \widehat K=K_0\rtimes\langle b\rangle
 \subset \widetilde G.
\end{equation*}
\par\noindent
After composing the second factor with a quaternionic inner
automorphism, this is the diagram
$\Delta Q_8\subset\Delta\mathrm{Pin}(2)\subset S^3\times S^3$.
Consequently
the effective isotropy groups are
$\mathbb Z_2^2\subset O(2)\subset\SO(4)$, exactly as described in Section \ref{section:coho-1action}.
\end{prop}

\begin{proof}
    Since $C^i_{1,-1}$ acts on the normal disc by weight $4$, the original principal isotropy is
$$ H_0=\ker\rho =\langle a\rangle\cong\mathbb Z_4, \qquad a=u_{\pi/2}=(i,-i). $$
Now let $b=(j,j)$. Since $bu_\theta b^{-1}=u_{-\theta}$, the element $b$ normalises $C^i_{1,-1}$. The free involution
$$ \sigma[g,z]=[gb,\overline z] $$
therefore descends to $X_{C_7}$. After taking this quotient, $b$ becomes part of the principal isotropy, which is enlarged to
$$ \widehat H=\langle a,b\rangle. $$
The relations $a^2=b^2=(-1,-1), bab^{-1}=a^{-1}$ show that
\begin{align*}
    \widehat H &= \{\pm(1,1),\pm(i,-i),\pm(j,j),\pm(k,-k)\} \cong Q_8, \qquad
    \widehat K= \langle C^i_{1,-1},b\rangle = K^i_{1,-1}
\end{align*}
where $K^i_{1,-1}$ is the singular isotropy group for the $\SO(4)$ action.
\end{proof}

The quotient $ X_{C_7}$ is therefore a smooth AC torsion-free $\mathrm G_2$-manifold whose asymptotic cone is
$C\left(\SO(4)/\mathbb Z_2^2\right)$. This is exactly the cone appearing as the tangent cone of the sine cone at $t=0$.

\subsection{Scaled singular orbit and the limiting solution}\label{subsection:inner-limit}

Fix $r_0>0$ and define a subfamily $\varphi_{\epsilon,\beta}\coloneqq \xi_{a_\epsilon,0,c_\epsilon}$ of the family of solutions
$\xi_{a,b,c}$ defined in Theorem \ref{thm:complete-11} by
\begin{equation}\label{eq:small-bolt-family}
 b_\epsilon=0,\qquad
 a_\epsilon=r_0^3\epsilon^3,\qquad
 c_\epsilon=a_\epsilon+\lambda\beta^2r_0^4\epsilon^4.
\end{equation}

With the convention $p_1p_2p_3<0$, the initial square roots are chosen
so that
\begin{equation*}
 p_1(0)=0,\qquad
 p_2(0)=\beta r_0^2\epsilon^2,\qquad
 p_3(0)=-\beta r_0^2\epsilon^2.
\end{equation*}

Introduce the inner variable $\hat t=t/\epsilon$ and set
\begin{equation}\label{eq:inner-rescaling}
 \widehat q_i(\hat t)=\epsilon^{-3}q_i(\epsilon\hat t),
 \qquad
 \widehat p_i(\hat t)=\epsilon^{-2}p_i(\epsilon\hat t),
 \qquad
 \widehat\lambda=\epsilon\lambda.
\end{equation}
Consequently, if
\[
    d\varphi_{\epsilon,\beta}
       =\lambda *\varphi_{\epsilon,\beta},
\]
then
\[
    d\widehat\varphi_{\epsilon,\beta}
       =\widehat\lambda
          *\widehat\varphi_{\epsilon,\beta},
    \qquad
    \widehat\lambda=\epsilon\lambda.
\]
Thus the rescaled nearly parallel equations converge, as
$\epsilon\to0$, to the torsion-free equations obtained by setting
$\widehat\lambda=0$.

\begin{prop}[Torsion-free $\mathrm G_2$ limits under inner rescaling]
\label{prop:inner-limit}
Let $\widehat\varphi_{\epsilon,\beta}$ be the rescaling of
$\g2_{\epsilon,\beta}$ defined by \eqref{eq:inner-rescaling}.
If $\varphi_\beta$ is regular on $[0,R]$, then, for every $k\geq0$,
\[
    \widehat\varphi_{\epsilon,\beta}
        \longrightarrow \varphi_\beta
        \qquad\text{in }C^k([0,R])
\]
as $\epsilon\to0$. In particular, when
$\beta=\beta_{\mathrm{ac}}$, the convergence holds on every compact
interval $[0,R]$, and the inner limit is the AC $C_7$ metric in \eqref{eqn:phi_c7}.
\end{prop}

\begin{proof}
Under the rescaling of \eqref{eq:inner-rescaling} the rescaled parameters at the singular orbit are
\[
   \widehat a_\epsilon
       =\epsilon^{-3}a_\epsilon=r_0^3,
   \qquad
   \widehat b_\epsilon=0, \qquad \widehat c_\epsilon
       =\epsilon^{-3}c_\epsilon
       =r_0^3+\widehat\lambda\beta^2r_0^4.
\]
and
\[
   \frac{\widehat c_\epsilon-\widehat a_\epsilon}
        {\widehat\lambda}
       =\beta^2r_0^4.
\]
The choice of square roots determined by $p_1p_2p_3<0$ gives
\[
   \widehat p_1(0)=0,
   \qquad
   \widehat p_2(0)=\beta r_0^2,
   \qquad
   \widehat p_3(0)=-\beta r_0^2.
\]
Under this scaling, the expressions in the first-order terms of the Taylor expansion of $\widehat\varphi_{\epsilon,\beta}$ in \eqref{eqn:y_intial_11} have removable singularities. Indeed,
with the chosen branch,
\[
   \sqrt{\frac{\widehat c_\epsilon-\widehat a_\epsilon}
                    {\widehat\lambda}}
       =\beta r_0^2,
   \qquad
   \sqrt{\widehat\lambda
          (\widehat c_\epsilon-\widehat a_\epsilon)}
       =\widehat\lambda\beta r_0^2.
\]
Substitution into the smooth initial data \eqref{eqn:y_intial_11}
therefore gives an initial point depending analytically on
$(\widehat\lambda,\beta)$.

Since $\widehat b_\epsilon=0$, Theorem \ref{thm:complete-11} shows that the rescaled solutions lie in the distinguished $\U(1)$-invariant locus. At $\widehat\lambda=0$, the expansion
is therefore
\[
 \varphi
   =r_0^3e^{135}-r_0^3e^{246}
     +d\left(
        A(e^{34}-e^{56})-Be^{12}
       \right),
\]
with
\[
       A(0)=0,\qquad A'(0)=r_0^2\beta,
       \qquad B(0)=r_0^3,\qquad B'(0)=0.
\]
These are precisely the initial conditions defining the FHN
torsion-free solution $\varphi_\beta$. Hence uniqueness of the
torsion-free singular initial-value problem identifies the
$\widehat\lambda=0$ solution with $\varphi_\beta$.

For completeness, let $Y_{\widehat\lambda,\beta}$ denote the
regularised variables used in Section~\ref{sec:smoothness}. Their
equations have the form
\[
  \frac{dY}{d\widehat t}
    =
    \frac{1}{\widehat t}
       M_{-1}(Y;\widehat\lambda,\beta)
       +
       M_0(\widehat t,Y;\widehat\lambda,\beta),
\]
and the initial condition
\[
   Y(0)=Y_0(\widehat\lambda,\beta)
\]
depends analytically on $(\widehat\lambda,\beta)$. For the
$(1,-1)$ singular orbit, \eqref{eqn:det_11} gives
\[
 \det\left(
      n\Id-d_{Y_0}M_{-1}
     \right)
   =n^3(n+1)(n+2)^4.
\]
This is nonzero for every positive integer $n$. The solution is therefore unique given $\widehat\lambda,\beta$. Its Taylor
coefficients are determined recursively by equations of the form
\[
 \left(n\Id-d_{Y_0}M_{-1}\right)Y_n
     =\mathcal F_n(Y_1,\ldots,Y_{n-1}),
\]
and depend analytically on $(\widehat\lambda,\beta)$. Hence there exists  $\widehat t_0>0$, independent of sufficiently small
$\widehat\lambda$, such that
\[
  Y_{\widehat\lambda,\beta}
      \longrightarrow Y_{0,\beta}
      \qquad\text{in }C^k([0,\widehat t_0])
\]
for every $k\geq0$. Since
$\widehat\lambda=\epsilon\lambda$, this proves
\[
  \widehat\varphi_{\epsilon,\beta}
      \longrightarrow\varphi_\beta
      \qquad\text{in }C^k([0,\widehat t_0]).
\]

We now extend the convergence away from the singular orbit. Choose
$0<\widehat t_0<R$ such that $\varphi_\beta$ is regular on
$[0,R]$. On $[\widehat t_0,R]$ the orbit is principal, so the
regularised singular equations reduce to an ordinary nonsingular
ODE. Ordinary continuous dependence extends the convergence to every
fixed interval $[0,R]$ on which the limiting torsion-free solution is
regular.

Applying the same argument to the differentiated equations, or
differentiating the ODE inductively, gives convergence in
$C^k([\widehat t_0,R])$ for every $k\geq0$. Combining this with the
convergence near the singular orbit proves
\[
  \widehat\varphi_{\epsilon,\beta}
     \longrightarrow\varphi_\beta
     \qquad\text{in }C^k([0,R]).
\]

Finally, for $\beta=\beta_{\mathrm{ac}}$, the FHN solution
$\varphi_{\beta_{\mathrm{ac}}}$ is complete and asymptotically
conical. It is the AC member of the $C_7$ family. Hence it is regular
on every compact interval $[0,R]$, and the preceding convergence
holds for arbitrary $R>0$.
\end{proof}

The proposition identifies the geometry near the shrinking singular
orbit with the AC $C_7$ model.  To connect this scaled singular orbit limit with the
sine cone solution away from the singular orbit, we study how both solutions approach their common cone.  We therefore turn to the linearisation of the equations \eqref{eqns:gen_ng2} at the cone solution. Although the inner and outer limits have the same cone, this alone does
not guarantee that they can be matched.  We must understand how the $C_7$ solution approaches the cone as $\hat t\to\infty$ and how the sine cone solution approaches it as $t\to0$. The linearisation at the cone determines these directions and allows us to verify that the matching is transverse.

\subsection{The cone solution and its linearisation}\label{subsection:cone-fixed-point}

For a rescaled solution put $ \hat t=e^s$ and
\begin{equation}\label{eq:log-coordinates}
 Q_i(s)=e^{-3s} \widehat{q_i}(e^s),\qquad
 P_i(s)=e^{-2s} \widehat{p_i}(e^s),\qquad
 z(s)=\widehat\lambda e^s.
\end{equation}
After eliminating $P_i$ by the three algebraic constraints \eqref{eq:constraint-set}, the rescaled equations \eqref{eqns:gen_ng2} and $\frac{\d z}{\d s} = z$ become an autonomous nine-dimensional system
\begin{equation}\label{eq:log-system}
 \frac{\mathrm d Y}{\mathrm ds}=\mathcal F(Y),\qquad
 Y=(Q_1,\ldots,Q_8,z).
\end{equation}
If $(\omega,\psi_+,\psi_-)$ is the homogeneous nearly Kähler structure on $S^3\times S^3$, its torsion-free $\G2$-cone has the form, up to conventions,
\begin{align*}
    \varphi_C&=t^2dt\wedge\omega+t^3\psi_+.
\end{align*}
Consequently, its coefficients are exactly homogeneous
\begin{align*}
    p_i^C(t)&=t^2P_i^{NK},\qquad q_i^C(t)=t^3Q_i^{NK}, \ \text{and} \qquad \lambda_C=0.
\end{align*}
Therefore, in normalized variables,
\begin{align*}
    Q_i(s)&=Q_i^{NK},\qquad z(s)=0,
\end{align*}
independently of $s$ and the homogeneous nearly K\"ahler cone is a fixed point $x_C$ of \eqref{eq:log-system} in the slice $z=0$. The nearly-parallel sine cone trajectory is asymptotic to this equilibrium as $s\to-\infty$. With the normalisation $\lambda=4$ in the sine cone formula
of \Cref{sec:open-orbit}, 
\begin{equation}\label{eq:cone-fixed-point}
 x_C=\frac1{54}
 (0,\sqrt3,-\sqrt3,-\sqrt3,\sqrt3,\sqrt3,-\sqrt3,0,0).
\end{equation}

The $\U(1)$-invariant locus $\mathcal P_{\U(1)}$ is an invariant
six-dimensional submanifold of \eqref{eq:log-system}.  Direct
linearisation of the reduced six-dimensional system at $x_C$ gives
\begin{equation}\label{eq:u1-spectrum}
 \operatorname{spec}
 \left(D\mathcal F(x_C)|_{T_{x_C}\mathcal P_{\U(1)}}\right)
 =\left\{1,-1,-3,-3,
 \frac{-7+\sqrt{145}}2,
 \frac{-7-\sqrt{145}}2\right\}.
\end{equation}

Recall that a fixed point $p$ of a dynamical system $x' = F(x)$
is called hyperbolic if the real parts of all eigenvalues of the linearisation $DF(p)$ of the system at the
fixed point are non-zero. Suppose that $DF(p)$ has $k$ eigenvalues with negative real part
and $n-k$ eigenvalues with positive real part, counted with algebraic
multiplicity.  The stable manifold theorem then gives local invariant
submanifolds $W^s_{\mathrm{loc}}(p)$ and
$W^u_{\mathrm{loc}}(p)$, of dimensions $k$ and $n-k$,
respectively.  Their tangent spaces at $p$ are the stable and
unstable spectral subspaces of $DF(p)$.  Trajectories in
$W^s_{\mathrm{loc}}(p)$ converge to $p$ in forward time, whereas
trajectories in $W^u_{\mathrm{loc}}(p)$ converge to $p$ in backward time \cite[Chapter 2.7]{Per96}.  Moreover, the
Hartman--Grobman theorem states that, in a neighbourhood of $p$, the nonlinear flow is topologically conjugate to the flow of the linearised system $\dot{y}=DF(p)y$ \cite[Chapter 2.8]{Per96}.  Thus the signs of the real parts of the eigenvalues determine the qualitative behaviour of the flow near a hyperbolic equilibrium.

Since all eigenvalues of $D\mathcal F(x_C)$ are real and nonzero, the conical fixed point $x_C$ is hyperbolic. The stable manifold theorem therefore yields smooth local stable and unstable manifolds tangent at $x_C$ to the sums of the negative and positive eigenspaces, respectively.
 
For the unreduced nine-dimensional algebraic system the additional
eigenvalues are $-12$ and a second copy of each of the last two roots. Let $n=D\mathcal H(x_C)$ denote the normal covector to the volume-normalisation constraint \eqref{eqn:normalization}; then $nD\mathcal F(x_C)=-12n$, providing an independent check of the
linearisation.  The tangent $v_{\mathrm{sc}}$ to the sine cone orbit
satisfies
\begin{equation*}
 D\mathcal F(x_C)v_{\mathrm{sc}}=v_{\mathrm{sc}}.
\end{equation*}

We record more precisely which infinitesimal deformations remain after fixing the metric of $X_{C_7}$.  This identification is essential in the proof that the matching is transverse.

\begin{lemma}
\label{lem:fixed-r0-modes}
Let $\mathcal T=\mathcal P_{\U(1)}\cap\{z=0\}$ be the torsion-free $\U(1)$-invariant locus. At $x_C$, one has
\begin{equation}\label{eq:torsion-free-splitting}
T_{x_C}\mathcal T
=E_{-1}\oplus E_{-3}^{(1)}\oplus E_{-3}^{(2)}
\oplus E_{\gamma_-}\oplus E_{\gamma_+},
\qquad
\gamma_\pm=\frac{-7\pm\sqrt{145}}2.
\end{equation}
The eigenspace $E_{-1}$ consists of infinitesimal deformations induced
by translating the radial coordinate $\hat t$, while
$E_{-3}^{(1)}\oplus E_{-3}^{(2)}$ consists of infinitesimal
deformations varying the conserved coefficients $q_1$ and $q_8$.
After fixing the radial origin at $\hat t=0$ and fixing the
singular-orbit data $q_1=r_0^3, q_8=-r_0^3$, these deformation directions are excluded. In the resulting
two-dimensional reduced phase space, the stable and unstable tangent
spaces at the cone are respectively
\begin{equation}\label{eqn:fixed-r0-modes}
E^s_{r_0}(x_C)=E_{\gamma_-},
\qquad
E^u_{r_0}(x_C)=E_{\gamma_+}.
\end{equation}
Moreover, the infinitesimal variation obtained by differentiating the
one-parameter family of AC ends in
\cite[Proposition~5.3(ii)]{FHN} with respect to its parameter lies in
$E_{\gamma_-}$. Thus this family is tangent to
$E_{\gamma_-}$ at the cone.
\end{lemma}
\begin{proof}
Since $z'=z$, the infinitesimal deformation in the $z$-direction
has eigenvalue $1$. By \eqref{eq:u1-spectrum}, this eigenvalue is
simple, and its eigenspace is generated by the tangent
$v_{\mathrm{sc}}$ to the sine cone trajectory. Since
$v_{\mathrm{sc}}$ has nonzero $z$-component, it is transverse to
the torsion-free locus $\{z=0\}$. Restricting the linearisation to $\mathcal T$ therefore removes this eigenspace and gives \eqref{eq:torsion-free-splitting}.

We next identify the geometric meaning of the remaining infinitesimal deformations. On the cone, $\widehat q_i^C(\hat t)=Q_i^C\hat t^3$. Translating the radial coordinate by $\hat t\mapsto\hat t+\delta$ gives
$$
e^{-3s}\widehat q_i^C(e^s+\delta)
 =
Q_i^C(1+\delta e^{-s})^3
 =
Q_i^C+3\delta Q_i^C e^{-s}
 +O(\delta^2e^{-2s}).
$$

Thus the infinitesimal deformation induced by translating the radial coordinate has leading order $e^{-s}$, and belongs to $E_{-1}$.

In the torsion-free system, $q_1$ and $q_8$ are conserved under the evolution equations. Infinitesimal variations of these constants give
$$
\delta Q_1(s)=\delta q_1\,e^{-3s},
\qquad
\delta Q_8(s)=\delta q_8\,e^{-3s},
$$
hence these two independent infinitesimal deformations span
$E_{-3}^{(1)}\oplus E_{-3}^{(2)}$. Fixing the radial origin at $\hat t=0$ and fixing $q_1=r_0^3,q_8=-r_0^3$ excludes the infinitesimal deformations in
$E_{-1}\oplus E_{-3}^{(1)}\oplus E_{-3}^{(2)}$.

It remains to identify the infinitesimal deformation determined by the one-parameter family of AC ends. Let $\chi$ denote the parameter
in \cite[Proposition~5.3(ii)]{FHN}. In terms of the FHN coefficient functions $A$ and $B$, their asymptotic expansion satisfies
$$
\frac{54}{\sqrt3}\hat t^{-3}(B-A)
 =
\chi\hat t^{-\nu_\infty}+O(\hat t^{-12}),
\qquad
\nu_\infty=\frac{7+\sqrt{145}}2.
$$
Differentiating with respect to $\chi$ therefore produces a nonzero infinitesimal deformation whose leading term is
$$
\hat t^{-\nu_\infty}
 =
e^{-\nu_\infty s}
 =
e^{\gamma_-s},
\qquad
\gamma_-=-\nu_\infty.
$$
Since $E_{\gamma_-}$ is one-dimensional, this infinitesimal variation spans $E_{\gamma_-}$. Finally, since $\gamma_-<0<\gamma_+$ after the radial origin and the singular-orbit data have been fixed, $E_{\gamma_-}$ and $E_{\gamma_+}$ are respectively the
stable and unstable tangent spaces at $x_C$. 
\end{proof}

The desingularisation is a two-scale matching problem. On the inner
scale $\widehat t=t/\epsilon$, the solution converges to the AC
$C_7$ metric, while on the outer scale $t$ it should converge to
the sine cone. These two limiting solutions are not compared directly.
Instead, both are compared with their common torsion-free cone in the
intermediate region
\[
        \epsilon\ll t\ll1,
        \qquad\text{equivalently}\qquad
        1\ll\widehat t\ll\epsilon^{-1}.
\]
In the scale-invariant variables, this common cone is the fixed point
$x_C$.

The rate at which the $C_7$ metric approaches $x_C$ is essential
for controlling this overlap. More precisely,
\[
   Y_{C_7}(s)-x_C=O(e^{-3s})
   =O(\widehat t^{-3}),
   \qquad \widehat t=e^s\longrightarrow\infty.
\]
The exponent $-3$ arises from the two conserved coefficients
$q_1$ and $q_8$. Since the scale-invariant variables are
$Q_i=\widehat t^{-3}q_i$, fixing the nonzero values
\[
        q_1=r_0^3,\qquad q_8=-r_0^3
\]
produces precisely an $O(\widehat t^{-3})$ contribution. Thus the
$-3$-eigenspaces describe the leading approach of the $C_7$
trajectory to the cone. Once
the scale $r_0$ has been fixed, their coefficients are fixed as well.

On the other hand, the sine cone approaches the same fixed point with
\[
       Y_{\mathrm{sc}}(\log t)-x_C=O(t)
       \qquad\text{as }t\longrightarrow0.
\]
Evaluating the two estimates at $t=\epsilon R$, or equivalently
$\widehat t=R$, gives the two principal errors
\[
       O(R^{-3})
       \qquad\text{and}\qquad
       O(\epsilon R).
\]
Consequently one may choose $R\to\infty$ while
$\epsilon R\to0$, so that both errors tend to zero. The following lemma makes the overlap estimate precise.

\begin{lemma}[Overlap near the cone]\label{lem:overlap}
Fix $R>1$ sufficiently large.  For $\epsilon>0$ small and
$\beta$ close to $\beta_{\mathrm{ac}}$, the solution $Y_{\epsilon,\beta}$ of \eqref{eq:log-system} is
defined up to $\widehat t=R$. Away from the singular orbit at $t=\epsilon R$, 
$$
\left|
 Y_{\epsilon,\beta}(\log R)
-
Y_{\mathrm{sc}}(\log(\epsilon R))
\right|
\leq
CR^{-3}+\rho_R(\epsilon,\beta)+C\epsilon R.
$$
Thus the two solutions are arbitrarily close if $R$ is 
large and  $\epsilon$,
$\lvert\beta-\beta_{\mathrm{ac}}\rvert$ are small.
\end{lemma}

\begin{proof}
By Proposition~\ref{prop:inner-limit}, for every fixed $R$ as $(\epsilon,\beta)\to(0,\beta_{\mathrm{ac}})$,
$$
Y_{\epsilon,\beta}(\log R)
\longrightarrow
Y_{\beta_{\mathrm{ac}}}(\log R).
$$
We denote the difference between these two terms by
$\rho_R(\epsilon,\beta)$. The solution $Y_{\beta_{\mathrm{ac}}}=Y_{C_7}$ approaches the cone
with rate $-3$.  Hence
$$
\left|
Y_{\beta_{\mathrm{ac}}}(\log R)-x_C
\right|
\leq CR^{-3}.
$$
Thus we get
\begin{equation*}
\left|
Y_{\epsilon,\beta}(\log R)-x_C
\right|
\leq
CR^{-3}+\rho_R(\epsilon,\beta),
\end{equation*}
where $\rho_R(\epsilon,\beta)\longrightarrow 0$ as
$(\epsilon,\beta)\longrightarrow (0,\beta_{\mathrm{ac}})$.
The sine cone approaches the same cone as $t\to0$, and
$$
\left|Y_{\mathrm{sc}}(\log t)-x_C\right|\leq Ct.
$$
Setting $t=\epsilon R$ gives
$$
\left|
Y_{\mathrm{sc}}(\log(\epsilon R))-x_C
\right|
\leq C\epsilon R.
$$
The triangle inequality now gives the claimed overlap estimate.
\end{proof}

There are two distinct pieces of asymptotic information here. In the
full scale-invariant phase space, the reference $C_7$ trajectory
approaches $x_C$ at rate $-3$, because of the fixed conserved
coefficients $q_1$ and $q_8$. After fixing the radial origin, the
scale $r_0$, and hence these conserved coefficients, the remaining
torsion-free deformation space has one stable direction
$E_{\gamma_-}$ and one unstable direction $E_{\gamma_+}$. The
variation of the AC family with respect to $\beta$ is tangent to
$E_{\gamma_-}$.

For the nearly parallel problem there are two outgoing directions:
$E_1$, tangent to the sine cone trajectory, and
$E_{\gamma_+}$, which moves the solution away from the sine cone.
A general perturbation of the $C_7$ solution therefore leaves a
neighbourhood of the cone with a nonzero $E_{\gamma_+}$-component.
The role of the parameter $\beta$ is to eliminate precisely this
component. Once it has been removed, the only remaining outgoing
direction is $E_1$, and the solution follows the sine cone
trajectory on the outer scale.

\subsection{Transverse matching and local desingularisation}

The spectrum \eqref{eq:u1-spectrum} contains two positive roots,
$1$ and $\gamma_+=\frac{-7+\sqrt{145}}2$.
The first is the sine cone direction, while the second is an unwanted
outgoing direction. Hence dimension counting alone does not show that the trajectory of
$\widehat{\varphi}_{\epsilon,\beta}$ leaves a neighbourhood of
$x_C$ in the sine cone direction.

Let $Y_\beta$ denote the torsion-free FHN trajectory and set $ V=\left.\partial_\beta Y_\beta\right|_{\beta=\beta_{\mathrm{ac}}} $. Differentiating $ \dot Y_\beta=\mathcal F_0(Y_\beta) $ with respect to $\beta$ gives
$$ \dot V=D\mathcal F_0(Y_{\beta_{\mathrm{ac}}})V. $$
The asymptotic condition for $V$ at the singular orbit is obtained by differentiating the $\beta$-dependent Taylor expansion of the FHN family. From Lemma \ref{lem:fixed-r0-modes} in the fixed-$r_0$, torsion-free, $\U(1)$-invariant reduced space, the two eigenvalues are $\gamma_-<0<\gamma_+$. Let $v_\pm$ be the generators of $E_{\gamma_\pm}$, respectively. Then 
\begin{align*}
    V(s)&=\mathfrak{a}e^{\gamma_+ s} v_++ \mathfrak{b} e^{\gamma_-s}v_{-}+ \textup{l.o.t}.
\end{align*}
Choose a left eigenvector $\ell_+$ for $\gamma_+$ and normalise it by
$\ell_+(v_+)=1$. Since $\ell_+(v_-)=0$, the $E_{\gamma_+}$ coefficient is
\begin{equation}\label{eq:exchange-coefficient}
\mathfrak a(s):=
e^{-\gamma_+s}\ell_+(V(s)),
\qquad
\mathfrak a:=
\lim_{s\to\infty}\mathfrak a(s).
\end{equation}
where the limit is taken in the asymptotic coordinates. Since the $C_7$ solution approaches $x_C$ with rate $-3$ and $|V(s)|=O(e^{\gamma_+s})$, the limit exists because
\begin{align*}
    \mathfrak{a}'(s)&=e^{-\gamma_+s}(D\mathcal{F}_0(Y_{\beta_{ac}}(s))-D\mathcal{F}_0(x_C))V(s)=O(e^{-\gamma_+s} e^{-3s}e^{\gamma_+s})=O(e^{-3s})
\end{align*}
 which is integrable on $[s_0,\infty)$ for $s_0>0$.  Its vanishing is independent of the normalisation of $\ell_+$.

\begin{prop} \label{prop:exchange-nonzero}
For the AC member of the $C_7$ family $ \mathfrak a\neq0.$
\end{prop}
\begin{proof}
Since $A_\beta$ is increasing, we may use $x=A_\beta$
as the independent variable.  Write $B_\beta$ as a function of $x$ and define
$$
W(x)=\left.\partial_\beta B_\beta(x)
\right|_{\beta=\beta_{\mathrm{ac}}}.
$$
The expansion at the singular orbit is
$$
B_\beta(x)
=
r_0^3+\frac{x^2}{\beta^3r_0^3}+O(x^4).
$$
Differentiating with respect to $\beta$ gives
$$
W(x)
=
-\frac{3x^2}{\beta_{\mathrm{ac}}^4r_0^3}
+O(x^4), \qquad W'(x)
=
-\frac{6x}{\beta_{\mathrm{ac}}^4r_0^3}
+O(x^3).
$$
Therefore both $W(x)$ and $W'(x)$ are strictly negative for all
sufficiently small $x>0$. For $h>0$ sufficiently small, consider the two solutions
$B_{\beta_{\mathrm{ac}}+h}$ and $B_{\beta_{\mathrm{ac}}}$,
parametrised by $x$. The smooth dependence of the singular initial-value problem on
$\beta$, together with the expansion above, gives
\begin{align*}
 B_{\beta_{\mathrm{ac}}+h}(x)
 -B_{\beta_{\mathrm{ac}}}(x)
 &=
 -\frac{3h}{\beta_{\mathrm{ac}}^4r_0^3}x^2
 +O(h^2x^2)+O(hx^4),\\
 \frac{d}{dx}B_{\beta_{\mathrm{ac}}+h}(x)
 -\frac{d}{dx}B_{\beta_{\mathrm{ac}}}(x)
 &=
 -\frac{6h}{\beta_{\mathrm{ac}}^4r_0^3}x
 +O(h^2x)+O(hx^3).
\end{align*}
Consequently, for some sufficiently small $x_0>0$ and all sufficiently
small $h>0$,
\begin{align}\label{eq:initial-beta-ordering}
 B_{\beta_{\mathrm{ac}}+h}(x_0)
 &<
 B_{\beta_{\mathrm{ac}}}(x_0),\quad 
 \frac{d}{dx}B_{\beta_{\mathrm{ac}}+h}(x_0)
 <
 \frac{d}{dx}B_{\beta_{\mathrm{ac}}}(x_0).
\end{align}

For $h>0$, the solution with parameter
$\beta_{\mathrm{ac}}+h$ is an ALC member of the
$C_7$ family, while the solution with parameter
$\beta_{\mathrm{ac}}$ is the AC member. Both solutions remain in the
comparison region
\[
 B>\max\{x,r_0^3\},\qquad F(x,B)>0,
 \qquad
 \frac{dB}{dx}>0.
\]
Hence \cite[Lemma~7.9(ii)]{FHN}, applied with
$ B_1=B_{\beta_{\mathrm{ac}}+h}$ and $B_2=B_{\beta_{\mathrm{ac}}}$
shows that the two strict inequalities in
\eqref{eq:initial-beta-ordering} are preserved for every $x\geq x_0$.
Together with the local expansion near $x=0$, this gives
\[
 B_{\beta_{\mathrm{ac}}+h}(x)
 <
 B_{\beta_{\mathrm{ac}}}(x),
 \qquad
 \frac{d}{dx}B_{\beta_{\mathrm{ac}}+h}(x)
 <
 \frac{d}{dx}B_{\beta_{\mathrm{ac}}}(x)
\]
for every $x>0$. Dividing these inequalities by $h>0$ and letting $h\to 0$, gives for every $x>0$,
\begin{equation}\label{eq:W-signs}
 W(x)\leq0,
 \qquad
 W'(x)\leq0.
\end{equation}

Suppose, for a contradiction, that $\mathfrak a=0$.
The radial origin and the scale $r_0$ have already been fixed.  Thus the deformation spaces with eigenvalues $-1$ and $-3$ are absent.  Moreover,
$V$ lies in the torsion-free locus $z=0$.  After removing the
$E_{\gamma_+}$, the only remaining possible asymptotic deformation direction is
therefore $E_{\gamma_-}$. Since $B_\beta=e^{3s}\overline B_\beta$, an $E_{\gamma_-}$-variation of the scale-invariant variable $\overline B_\beta$ gives a variation of $B_\beta$ of order $e^{(3+\gamma_-)s}$. Consequently, for some constant $\widetilde C$,
$$
W(s)
=
\widetilde C e^{(3+\gamma_-)s}
+o\bigl(e^{(3+\gamma_-)s}\bigr).
$$
On the other hand, $A(s)=A_0e^{3s}(1+o(1))$ for some $A_0>0$.  Since $x=A(s)$, the preceding expansion becomes
$$
W(x)=\tilde{c}x^\rho+o(x^\rho),
$$
where the constant $A_0$ has been absorbed into $\tilde c$, and
$\rho=\frac{3+\gamma_-}{3}=-\frac{1+\sqrt{145}}6<0$.
Differentiating with respect to $x$ gives
$$
W'(x)
=
\rho  \tilde c  x^{\rho-1}
+o(x^{\rho-1})
$$
as $x\to\infty$.
Since $\rho<0$, if $\tilde c\neq0$, then $W$ and $W'$ have opposite signs.  This contradicts \eqref{eq:W-signs}.

It remains to consider $\tilde c=0$.  In this case the
$E_{\gamma_-}$-coefficient also vanishes.  Thus $V$ has no component in either $E_{\gamma_+}$ or $E_{\gamma_-}$. Since the reduced variational equation is two-dimensional and its
coefficient matrix converges to $D\mathcal F_0(x_C)$ at the integrable
rate $O(e^{-3s})$, every solution is uniquely determined by its
$E_{\gamma_+}$- and $E_{\gamma_-}$-coefficients. Since both
coefficients vanish, we obtain $V\equiv0$. This is impossible because
the expansion at the singular orbit shows that $W(x)$ is not
identically zero.

\noindent
Both possibilities lead to a contradiction.  Hence $\mathfrak a\neq0.$
\end{proof}

The non-vanishing of $\mathfrak a$ says that varying $\beta$ moves the
torsion-free $C_7$ trajectory transversely in the
$E_{\gamma_+}$-direction.  We now formulate the nonlinear matching
condition that detects this direction for the rescaled nearly parallel
solutions.

\begin{lemma}[Local passage near the cone]\label{lem:local-passage}
Consider the reduced nearly parallel system near $x_C$, after fixing
$r_0$ and the radial origin.  Suppose that
\[
 T_{x_C}\mathcal P_{r_0}
 =E_{\gamma_-}\oplus E_1\oplus E_{\gamma_+},
 \qquad
 \gamma_-<0<1<\gamma_+,
\]
and that the sine cone orbit is tangent to $E_1$.

Choose sufficiently small incoming and outgoing sections
$\Sigma_{\mathrm{in}}$ and $\Sigma_{\mathrm{out}}$, where the stable
$C_7$ orbit meets $\Sigma_{\mathrm{in}}$ at $P_{\mathrm{in}}$.  Then
there are a neighbourhood $\mathcal O$ of $P_{\mathrm{in}}$ in
$\Sigma_{\mathrm{in}}$ and a $C^1$ function
\[
 \Theta:\mathcal O\longrightarrow\mathbb R
\]
with the following properties:
\begin{enumerate}
\item $\Theta(P_{\mathrm{in}})=0$;
\item up to multiplication by a nonzero constant,
      $D\Theta(P_{\mathrm{in}})$ is the $E_{\gamma_+}$-coefficient
      transported to $\Sigma_{\mathrm{in}}$;
\item if $P_j\to P_{\mathrm{in}}$, $z(P_j)>0$, and
      $\Theta(P_j)=0$, then the trajectories through $P_j$ meet
      $\Sigma_{\mathrm{out}}$, and their exit points converge to the
      intersection of the sine cone orbit with
      $\Sigma_{\mathrm{out}}$.
\end{enumerate}
\end{lemma}

\begin{proof}
Choose coordinates $(x,z,w)$ adapted to
$E_{\gamma_-}\oplus E_1\oplus E_{\gamma_+}$.  We may arrange that the
$C_7$ orbit is contained in $\{z=w=0\}$, the sine cone orbit is
$\{x=w=0,\ z>0\}$, and
\[
 \dot x=\gamma_-x+f_-(x,z,w),\qquad
 \dot z=z,\qquad
 \dot w=\gamma_+w+f_+(x,z,w),
\]
where $f_\pm$ vanish to second order and $f_\pm(0,z,0)=0$.  Take
\[
 \Sigma_{\mathrm{in}}=\{x=x_{\mathrm{in}}\},
 \qquad
 \Sigma_{\mathrm{out}}=\{z=z_{\mathrm{out}}\},
\]
with $x_{\mathrm{in}}$ and $z_{\mathrm{out}}$ fixed and small.  Then
$P_{\mathrm{in}}=(x_{\mathrm{in}},0,0)$.

For an incoming point $(x_{\mathrm{in}},z_0,w_0)$ with $z_0>0$, the
time at which its trajectory reaches $\Sigma_{\mathrm{out}}$ is
\[
 S(z_0)=\log\frac{z_{\mathrm{out}}}{z_0}.
\]
Variation of constants in the $w$-equation gives
\[
 w(S)=e^{\gamma_+S}
 \left(
 w_0+\int_0^S e^{-\gamma_+\tau}
 f_+(x(\tau),z(\tau),w(\tau))\,d\tau
 \right).
\]
Thus the boundary condition $w(S(z_0))=0$ is equivalent to
\[
 w_0=-\int_0^{S(z_0)}e^{-\gamma_+\tau}
 f_+(x(\tau),z(\tau),w(\tau))\,d\tau.
\]
Standard contraction estimates, uniform as $z_0\downarrow0$, give a
unique solution $w_0=h(z_0)$.  The function $h$ extends as a $C^1$
function to $z_0=0$ and satisfies $h(0)=0$.  Define
\[
 \Theta(x_{\mathrm{in}},z,w)=w-h(z).
\]
Then $\Theta=0$ is precisely the nonlinear matching condition that
removes the $E_{\gamma_+}$ exit.  At $P_{\mathrm{in}}$ its differential
has a nonzero $dw$-component and is therefore, up to scale, the
transported $E_{\gamma_+}$-coefficient.

For a trajectory satisfying $\Theta=0$, we have $w(S)=0$, while
\[
 x(S)=O\!\left(z_0^{-\gamma_-}\right)\longrightarrow0.
\]
Its exit point consequently converges to
$(0,z_{\mathrm{out}},0)$, the intersection of the sine cone orbit with
$\Sigma_{\mathrm{out}}$.
\end{proof}

The preceding lemma turns the local matching problem into a scalar
shooting problem.  Indeed, the equation $\Theta=0$ is precisely the
nonlinear condition that eliminates the $E_{\gamma_+}$-component
during passage near the cone.  The rescaled singular-orbit solutions
meet $\Sigma_{\mathrm{in}}$ at points depending on
$(\epsilon,\beta)$.  Composing this intersection map with $\Theta$
therefore defines a scalar matching function
$\mathcal A(\epsilon,\beta)$.

At $(\epsilon,\beta)=(0,\beta_{\mathrm{ac}})$, the incoming point
lies on the $C_7$ trajectory, and hence
$\mathcal A(0,\beta_{\mathrm{ac}})=0$.  The non-vanishing of the
exchange coefficient $\mathfrak a$ shows that varying $\beta$
changes this matching condition transversely.  The implicit function
theorem can therefore be used to choose $\beta=\beta(\epsilon)$ so
that

$$
\mathcal A(\epsilon,\beta(\epsilon))=0.
$$

For this choice, the local passage lemma identifies the outgoing limit
with the sine cone trajectory, while the rescaled inner solutions
converge to the AC $C_7$ structure (see Figure \ref{fig:local-passage-xC}).  This gives the two limiting
statements in the following theorem.
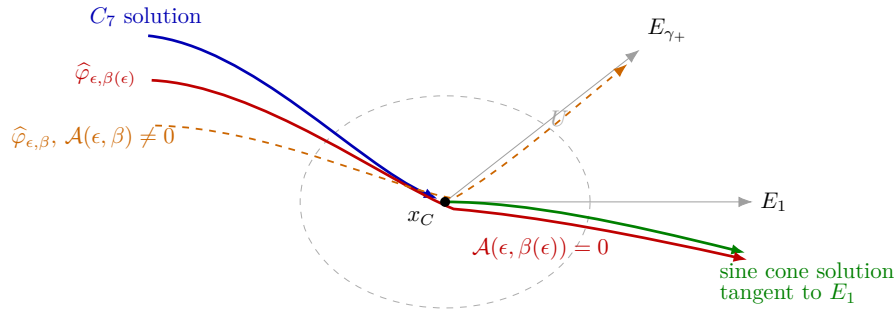
\begin{figure}[H]
\centering
\scalebox{0.85}{
\begin{tikzpicture}[
  x=1.1cm,
  y=1.1cm,
  >={Latex[length=2.2mm]},
  reference/.style={very thick},
  selected/.style={very thick},
  untuned/.style={thick,dashed},
  direction/.style={thin,gray!75,->}
]

\coordinate (xC) at (0,0);

\draw[dashed,gray!65] (xC) ellipse (2.05 and 1.5);
\node[gray!70] at (1.62,1.18) {$U$};

\draw[direction] (xC) -- (4.35,0)
  node[right,black] {$E_1$};

\draw[direction] (xC) -- (2.75,2.15)
  node[above right,black] {$E_{\gamma_+}$};

\draw[reference,blue!70!black,->]
  (-4.2,2.35)
  .. controls (-2.75,2.20) and (-1.25,0.55) ..
  (-0.10,0.04);

\node[blue!70!black,align=left] at (-4.18,2.62)
  {$C_7$ solution };

\draw[reference,green!50!black,->]
  (xC)
  .. controls (1.15,0) and (2.55,-0.30) ..
  (4.25,-0.72);

\node[green!50!black,align=left] at (5.12,-1.15)
  {sine cone solution\\[-1mm]
   tangent to $E_1$};

\draw[selected,red!75!black,->]
  (-4.15,1.72)
  .. controls (-2.65,1.65) and (-0.82,0.28) ..
  (0.12,-0.10)
  .. controls (1.30,-0.20) and (2.75,-0.48) ..
  (4.28,-0.82);

\node[red!75!black,align=left] at (-4.8,1.8)
  {$\widehat\varphi_{\epsilon,\beta(\epsilon)}$};

\node[red!75!black] at (1.35,-0.66)
  {$\mathcal A(\epsilon,\beta(\epsilon))=0$};

\draw[untuned,orange!80!black,->]
  (-4.10,1.08)
  .. controls (-2.50,1.05) and (-0.72,0.18) ..
  (0.12,0.05)
  .. controls (0.82,0.45) and (1.65,1.18) ..
  (2.58,1.95);

\node[orange!80!black,align=left] at (-4.98,0.9)
  { $\widehat\varphi_{\epsilon,\beta}$, $\mathcal{A}(\epsilon,\beta)\neq 0$};

\fill (xC) circle (2.2pt);
\node[below left=1pt] at (xC) {$x_C$};

\end{tikzpicture}}

\caption{The $C_7$ trajectory approaches $x_C$ through the local stable manifold. In general, a nearby solution $\widehat{\g2}_{\epsilon,\beta}$ exits with a nonzero
$E_{\gamma_+}$-component. Choosing $\beta=\beta(\epsilon)$ so that
$\mathcal A(\epsilon,\beta(\epsilon))=0$ removes this component, and
the selected solution leaves a neighbourhood of $x_C$ tangent to $E_1$ which is the sine cone trajectory.}
\label{fig:local-passage-xC}
\end{figure}
\begin{thm}[Local desingularisation of the sine cone]
\label{thm:local-desingularisation}
There exist $\epsilon_0>0$ and a $C^1$ function
\[
\beta:[0,\epsilon_0)\longrightarrow(0,\infty),
\qquad
\beta(0)=\beta_{\mathrm{ac}},
\]
with the following property.  For every $T<T_{\mathrm{sc}}$, there
exists $\epsilon_T\in(0,\epsilon_0)$ such that, for every
$0<\epsilon<\epsilon_T$, $\varphi_{\epsilon,\beta(\epsilon)}$  is
defined on $[0,T]$.  Moreover, if $\varphi_{\mathrm{sc}}$ denotes the
sine-cone structure, then for every $0<\delta<T<T_{\mathrm{sc}}$ we have
\begin{equation}\label{eq:outer-convergence}
\varphi_{\epsilon,\beta(\epsilon)}
\longrightarrow\varphi_{\mathrm{sc}}
\qquad\text{in }C^\infty([\delta,T])
\end{equation}
as $\epsilon\to0$.  After the inner rescaling, the same family
satisfies
\begin{equation}\label{eq:c7-inner-convergence}
 \widehat\varphi_{\epsilon,\beta(\epsilon)}
 \longrightarrow\varphi_{\beta_{\mathrm{ac}}}
 \qquad\text{in }C^\infty_{\mathrm{loc}}(X_{C_7}).
\end{equation}
\end{thm}

\begin{proof} Write
\[
\widehat Y_{\epsilon,\beta}(s)
=
\bigl(x_{\epsilon,\beta}(s),
z_{\epsilon,\beta}(s),
w_{\epsilon,\beta}(s)\bigr)
\]
in the coordinates of Lemma~\ref{lem:local-passage}, so that
\[
\Sigma_{\mathrm{in}}=\{x=x_{\mathrm{in}}\}.
\]
Define
\[
F(\epsilon,\beta,s)
:=
x_{\epsilon,\beta}(s)-x_{\mathrm{in}}.
\]
The equation $F(\epsilon,\beta,s)=0$ is precisely the condition that
the trajectory $\widehat Y_{\epsilon,\beta}$ meets
$\Sigma_{\mathrm{in}}$. The $C_7$ trajectory meets this section transversely at time $s_{\mathrm{in}}^0$ , and
hence
\[
\partial_sF(0,\beta_{\mathrm{ac}},s_{\mathrm{in}}^0)
=
\dot x_{0,\beta_{\mathrm{ac}}}(s_{\mathrm{in}}^0)\neq0.
\]
Although geometrically \(\epsilon\ge0\), the rescaled equations admit a smooth formal extension to \(\epsilon<0\). This allows us to apply the implicit-function theorem at \(\epsilon=0\), and then restrict the resulting family to \(\epsilon\ge0\). Since the solutions depend $C^1$
on $(\epsilon,\beta)$, the implicit-function theorem gives a unique
$C^1$ intersection time $s_{\mathrm{in}}(\epsilon,\beta)$ and,
consequently, a $C^1$ map
\[
\mathcal I(\epsilon,\beta)
=
\widehat Y_{\epsilon,\beta}
\bigl(s_{\mathrm{in}}(\epsilon,\beta)\bigr).
\]
After shrinking the parameter neighbourhood if necessary, this is the
unique intersection of the rescaled trajectory with
$\Sigma_{\mathrm{in}}$ in the local passage neighbourhood, and $\mathcal I(0,\beta_{\mathrm{ac}})=P_{\mathrm{in}}$.

Define the scalar function
\begin{equation}\label{eq:matching-function}
 \mathcal A(\epsilon,\beta)
 :=\Theta\bigl(\mathcal I(\epsilon,\beta)\bigr).
\end{equation}
Then $\mathcal A$ is $C^1$ and $\mathcal A(0,\beta_{\mathrm{ac}})=0$. Let
\[
V(s)
=
\left.
\partial_\beta\widehat Y_{0,\beta}(s)
\right|_{\beta=\beta_{\mathrm{ac}}}.
\]
Differentiating the identity $\mathcal I(0,\beta)
=
\widehat Y_{0,\beta}\bigl(s_{\mathrm{in}}(0,\beta)\bigr)$ at $\beta=\beta_{\mathrm{ac}}$ gives
\[
\partial_\beta\mathcal I(0,\beta_{\mathrm{ac}})
=
V(s_{\mathrm{in}}^0)
+
\mathcal F_0(P_{\mathrm{in}})
\,\partial_\beta s_{\mathrm{in}}(0,\beta_{\mathrm{ac}}).
\]
Thus $\partial_\beta\mathcal I(0,\beta_{\mathrm{ac}})$ is obtained
from $V(s_{\mathrm{in}}^0)$ by projecting along the flow direction
onto $T_{P_{\mathrm{in}}}\Sigma_{\mathrm{in}}$.  The flow direction is
tangent to the local stable manifold and therefore has zero
$E_{\gamma_+}$-coefficient.  Hence part~(2) of Lemma \ref{lem:local-passage} gives a constant $c_{\mathrm{in}}\neq0$ such that
\[
\partial_\beta\mathcal A(0,\beta_{\mathrm{ac}})
=
c_{\mathrm{in}}\mathfrak a.
\]
By Proposition~\ref{prop:exchange-nonzero}, $\mathfrak a\neq0$,
and consequently $\partial_\beta\mathcal A(0,\beta_{\mathrm{ac}})\neq0$.
The implicit-function theorem now gives $\epsilon_0>0$ and a unique
$C^1$ function
\[
 \beta:[0,\epsilon_0)\longrightarrow(0,\infty),
 \qquad
 \beta(0)=\beta_{\mathrm{ac}},
\]
such that
\[
 \mathcal A(\epsilon,\beta(\epsilon))=0.
\]
In particular, $\beta(\epsilon)\to\beta_{\mathrm{ac}}$ as
$\epsilon\to0$.  For $\epsilon>0$, set
\[
 P_\epsilon=\mathcal I(\epsilon,\beta(\epsilon)).
\]
Then $P_\epsilon\to P_{\mathrm{in}}$, $z(P_\epsilon)>0$, and
$\Theta(P_\epsilon)=0$.  

By part~(3) of Lemma \ref{lem:local-passage} the corresponding intersection with the outgoing section converges, as $\epsilon\to0$, to the outgoing
intersection of the sine-cone trajectory.  Since
\[
 z=\widehat\lambda e^s=\lambda t,
\]
the section $z=z_{\mathrm{out}}$ lies at the fixed regular time
$t_{\mathrm{out}}=z_{\mathrm{out}}/\lambda$ independently of $\epsilon$. Since the original equations are regular
for $t>0$, smooth dependence on initial data then gives \eqref{eq:outer-convergence}. 

Finally, the parameter dependence in Proposition \ref{prop:inner-limit} and
$\beta(\epsilon)\to\beta_{\mathrm{ac}}$ imply that, for every $R>0$,
\[
 \widehat\varphi_{\epsilon,\beta(\epsilon)}
 \longrightarrow\varphi_{\beta_{\mathrm{ac}}}
 \qquad\text{in }C^\infty([0,R]),
\]
which proves \eqref{eq:c7-inner-convergence} and the theorem. 
\end{proof}

\begin{rem}[The double-cover formulation]
Let $P:M_{1,1}\longrightarrow
X_{C_7}=M_{1,1}/\langle\sigma\rangle$ be the free double cover constructed in Proposition \ref{prop:c7-finite-quotient}.  Every nearly parallel $\mathrm G_2$-structure determined by a solution
$\xi_{a,b,c}$ pulls back along $P$, since
$$
d(P^*\varphi)
=
P^*(d\varphi)
=
\lambda P^*(*_\varphi\varphi)
=
\lambda *_{P^*\varphi}P^*\varphi.
$$
In the cohomogeneity-one variables this pullback leaves the coefficient functions, and hence the differential equations, unchanged.  The
desingularising family therefore lifts to the group diagram for the $C_7$ family. On this double cover, the outer limit is the sine cone over $(S^3\times S^3)/\mathbb Z_4$, while the rescaled inner limit is the AC $C_7$ metric.
Thus the result may equivalently be regarded as a local desingularisation of the sine cone over
$(S^3\times S^3)/\mathbb Z_4$ by the AC $C_7$ metric.
\end{rem}

Theorem \ref{thm:local-desingularisation} desingularises only one vertex of the sine cone. It neither
asserts that $\varphi_\epsilon$ closes at a second singular orbit nor that a reflection symmetry can be used. That would constitute an independent global shooting problem.  

\phantomsection
\addcontentsline{toc}{section}{References}
\bibliographystyle{amsalpha}
\bibliography{cohomogeneity_one_nearly_G2}
   
\end{document}